%% file: main.tex
\documentclass[twoside,11pt]{article}
\usepackage[letterpaper, margin=1.0in]{geometry}
\usepackage{bm}
\usepackage{comment}
\usepackage{xcolor}

\makeatletter

\usepackage[numbers]{natbib}
\usepackage{appendix}

\newcommand\@shorttitle{}
\newcommand\shorttitle[1]{\renewcommand\@shorttitle{#1}}

\usepackage{enumerate}
\usepackage[ruled,linesnumbered]{algorithm2e}
\usepackage{multirow}
\usepackage{subfig}
\usepackage{tabularx}
\usepackage{makecell}
\usepackage{booktabs}

\usepackage{setspace}

\usepackage{amsmath}
\usepackage{amstext}
\usepackage{amsthm}
\usepackage{amssymb}
\usepackage{bbm}
\usepackage{amsbsy}
\usepackage{mathtools}

    \newtheorem{prop}{Proposition}[section]
    \numberwithin{prop}{section} 
    \newtheorem{lem}{Lemma}[section]
    \numberwithin{lem}{section} 
    \newtheorem{rem}{Remark}[section]
    \numberwithin{rem}{section} 
    
    \numberwithin{cor}{section} 
    
    \numberwithin{thm}{section} 
    \newtheorem{defn}{Definition}
    \numberwithin{defn}{section}
    \newtheorem{fact}{Fact}
    \numberwithin{fact}{section}

    \theoremstyle{plain}
    \newtheorem*{thm*}{Theorem}
    
    \newcommand{\beq}{\begin{equation}}
    \newcommand{\eeq}{\end{equation}}
    \newcommand{\beqnn}{\begin{equation*}}
    \newcommand{\eeqnn}{\end{equation*}}

    \def\bb {\boldsymbol{b}}

    \def\bu {\boldsymbol{u}}
    \def\bs {\boldsymbol{s}}
    \def\bv {\boldsymbol{v}}
    \def\bw {\boldsymbol{w}}
    \def\bx {\boldsymbol{x}}
    \def\by {\boldsymbol{y}}
    \def\bz {\boldsymbol{z}}
    
    \def\phix {\phi_{\xcal}}
    \def\phiys {\phi_{\ycal^*}}
    
    \def\bxb {\bar{\boldsymbol{x}}}
    \def\bxh {\hat{\boldsymbol{x}}}
    \def\bxt {\tilde{\boldsymbol{x}}}
    
    \def\bxs {\boldsymbol{x^*}}
    \def\bys {\boldsymbol{y^*}}
    \def\bysb {\boldsymbol{\bar{y}}^{*}}
    \def\bysh {\boldsymbol{\hat{y}}^{*}}
    \def\byst {\boldsymbol{\tilde{y}}^{*}}

    \def\bX {\boldsymbol{X}}
    \def\bY {\boldsymbol{Y}}
    \def\bYs {\boldsymbol{Y^*}}

    \def\R {\mathbb{R}}

    \def\bxi {\boldsymbol{\xi}}

    \def\vecmax {\mathrm{vecmax}}
    
    \def\xcal {\mathcal{X}}
    \def\xcals {\mathcal{X}^{*}}
    \def\ycal {\mathcal{Y}}
    \def\ycals {\mathcal{Y}^*}
    
    \newcommand{\norm}[1]{\left\Vert{#1} \right\Vert}
    
    \newcommand{\normx}[1]{\left\Vert{#1} \right\Vert_{\mathcal{X}}}
    \newcommand{\normy}[1]{\left\Vert{#1} \right\Vert_{\mathcal{Y}}}
    \newcommand{\normxs}[1]{\left\Vert{#1} \right\Vert_{\mathcal{X}^*}}
    \newcommand{\normys}[1]{\left\Vert{#1} \right\Vert_{\mathcal{Y}^*}}
    \newcommand{\normop}[1]{\left\Vert{#1} \right\Vert_{\mathrm{op}}}

    \newcommand{\matr}[1]{\bm{#1}}

    \def\dom {\mathrm{dom~}}    
    \def\interior {\mathrm{int}}  

    \DeclareMathAlphabet{\mymathbb}{U}{BOONDOX-ds}{m}{n}
    
    \def\Rn {\R^n}
    \def\Rm {\R^m}
    \def\Rd {\R^d}
    \def\gmxcal {\Gamma_0(\xcal)}
    \def\gmycal {\Gamma_0(\ycal)}
    \def\gmycals {\Gamma_0(\ycal^*)}
    \def\gmxcals {\Gamma_0(\xcal^*)}

    \DeclareMathOperator*{\argmin}{arg\,min}
    \DeclareMathOperator*{\argmax}{arg\,max}

    \newcommand{\normsq}[1]{\left\Vert{#1} \right\Vert_{2}^{2}}
    \newcommand{\normtwo}[1]{\left\Vert{#1} \right\Vert_{2}}
    \newcommand{\normone}[1]{\left\Vert{#1} \right\Vert_{1}}
    \newcommand{\norminf}[1]{\left\Vert{#1} \right\Vert_{\infty}}

\usepackage{jmlr2e}

\usepackage{lastpage} 
\jmlrheading{volume}{2022}{1-\pageref{LastPage}}{[4/22]; Revised [DATE REVISED]}{[DATE PUBLISHED]}{PAPER ID}{J\'er\^ome Darbon and Gabriel P. Langlois} 

\ShortHeadings{Accelerated nonlinear methods for machine learning}{Darbon and Langlois}\firstpageno{1}

\begin{document}

\title{Accelerated nonlinear primal-dual hybrid gradient methods with applications to supervised machine learning}

\author{\name J\'er\^ome Darbon \email jerome\_darbon@brown.edu \\
        \addr Division of Applied Mathematics\\
              Brown University\\
              Providence, RI 02912, USA
              \AND
              \name Gabriel P.\ Langlois\thanks{Corresponding author.} \email gabriel\_provencher\_langlois@brown.edu \\
              \addr Division of Applied Mathematics\\
              Brown University\\
              Providence, RI 02912, USA}

\editor{ }

\maketitle

\begin{abstract}
The linear primal-dual hybrid gradient (PDHG) method is a first-order method that splits convex optimization problems with saddle-point structure into smaller subproblems. Unlike those obtained in most splitting methods, these subproblems can generally be solved efficiently because they involve simple operations such as matrix-vector multiplications or proximal mappings that are fast to evaluate numerically. This advantage comes at the price that the linear PDHG method requires precise stepsize parameters for the problem at hand to achieve an optimal convergence rate. Unfortunately, these stepsize parameters are often prohibitively expensive to compute for large-scale optimization problems, such as those in machine learning. This issue makes the otherwise simple linear PDHG method unsuitable for such problems, and it is also shared by most first-order optimization methods as well. To address this issue, we introduce accelerated nonlinear PDHG methods that achieve an optimal convergence rate with stepsize parameters that are simple and efficient to compute. We prove rigorous convergence results, including results for strongly convex or smooth problems posed on infinite-dimensional reflexive Banach spaces. We illustrate the efficiency of our methods on $\ell_{1}$-constrained logistic regression and entropy-regularized matrix games. Our numerical experiments show that the nonlinear PDHG methods are considerably faster than competing methods.
\end{abstract}
\begin{keywords}
 Convex optimization, primal-dual hybrid gradient splitting methods, Bregman divergences, logistic regression, matrix games.
\end{keywords}


\input{1_introduction.tex}
\input{2_preliminaries.tex}
\input{3_basic_method.tex}
\input{4_accelerated_schemes.tex}
\input{5_applications.tex}
\input{6_Numerical_experiments.tex}
\input{7_discussion.tex}

\section*{Acknowledgement}
Gabriel P. Langlois would like to thank Tingwei Meng for useful discussions and for spotting typos in an earlier version of this manuscript.

\appendix
\section*{Appendices}

\section{Definitions and facts from convex and functional analysis}\label{app:def}
\input{app_definitions.tex}

\section{Proof of Lemma~\ref{lem:descent_rule}} \label{app:A}
\input{app_proof_basic.tex}

\section{Proof of Proposition~\ref{prop:basic_algI}} \label{app:B}
\input{app_proof_propI}

\section{Proof of Proposition~\ref{prop:acc_sq_algI}} \label{app:C}
\input{app_proof_propacc1}

\section{Proof of Proposition~\ref{prop:acc_sq_algIII}}\label{app:D}
\input{app_proof_propacc2}

\bibliography{proj-bib}

\end{document}

%% file: 1_introduction.tex






\section{Introduction} \label{sec:intro}
\subsection*{Overview}
The linear primal-dual hybrid gradient (PDHG) method is a first-order splitting method for minimizing the sum of two convex functions~\cite{chambolle2011first,chambolle2016ergodic,esser2010general,pock2011diagonal,pock2009algorithm,zhu2008efficient}. It works by splitting the sum into smaller subproblems, each of which is easier to solve. These subproblems, unlike those obtained from most splitting methods, can generally be solved efficiently because they involve simple operations such as matrix-vector multiplications or proximal mappings that are fast to evaluate numerically. This makes the linear PDHG method flexible and easy to implement for solving a wide range of constrained and nondifferentiable optimization problems. Due to this advantage, the linear PDHG method is widely used for solving problems in imaging science~\cite{benning2016explorations,bredies2015tgv,estellers2015adaptive,gilboa2016nonlinear,knoll2016joint,kongskov2019directional,rigie2015joint}, optimal control~\cite{faessler2016autonomous,kirchner2018primal}, compressive sensing~\cite{foucart2013invitation,hou2019metal}, distributed optimization~\cite{pesquet2014class,scaman2018optimal,scaman2019optimal}, and optimal transport~\cite{carrillo2021primal,dvurechensky2018computational,ferradans2014regularized,gangbo2019unnormalized,liu2021multilevel,papadakis2014optimal}. It is also used, to a limited extent, for solving large-scale problems in machine learning~\cite{arridge2019solving,barlaud2021classification,cevher2014convex,hien2021algorithms,polson2015proximal,schaeffer2017learning,yanez2017primal}.

Despite its flexibility and ease of implementation, the linear PDHG method requires precise stepsize parameters for the problem at hand to to achieve an optimal convergence rate. Unfortunately, these stepsize parameters are often prohibitively expensive to compute for large-scale optimization problems. This issue makes the otherwise simple linear PDHG method unsuitable for solving large-scale optimization problems, such as those in machine learning. This issue is shared by most first-order optimization methods as well.

To illustrate this point, consider the $\ell_{1}$-constrained logistic regression problem
\begin{equation}\label{eq:intro_l1-logreg}
\inf_{\substack{\bv \in \Rd \\ \normone{\bv} \leqslant \lambda}} \frac{1}{m}\sum_{i=1}^{m} \log\left(1 + e^{-[\bb]_{i}\left\langle [\bu]_{i},\bv\right\rangle}\right),
\end{equation}
where $\{[\bu]_{i},[\bb]_{i}\}_{i=1}^{m}$ denote a collection of $m$ feature vectors $[\bu]_{i} \in \Rd$ with labels $[\bb]_{i} \in \{-1,+1\}$ and $\lambda > 0$ is a parameter. This problem can be solved using the linear PDHG method as follows. Let $\matr{B}$ denote the $m \times d$ matrix whose rows are the elements $-[\bb]_{i}[\bu]_{i}$, let $\matr{B}^{*}$ denote its matrix transpose, and let $\norm{\matr{B}}_{2,2}$ denote the largest singular value of $\matr{B}$. Formally, the linear PDHG method computes a global minimum of problem~\eqref{eq:intro_l1-logreg} via the iterations~\cite[Algorithm 5]{chambolle2016ergodic}\cite{moreau1965proximite}
\begin{equation}\label{eq:intro_l1-logreg_alg}
    \begin{alignedat}{1}
    \bz_{k} &= \bw_{k} + \sigma_{k}\matr{B}(\bv_{k} + \theta_{k}[\bv_{k}-\bv_{k-1}]), \\
    \bw_{k+1} &= \bz_{k} - \argmin_{\bw \in \Rm} \left\{\frac{1}{2}\normsq{\bw - \bz_{k}} + \frac{\sigma_{k}}{m}\sum_{i=1}^{m}\log\left(1 + e^{[\bw]_{i}/\sigma_{k}}\right)\right\}, \\
    \bv_{k+1} &= \argmin_{\substack{\bv \in \Rn \\ \normone{\bv} \leqslant \lambda}} \frac{1}{2}\normsq{\bv - \left(\bv_{k} - \tau_{k}\matr{B}\bw_{k+1}\right)}, \\
    \theta_{k+1} &= 1/\sqrt{1+4m\sigma_{k}}, \quad \tau_{k+1} = \tau_{k}/\theta_{k+1} \quad \text{and} \quad \sigma_{k+1} = \theta_{k+1}\sigma_{k},
    \end{alignedat}
\end{equation}
where $\bv_{-1} = \bv_{0}$ are vectors in the interior of the $d$-dimensional $\ell_{1}$-ball of radius $\lambda$, $\bw_{0}$ is a vector in $\Rm$, and $\tau_{0} > 0$, $\sigma_{0} = 1/(\norm{\matr{B}}_{2,2}^{2}\tau_{0})$ and $\theta_{0} = 0$ are the initial stepsize parameters. The updates for $\bw_{k+1}$ and $\bv_{k+1}$ in~\eqref{eq:intro_l1-logreg_alg} can be evaluated efficiently using standard first or second-order optimization methods and efficient $\ell_{1}$-ball projection algorithms~\cite{condat2016fast}, respectively. The other operations in the updates can all be computed exactly in at most $O(md)$ operations. The convergence rate for this method is $O(1/k^{2})$ in the number of iterations $k$, which is the best possible achievable rate of convergence for this problem in the Nesterov class of optimal first-order methods~\cite{Nesterov2018}.

Attaining this optimal rate of convergence requires a precise estimate of the largest singular value $\norm{\matr{B}}_{2,2}$ of the matrix $\matr{B}$. However, this quantity takes on the order of $O(\min(m^2d,md^2))$ operations to compute~\cite{hastie2009elements}. This computational cost makes it essentially impossible to estimate the largest singular value for large matrices. Line search methods and other heuristics are often used to bypass this issue, but they typically slow down the convergence. Most first-order optimization methods used for solving large-scale optimization problems share this issue as well.

To address this issue, we present novel accelerated nonlinear PDHG methods that can achieve an optimal rate of convergence with stepsize parameters that are simple and efficient to compute. Returning to the previous example, let $\norm{\matr{B}}_{1,2}$ denote the maximum $\ell_{2}$ norm of a column of the matrix $\matr{B}$ and define new parameters $\hat{\tau}_{0} > 0$, $\hat{\sigma}_{0} = 1/(\norm{\matr{B}}_{1,2}^{2}\hat{\tau}_{0})$ and $\hat{\theta} = 0$. In addition, let $\bx_{-1} = \bx_{0}$ denote vectors contained in the interior of the $2d$-dimensional unit simplex $\Delta_{2d}$, let $\by_{0}^{*}$ denote a vector in the $m$-dimensional cube $(0,1/m)^{m}$, let $[\hat{\bw}_{0}^{*}]_{i} = \log\left(m[\by_{0}^{*}]_{i}/(1-m[\by^{*}_{0}]_{i})\right)$ for $i \in \{1,\dots,m\}$, and let $(\matr{A} | -\matr{A})$ denote the horizontal concatenation of the matrices $\matr{A}$ and $-\matr{A}$. Then, we show in Sections~\ref{subsec:accII_var} and~\ref{subsec:l1-logreg} that the accelerated nonlinear PDHG method
\begin{equation}\label{eq:intro_l1-logreg_alg_nl}
    \begin{alignedat}{1}
    \hat{\bw}_{k+1} &= \left(4m\hat{\sigma}_{k}\bx_{k} + 4m\hat{\sigma}_{k}\hat{\theta}\left(\bx_{k} - \bx_{k-1}\right) + \hat{\bw}_{k}\right)/(1+4m\hat{\sigma}_{k}),\\
    [\by_{k+1}^{*}]_{i} &= \frac{1}{m + me^{-[\bw_{k+1}]_{i}}} \quad \text{for}\;i\in\left\{1,\dots,m\right\},\\
    [\bx_{k+1}]_{j} &= \frac{[\bx_{k}]_{j}e^{-\hat{\tau}_{k}[\matr{A}^{*}\by_{k+1}^{*}]_{j}}}{\sum_{j=1}^{m}[\bx_{k}]_{j}e^{-\hat{\tau}_{k}[\matr{A}^{*}\by_{k+1}^{*}]_{j}}} \quad \text{for}\;j\in\left\{1,\dots,n\right\},\\
    \bv_{k+1} &= \lambda(\matr{A} \mid -\matr{A})\bx_{k+1}\\
    \hat{\theta}_{k+1} &= 1/\sqrt{1 + 4m\hat{\sigma}_{k}}, \quad \hat{\tau}_{k+1} = \hat{\tau}_{k}/\hat{\theta}_{k+1}, \quad \text{and} \quad \hat{\sigma}_{k+1} = \hat{\theta}_{k+1}\hat{\sigma}_{k},
    \end{alignedat}
\end{equation}
computes a global minimum of problem~\eqref{eq:intro_l1-logreg} through the iterates $\bv_{k}$. Moreover, the convergence rate is $O(1/k^2)$ in the number of iterations $k$, which is the best possible achievable rate of convergence for this problem in the Nesterov class of optimal first-order methods~\cite{Nesterov2018}. 

Unlike in the linear PDHG method~\eqref{eq:intro_l1-logreg_alg}, the stepsize parameters in the nonlinear PDHG method~\eqref{eq:intro_l1-logreg_alg_nl} are computed in optimal $\Theta(md)$ operations from the matrix norm $\norm{\matr{A}}_{1,2}$. In addition, the computational bottleneck in the iterates consists of matrix-vector multiplications that can be computed in $O(md)$ operations or better with appropriate parallel algorithms. Thus all stepsize parameters and updates in the nonlinear method~\eqref{eq:intro_l1-logreg_alg_nl} are computed in quadratic $O(md)$ time, in contrast to the stepsize parameters in the linear method~\eqref{eq:intro_l1-logreg_alg} which are computed in cubic $O(\min(m^2d,md^2))$ time. Numerical experiments in Section~\eqref{sec:numerics} show that the nonlinear PDHG method~\eqref{eq:intro_l1-logreg_alg_nl} converges 5 to 10 times faster than the linear PDHG method~\eqref{eq:intro_l1-logreg_alg}.

\subsection*{Related work}
The linear PDHG method was introduced at around the same time by~\citet{pock2009algorithm} and~\citet{esser2010general} to solve problems in imaging science (see also earlier work from \cite{popov1980modification,zhu2008efficient}). The convergence of the linear PDHG method for problems posed on Euclidean spaces was later proven by~\citet{chambolle2011first}. In addition to a proof of convergence, their work provided accelerated schemes of the linear PDHG method for problems with some degree of smoothness or strong convexity or both.

Since then, many variants and extensions of the linear PDHG method have been proposed; see~\cite{chambolle2016introduction,chambolle2016ergodic,chambolle2018stochastic} for further details and references. A partial list of these variants include: overrelaxed~\cite{condat2013primal,he2012convergence}, inertial~\cite{lorenz2015inertial}, operator, forward-backward, and proximal-gradient splitting~\cite{boct2015convergence,combettes2014forward,davis2017three,drori2015simple,vu2013splitting}, multistep~\cite{chen2014optimal}, stochastic~\cite{balamurugan2016stochastic,chambolle2018stochastic,fercoq2019coordinate,pesquet2014class,valkonen2016block,wen2016randomized,yanez2017primal}, and nonlinear~\cite{chambolle2016ergodic,hohage2014generalization} variants, including the mirror descent method~\cite{nemirovski2004prox}. Here, we focus on nonlinear PDHG methods.

The extension of the linear PDHG method to the nonlinear setting was first done, to our knowledge, by~\citet{hohage2014generalization} to solve non-smooth convex optimization problems posed on Banach spaces. A nonlinear PDHG method for solving such problems using nonlinear proximity operators based on Bregman divergences was later proposed by \citet{chambolle2016ergodic}. Their work also provided an accelerated and partially nonlinear scheme for solving strongly convex problems. Their scheme is not fully nonlinear, however, as it requires one of the Bregman divergence to be a quadratic function. Moreover, their work did not provide accelerated nonlinear schemes for smooth convex problems or smooth and strongly convex problems.

\subsection*{Our contributions} 
This paper contributes accelerated nonlinear PDHG methods that achieve an optimal rate of convergence in the Nesterov class of optimal first-order methods with stepsize parameters that are simple and efficient to compute. To do so, we extend the theory of accelerated nonlinear PDHG methods initiated in~\cite{chambolle2016ergodic} to solve optimization problems on Banach spaces with nonlinear proximity operators based on Bregman divergences. The main theoretical results and accelerated nonlinear PDHG methods are described in Section~\ref{sec:acc_pd_algs}. We prove rigorous convergence results, including results strongly convex or smooth problems posed on infinite-dimensional reflexive Banach spaces. In addition, we provide in Section~\ref{sec:applications} practical implementations of accelerated nonlinear PDHG methods for $\ell_{1}$-constrained logistic regression and zero-sum matrix games with entropy regularization, and we perform numerical experiments on these problems in Section~\ref{sec:numerics} to compare the running times of nonlinear PDHG methods to other commonly-used first-order optimization methods. Our numerical experiments show that the nonlinear PDHG methods are considerably faster than competing methods.

The results we present are generally applicable to convex-concave saddle-point optimization problems posed on real reflexive Banach spaces. Before proceeding toward the technical setup considered in the next section, we describe some key symbols and the notation used in the remainder of this paper in Table 1. For a list of concepts and facts from real, convex and functional analysis that are used in this paper, see Appendix~\ref{app:def}.

\begin{table}[ht]
\caption{Notation}
    \centering
    \begin{tabular}{l|l}
    \hline
        Notation & Meaning \\
        \hline
        $\xcal$  & Real reflexive Banach space endowed with norm $\normx{\cdot}$\\
        $\xcals$ & Dual space of all continuous linear functionals defined on $\xcal$  \\ 
        $\bx \mapsto \left\langle \bx^{*},\bx \right\rangle$ & Value of the functional $\bx^{*}$ at $\bx$\\ 
        $\normxs{\cdot}$ & Norm over the dual space $\xcals$: $\normxs{\bx^{*}} = \sup_{\normx{\bx} = 1} \left\langle \bx^{*},\bx \right\rangle$   \\ 
        $\matr{A} \colon \xcal \to \ycal$ & Bounded linear operator between two reflexive Banach spaces $\xcal$ and $\ycal$ \\ 
        $\matr{A}^{*} \colon \ycals \to \xcals$ & Adjoint operator of $\matr{A}$  \\ 
        $\normop{\matr{A}}$ & Operator norm of $\matr{A}$: $\normop{\matr{A}} = \sup_{\normx{\bx} = 1} \normy{\matr{A}\bx} = \sup_{\normys{\by^*} = 1}\normxs{\matr{A}^*\by^*}$ \\ 
        $\norm{\matr{A}}_{2,2}$ & Largest singular value of an $m \times n$ real matrix $\matr{A}$ \\
        $\norm{\matr{A}}_{1,2}$ & Maximum $\ell_{2}$ norm of a column of an $m \times n$ real matrix $\matr{A}$\\
        $\norm{\matr{A}}_{1,\infty}$ & Maximum $\ell_{\infty}$ norm of a column of an $m \times n$ real matrix $\matr{A}$\\
        $(\matr{A} \mid \matr{B})$ & Horizontal concatenation of two $m \times n$ matrices $\matr{A}$ and $\matr{B}$\\
        $\interior{~C}$ & Interior of a non-empty subset $C$  \\
        $\dom g$ & Domain of a function $g$ \\
        $\gmxcal$ & Set of proper, convex and lower semicontinuous functions defined on $\xcal$ \\ 
        $\partial g(\bx)$ & Subdifferential of a function $g \in \gmxcal$ at $\bx \in \xcal$ \\
        $g^{*}$& Convex conjugate of a function $g$ \\
        $\matr{I}_{n \times n}$ & $n \times n$ identity matrix\\
        $\Delta_{n}$& Unit simplex over $\Rn$: $\Delta_{n} = \{\bx \in \Rn : \sum_{j=1}^{n} [\bx]_{j} = 1\}$ \\
        $\mathcal{H}_{n}(\bx)$ & Negative entropy of $\bx \in \Delta_{n}$: $\mathcal{H}_{n}(\bx) = \sum_{j=1}^{n}[\bx]_{j}\log\left([\bx]_{j}\right)$\\
        $\vecmax{(\bx)}$ & Maximum component of the vector $\bx \in \Rn$: $\vecmax{(\bx)} = \max{([\bx]_{1},\dots,[\bx]_{n})}$\\\hline
    \end{tabular}    \label{tab:notation}
\end{table}

%% file: 2_preliminaries.tex
\section{Setup}\label{sec:setup}
We are interested here with convex-concave saddle-point problems posed on real reflexive Banach spaces. Concretely, let $\xcal$ and $\ycal$ denote two real reflexive Banach spaces endowed with norms $\normx{\cdot}$ and $\normy{\cdot}$, and let $\matr{A}\colon\xcal\to\ycal$ denote a bounded linear operator between those two spaces. We consider the following convex-concave saddle-point problem
\begin{equation}\label{eq:saddle_prob}
  \inf_{\bx \in \xcal} \sup_{\by \in \ycal^*} \left\{g(\bx) + \left\langle \by^*,\matr{A} \bx \right\rangle - h^*(\by^{*}) \right\}
\end{equation}
where $g \in \gmxcal$ and $h \in \gmycal$. Formally, this is the primal-dual formulation associated to the primal problem
\begin{equation}\label{eq:primal_prob}
    \inf_{\bx \in \xcal} \{g(\bx) + h(\matr{A}\bx)\}
\end{equation}
and the dual problem
\begin{equation}\label{eq:dual_prob}
    \sup_{\bys \in \ycals} \{-g^{*}(-\matr{A}^* \bys) - h^*(\bys)\}.
\end{equation}
The objective function $\mathcal{L}\colon \xcal \times \ycal^{*} \to \R\cup\{+\infty\}$ in the saddle-point problem~\eqref{eq:saddle_prob}, namely
\begin{equation}\label{eq:lagrangian}
\mathcal{L}(\bx,\by^*) = g(\bx) + \left\langle \by^*,\matr{A} \bx \right\rangle - h^*(\by^{*}),
\end{equation}
is called the Lagrangian of the primal and dual problems~\eqref{eq:primal_prob} and~\eqref{eq:dual_prob}. Solutions to the saddle-point problem~\eqref{eq:saddle_prob}, when they exist, are saddle points of the Lagrangian~\eqref{eq:lagrangian} (see Definition~\eqref{def:saddle} and Fact~\eqref{fact:primal_dual_prob}).

This work focuses on accelerated nonlinear PDHG methods designed to compute saddle points of~\eqref{eq:saddle_prob}, and therefore solutions to the primal and dual problems~\eqref{eq:primal_prob} and~\eqref{eq:dual_prob}. We describe below the formalism behind the nonlinear PDHG method. Let $\phix \in \gmxcal$ and $\phiys \in \gmycals$ denote two essentially smooth and essentially strictly convex functions, and consider their corresponding Bregman divergences:
\[
\begin{alignedat}{1}
&D_{\phix}(\bx,\bxb) = \phix(\bx) - \phi(\bxb) - \left\langle \nabla \phix(\bxb),\bx-\bxb \right\rangle \\
&D_{\phiys}(\bys,\bysb) = \phiys(\bys) - \phiys(\bysb) - \left\langle \bys-\bysb, \nabla \phiys(\bysb) \right\rangle.
\end{alignedat}
\]
Formally, we propose using these Bregman divergence to alternate in~\eqref{eq:saddle_prob} a nonlinear proximal descent step in the variable $\bx$ and a nonlinear proximal ascent step in the variable $\by^{*}$ as follows:
\begin{equation}\label{eq:iteration_PG}
\begin{dcases}
    \bxh &= \argmin_{\bx \in \xcal} \left\{g(\bx) + \left\langle \byst, \matr{A}\bx\right\rangle + \frac{1}{\tau}D_{\phix}(\bx,\bxb)\right\} \\
    \bysh &= \argmax_{\bys \in \ycal^*} \left\{-h^{*}(\bys) + \left\langle \bys, \matr{A}\bxt\right\rangle - \frac{1}{\sigma}D_{\phiys}(\bys,\bysb)\right\}.
\end{dcases}
\end{equation}
The iteration scheme~\eqref{eq:iteration_PG} takes the stepsize parameters $\tau, \sigma > 0$, initial points $(\bxb,\bysb) \in \xcal \times \ycal^*$, and intermediate points $(\bxt,\byst) \in \xcal \times \ycal^*$ to output the new points $(\bxh,\bysh)$. The nonlinear PDHG method consists of this iteration scheme with appropriate parameter values and initial and intermediate points to attain an optimal convergence rate.

\subsubsection*{Assumptions}
\begin{itemize}
    \item[(A1)] The two functions $g$ and $h$ are proper, lower semicontinuous, and convex over their respective domains $\xcal$ and $\ycal$. Moreover, the primal problem~\eqref{eq:primal_prob} has at least one solution and there exists a point $\bx \in \dom g$ such that $\matr{A}\bx \in \dom h$ and $h$ is continuous at $\matr{A}\bx$.
    
    \item[(A2)] The two functions $\phix$ and $\phiys$ are proper, lower semicontinuous, and convex over their respective domains $\xcal$ and $\ycals$. Moreover, $\phix$ and $\phiys$ are both essentially smooth and essentially strictly convex.

    \item[(A3)] The domains of the two functions $g$ and $\phix$ satisfy the inclusion $\dom\partial g \subseteq \interior{(\dom \phix)}$, and at least one of $g$ and $\phix$ is supercoercive.
    
    \item[(A4)] The domains of the two functions $h^*$ and $\phiys$ satisfy the inclusion $\dom\partial h^* \subseteq \interior{(\dom \phiys)}$, and at least one of $h^*$ and $\phiys$ is supercoercive.
    
    \item[(A5)] The two functions $\phix$ and $\phiys$ are 1-strongly convex with respect to $\normx{\cdot}$ and $\normys{\cdot}$ on their respective domains.
\end{itemize}
Assumptions (A1) ensures that the primal problem~\eqref{eq:primal_prob} and dual problem~\eqref{eq:dual_prob} each has at least one solution~\cite[Theorem 4.1]{ekeland1999convex}, and that the saddle-point problem~\eqref{eq:saddle_prob} has at least one saddle point~\cite[Proposition 3.1]{ekeland1999convex}. Assumptions (A1)-(A4) ensure that the Bregman divergences of $\phix$ and $\phiys$ and the minimization problems in the iteration~\eqref{eq:iteration_PG} satisfy the properties described by Facts~\ref{fact:breg_div_props} and~\ref{fact:breg_prox_prop} in Appendix~\ref{app:def}. Finally, assumption (A5) is used later in Section~\ref{sec:basic_alg} and~\ref{sec:acc_pd_algs} to prove the convergence of the nonlinear PDHG methods. We note that the domain inclusions in (A3) and (A4) are more restrictive than those assumed in~\cite{chambolle2016ergodic} and are necessary for the optimization methods to work (see Fact.~\ref{fact:breg_prox_prop}).

Under assumptions (A1)-(A4) and an appropriate choice of stepsize parameters, initial points, and intermediate points, the iteration scheme~\eqref{eq:iteration_PG} is well-defined and satisfies a descent rule:
\begin{lem}\label{lem:descent_rule}
Assume (A1)-(A4) hold, and assume the iteration scheme~\eqref{eq:iteration_PG} takes as input the stepsize parameters $\tau, \sigma > 0$, initial points $(\bxb, \bysb) \in \dom \partial g \times \dom \partial h^*$, and intermediate points $(\bxt, \byst) \in \xcal \times \ycals$. Then the iteration scheme~\eqref{eq:iteration_PG} generates a unique output $(\bxh,\bysh) \in \dom \partial g \times \dom \partial h^*$, and for every $(\bx,\bys) \in \dom g \times \dom h^*$ the output $(\bxh,\bysh)$ satisfies the descent rule
\begin{equation}\label{eq:descent_rule}
    \begin{alignedat}{1}
    \mathcal{L}(\bxh,\bys) - \mathcal{L}(\bx,\bysh) &\leqslant \frac{1}{\tau}\left(D_{\phix}(\bx,\bxb) - D_{\phix}(\bxh,\bxb) - D_{\phix}(\bx,\bxh)\right)\\
    &\quad + \frac{1}{\sigma}\left(D_{\phiys}(\bys,\bysb) - D_{\phiys}(\bysh,\bysb) - D_{\phiys}(\bys,\bysh) \right) \\
    &\quad + \left\langle \byst - \bysh, \matr{A}(\bx  - \bxt)\right\rangle - \left\langle \bys - \byst, \matr{A}(\bxt - \bxh) \right\rangle.
    \end{alignedat}
\end{equation}
\end{lem}
\begin{proof}
See Appendix~\ref{app:A}.
\end{proof}

%% file: 3_basic_method.tex
\section{The basic nonlinear primal-dual hybrid gradient method}\label{sec:basic_alg}

The basic nonlinear PDHG method takes two stepsize parameters $\tau, \sigma > 0$ and an initial pair of points $(\bx_{0},\by_{0}^{*}) \in \dom \partial g \times \dom \partial h^*$ to generate the iterates
\begin{equation}\label{eq:basic_pdhg_alg}
\begin{dcases}
    \bx_{k+1} &= \argmin_{\bx \in \xcal} \left\{g(\bx) + \left\langle \by_{k}^{*}, \matr{A}\bx\right\rangle + \frac{1}{\tau}D_{\phix}(\bx,\bx_{k})\right\} \\
    \by_{k+1}^{*} &= \argmax_{\bys \in \ycal^*} \left\{-h^{*}(\bys) + \left\langle \bys, \matr{A}(2\bx_{k+1} - \bx_{k})\right\rangle - \frac{1}{\sigma}D_{\phiys}(\bys,\by_{k}^{*})\right\}.
\end{dcases}
\end{equation}
Under assumptions (A1)-(A4), Lemma~\ref{lem:descent_rule} applies to method~\eqref{eq:basic_pdhg_alg}, and starting from the input $(\bx_{0},\by_{0}^{*})$ the method~\eqref{eq:basic_pdhg_alg} generates a unique output $(\bx_{1},\by_{1}^{*}) \in \dom\partial g \times \dom\partial h^*$. A simple induction argument using Lemma~\ref{lem:descent_rule} then shows that $(\bx_{k},\by_{k}^{*}) \in \dom \partial g \times \dom \partial h^*$ for every $k \in \mathbb{N}$. As such, method~\eqref{eq:basic_pdhg_alg} is well-defined. In addition, under assumption (A5) and appropriate conditions on the values of the stepsize parameters $\tau$ and $\sigma$, the nonlinear PDHG method~\eqref{eq:basic_pdhg_alg} satisfies the following properties:
\begin{prop}\label{prop:basic_algI}
Assume (A1)-(A5) hold and assume $\tau,\sigma > 0$ satisfy the strict inequality
\begin{equation}\label{eq:suff_param_ineq}
    \tau\sigma\normop{\matr{A}}^2 < 1.
\end{equation}
Let $(\bx_{0},\by_{0}^{*})$ be a pair of points contained in $\dom \partial g \times \dom \partial h^*$, let $(\bx_s,\by_{s}^{*})$ be a saddle point of the Lagrangian~\eqref{eq:lagrangian}, and let $K \in \mathbb{N}$. Consider the sequence of iterates $\{(\bx_{k},\by_{k}^{*})\}_{k=1}^{K}$ generated by the nonlinear PDHG method~\eqref{eq:basic_pdhg_alg} from the initial points $(\bx_{0},\by_{0}^{*})$, define the averages
\[
\bX_{K} = \frac{1}{K}\sum_{k=1}^{K} \bx_{k} \quad \mathrm{and} \quad \bY_{K}^{*} = \frac{1}{K}\sum_{k=1}^{K} \by_{k}^{*},
\]
and for $(\bx,\by^{*}) \in \dom g \times \dom h^*$, define the quantity
\begin{equation}\label{eq:basic_d_symbol}
\Delta_{k}(\bx,\bys) = \frac{1}{\tau}D_{\phix}(\bx,\bx_{k}) + \frac{1}{\sigma}D_{\phiys}(\bys,\by_{k}^{*}) - \left\langle \bys - \by_{k}^{*}, \matr{A}(\bx  - \bx_{k})\right\rangle.
\end{equation}
Then:
\begin{enumerate}
    \item[(a)] For every $(\bx,\bys) \in \dom g \times \dom h^*$ and nonnegative integer $k$, the output $(\bx_{k+1},\by_{k+1}^{*})$ of the nonlinear PDHG method~\eqref{eq:basic_pdhg_alg} satisfies the descent rule
    \begin{equation}\label{eq:descent_rule_s_I}
    \mathcal{L}(\bx_{k+1},\bys) - \mathcal{L}(\bx,\by_{k+1}^{*}) \leqslant \Delta_{k}(\bx,\bys) - \Delta_{k+1}(\bx,\bys).
    \end{equation}
    
    \item[(b)]  For every $(\bx,\bys) \in \dom g \times \dom h^*$ and $K \in \mathbb{N}$, we have the estimate
    \begin{equation}\label{eq:basic_rateI}
    \begin{alignedat}{1}
    \mathcal{L}(\bX_{K}, \bys) - \mathcal{L}(\bx,\bY_{K}^{*}) &\leqslant \frac{1+\sqrt{\tau\sigma}\normop{\matr{A}}}{K}\left(\frac{1}{\tau}D_{\phix}(\bx,\bx_{0}) + \frac{1}{\sigma}D_{\phiys}(\bys,\by_{0}^{*})\right) \\
    &\qquad - \frac{1-\sqrt{\tau\sigma}\normop{\matr{A}}}{K}\left(\frac{1}{\tau}D_{\phix}(\bx,\bx_{K}) + \frac{1}{\sigma}D_{\phiys}(\bys,\by_{K}^{*})\right),
    \end{alignedat}
    \end{equation}
    and for $(\bx,\by^{*}) = (\bx_{s},\by_{s}^{*})$, the global bound
    \begin{equation}\label{eq:bound_seq}
    \frac{1}{\tau}D_{\phix}(\bx_{s},\bx_{K}) + \frac{1}{\sigma}D_{\phiys}(\by_{s}^{*},\by_{K}^{*}) \leqslant \frac{1+\sqrt{\tau\sigma}\normop{\matr{A}}}{1-\sqrt{\tau\sigma}\normop{\matr{A}}}\left(\frac{1}{\tau}D_{\phix}(\bx_{s},\bx_{0}) + \frac{1}{\sigma}D_{\phiys}(\by_{s}^{*},\by_{0}^{*})\right).
    \end{equation}
    
    \item[(c)][Convergence properties]\label{prop:basic_algII} The sequences $\{(\bx_{k},\by_{k}^{*})\}_{k=1}^{+\infty}$ and $\{(\bX_{K},\bY_{K}^{*})\}_{K=1}^{+\infty}$ are bounded, and the latter has a subsequence that converges weakly to a saddle point of the Lagrangian~\eqref{eq:lagrangian}. If, in addition, the spaces $\xcal$ and $\ycals$ are finite-dimensional, then the sequences $\{(\bx_{k},\by_{k}^{*})\}_{k=1}^{+\infty}$ and $\{(\bX_{K},\bY_{K}^{*})\}_{K=1}^{+\infty}$ both converge strongly to the same saddle point.
\end{enumerate}
\end{prop}
\begin{proof}
See Appendix~\ref{app:B}.
\end{proof}

%% file: 4_accelerated_schemes.tex
\section{Accelerated nonlinear primal-dual hybrid gradient methods}\label{sec:acc_pd_algs}

In this section, we describe accelerated nonlinear PDHG method~\eqref{eq:basic_pdhg_alg} that are suitable when the functions $g$ and $h^*$ in the saddle-point problem~\eqref{eq:saddle_prob} have additional structure beyond that stated in assumptions (A1)-(A5). Specifically, we assume either one or both of these statements:
\begin{itemize}
    \item[(A6)] There is a positive number $\gamma_g$ such that the function $\bx \mapsto g(\bx) - \gamma_g \phix(\bx)$ is convex.
    
    \item[(A7)] There is a positive number $\gamma_{h^{*}}$ such that the function $\bys \mapsto h^*(\bys) - \gamma_{h^*}\phiys(\bys)$ is convex.
\end{itemize}
With assumptions (A1)-(A7), the descent rule~\eqref{eq:descent_rule} in Lemma~\ref{lem:descent_rule} is improved: For every $(\bx,\bys) \in \dom g \times \dom h^*$, the output $(\bxh,\bysh)$ satisfies
\begin{equation}\label{eq:descent_rule_sc}
    \begin{alignedat}{1}
    \mathcal{L}(\bxh,\bys) - \mathcal{L}(\bx,\bysh) &\leqslant \frac{1}{\tau}\left(D_{\phix}(\bx,\bxb) - D_{\phix}(\bxh,\bxb)\right) - \left(\frac{1 + \gamma_g\tau}{\tau} \right)D_{\phix}(\bx,\bxh)\\
    &\quad + \frac{1}{\sigma}\left(D_{\phiys}(\bys,\bysb) - D_{\phiys}(\bysh,\bysb) \right) - \left(\frac{1+\gamma_{h^*}\sigma}{\sigma}\right)D_{\phiys}(\bys,\bysh) \\
    &\quad + \left\langle \byst - \bysh, \matr{A}(\bx  - \bxt)\right\rangle - \left\langle \bys - \byst, \matr{A}(\bxt - \bxh) \right\rangle.
    \end{alignedat}
\end{equation}
The proof of inequality~\eqref{eq:descent_rule_sc} is nearly identical to the proof of inequality~\eqref{eq:descent_rule}, with the only difference that we use the stronger inequality~\eqref{eq:char_prox_mapping_sc} (see Fact~\ref{fact:breg_prox_prop}(iii)) on each line of the iteration scheme~\eqref{eq:iteration_PG} to get
\begin{equation*}
g(\bxh) - g(\bx) \leqslant \frac{1}{\tau}\left(D_{\phix}(\bx,\bar{\bx}) - D_{\phix}(\bxh,\bxb)\right) - \left(\frac{1 + \gamma_g\tau}{\tau}\right)D_{\phix}(\bx,\bxh) + \left\langle \matr{A}^*\byst, \bx - \bxh \right\rangle.
\end{equation*}
and
\begin{equation*}
h^*(\bysh) - h^{*}(\bys) \leqslant \frac{1}{\sigma}\left(D_{\phiys}(\bys,\bysb) - D_{\phiys}(\bysh,\bysb))\right) - \left(\frac{1 + \gamma_{h^*}\sigma}{\sigma}\right)D_{\phiys}(\bys,\bysh) - \left\langle \bys -\bysh,\matr{A}\bxt \right\rangle.
\end{equation*}
Inequality~\eqref{eq:descent_rule_sc} then follows from these two inequalities and the same steps used to prove the descent rule~\eqref{eq:descent_rule} in Lemma~\ref{lem:descent_rule}.

\begin{rem}\label{rem:strong-conv}
Assumptions (A1)-(A7) imply that $g$ is $\gamma_g$-strongly convex over $\dom g \cap \dom \phix$ and $h^*$ is $\gamma_{h^{*}}$-strongly convex over $\dom h^* \cap \phiys$. For example,
\[
g(\bx) - \frac{\gamma_g}{2}\normx{\bx}^2 = \left(g(\bx) - \gamma_g\phix(\bx)\right) + \gamma_g\left(\phix(\bx) - \frac{1}{2}\normx{\bx}^2\right)
\]
for every $\bx \in \dom g \cap \dom \phix$, and as the set $\dom g \cap \dom \phix$ is convex and the right hand side is the sum of two convex functions, the left hand side is also convex.
\end{rem}

\begin{rem}\label{rem:uniqueness}
In light of Remark~\ref{rem:strong-conv} and Fact~\ref{fact:primal_dual_prob}, if assumptions (A1) and (A6) hold, then the primal problem~\eqref{eq:primal_prob} has a unique solution. Likewise, if assumptions (A1) and (A7) hold, then the dual problem~\eqref{eq:dual_prob} has a unique solution. Finally, if assumptions (A1) and (A6)-(A7) hold, then the Lagrangian~\eqref{eq:lagrangian} has a unique saddle point.
\end{rem}

The additional terms in~\eqref{eq:descent_rule_sc} allow us to create accelerated methods with better convergence rate than the $O(1/K)$ rate for estimate~\eqref{eq:basic_rateI}. The first accelerated method, which we describe in Section~\ref{subsec:accI}, has a sublinear $O(1/K^2)$ convergence rate and is applicable if assumption (A6) hold. A variant of the first accelerated method, which we describe in Section~\ref{subsec:accII}, has a sublinear $O(1/K^2)$ convergence rate and is applicable if assumption (A7) hold. The second accelerated method, which we describe in Section~\ref{subsec:accIII}, has a linear convergence rate and is applicable if both assumptions (A6) and (A7) hold. We also present another variant of this method in Section~\ref{subsec:accII_var}.

\subsection{Accelerated nonlinear PDHG methods for strongly convex problems}\label{subsec:accI}

This accelerated nonlinear PDHG method requires statement (A6) to hold with $\gamma_g > 0$. It takes two parameters $\theta_{0} \in (0,1]$ and $\sigma_{0} > 0$, a set parameter $\tau_{0} = 1/(\normop{\matr{A}}^2\sigma_0)$, an initial point $\bx_{0} \in \dom \partial g$, and the initial points $\by_{-1}^{*} = \by_{0}^{*} \in \dom \partial h^*$ to generate the iterates
\begin{equation}\label{eq:accI_pdhg_alg}
\begin{dcases}
    \bx_{k+1} &= \argmin_{\bx \in \xcal} \left\{g(\bx) + \left\langle \by_{k}^{*} + \theta_{k}(\by_{k}^{*} - \by_{k-1}^{*}), \matr{A}\bx\right\rangle + \frac{1}{\tau_{k}}D_{\phix}(\bx,\bx_{k})\right\} \\
    \by_{k+1}^{*} &= \argmax_{\by^{*} \in \ycal^*} \left\{-h^{*}(\by^{*}) + \left\langle \by^{*}, \matr{A}\bx_{k+1}\right\rangle - \frac{1}{\sigma_{k}}D_{\phiys}(\by^{*},\by_{k}^{*})\right\}, \\
\end{dcases}
\end{equation}
where the parameters $\tau_{k}, \sigma_{k}, \theta_{k}$ for $k \in \mathbb{N}$ satisfy the recurrence relations
\begin{equation}\label{eq:acc_sq_exact_iter}
\theta_{k+1} = \frac{1}{\sqrt{1 + \gamma_g \tau_{k}}}, \quad \tau_{k+1} = \theta_{k+1}\tau_{k}, \quad \mathrm{and} \quad \sigma_{k+1} = \sigma_{k}/\theta_{k+1}.
\end{equation}
Under assumptions (A1)-(A4), Lemma~\ref{lem:descent_rule} applies to method~\eqref{eq:accI_pdhg_alg}, and the method generates points $(\bx_{k},\by_{k}^{*})$ that are contained in $\dom \partial g \times \dom \partial h^*$. As such, method~\eqref{eq:accI_pdhg_alg} is well-defined. If, in addition, assumptions (A5)-(A6) hold, then this method satisfies the following properties:

\begin{prop}\label{prop:acc_sq_algI}
Assume (A1)-(A6) hold. Let $\theta_{0} \in (0,1]$, $\sigma_{0} > 0$ and $\tau_{0} = 1/(\normop{\matr{A}}^2\sigma_{0})$, let $(\bx_{0},\by_{0}^{*}) \in \dom \partial g \times \dom \partial h^*$, let $\by_{-1}^{*} = \by_{0}^{*}$, and let $(\bx_s,\by_{s}^{*})$ denote a saddle point of the Lagrangian~\eqref{eq:lagrangian}. Consider the sequence of iterates
$\{(\bx_{k},\by_{k}^{*})\}_{k=1}^{K}$ with $K \in \mathbb{N}$ generated by the accelerated nonlinear PDHG method~\eqref{eq:accI_pdhg_alg} and the recurrence relations~\eqref{eq:acc_sq_exact_iter} from the initial points $(\bx_0,\by_{0}^{*})$ and $\by_{-1}^{*}$ and initial parameters $\tau_0$, $\sigma_0$ and $\theta_0$. Define the averages
\[
T_{K} = \sum_{k=1}^{K}\frac{\sigma_{k-1}}{\sigma_{0}}, \quad \bX_{K} = \frac{1}{T_{K}}\sum_{k=1}^{K} \frac{\sigma_{k-1}}{\sigma_{0}} \bx_{k} \quad \mathrm{and} \quad \bY_{K}^{*} = \frac{1}{T_{K}}\sum_{k=1}^{K} \frac{\sigma_{k-1}}{\sigma_{0}} \by_{k}^{*}
\]
and for $(\bx,\by^{*}) \in \dom g \times \dom h^*$, the quantity
\begin{equation}\label{eq:app_d_symbol}
    \Delta_{k}(\bx,\by^{*}) = \frac{1}{\tau_{k}}D_{\phix}(\bx,\bx_{k}) + \frac{1}{\sigma_{k}}D_{\phiys}(\bys,\by_{k}^{*}) + \frac{1}{\sigma_{k}}D_{\phiys}(\by_{k}^{*},\by_{k-1}^{*}) + \theta_{k}\left\langle\by_{k}^{*}-\by_{k-1}^{*},\matr{A}(\bx-\bx_{k}) \right\rangle.
\end{equation}
Then:
\begin{enumerate}
    \item[(a)] For every $(\bx,\bys) \in \dom g \times \dom h^*$ and nonnegative integer $k$, the output $(\bx_{k+1},\by_{k+1}^{*})$ of the accelerated nonlinear PDHG method~\eqref{eq:accI_pdhg_alg} satisfies the descent rule
    \begin{equation}\label{eq:descent_rule_s_acc_I}
    \begin{alignedat}{1}
    \mathcal{L}(\bx_{k+1},\bys) - \mathcal{L}(\bx,\by_{k+1}^{*}) &\leqslant \Delta_{k}(\bx,\by^{*}) - \Delta_{k+1}(\bx,\by^{*})/\theta_{k+1}.
    \end{alignedat}    
    \end{equation}
    
    \item[(b)] For every $(\bx,\bys) \in \dom g \times \dom h^*$, we have the estimate
    \begin{equation}\label{eq:acc_sq__rateI}
    T_{K}\left(\mathcal{L}(\bX_{K},\bys) - \mathcal{L}(\bx,\bY_{K}^{*})\right) \leqslant \Delta_{0}(\bx,\by^{*}) - \frac{\sigma_{K}}{\sigma_{0}}\Delta_{K}(\bx,\by^{*})
    \end{equation}
    and, for the choice of the saddle point $(\bx,\by^{*}) = (\bx_{s},\by_{s}^{*})$, the global bound
    \begin{equation}\label{eq:bound_seq_acc_I}
        \frac{\gamma_{g}}{1 + \gamma_{g}\tau_{0}}D_{\phix}(\bx,\bx_{K}) + \frac{1}{\sigma_{K}}D_{\phiys}(\by^{*},\by_{K}^{*}) \leqslant \Delta_{K}(\bx_{s},\by_{s}^{*}) \leqslant \frac{\sigma_0}{\sigma_{K}}\Delta_{0}(\bx_{s},\by_{s}). 
    \end{equation}
    
    \item[(c)] The average quantity $T_K$ satisfies the formula
    \begin{equation}\label{eq:acc_exact_t_avg}
    T_{K} = \normop{\matr{A}}^2(\sigma_{K}^{2} - \sigma_{0}^{2})/(\gamma_g\sigma_{0})
    \end{equation}
    and, with $a = \gamma_g/(2\normop{\matr{A}}^2)$, the bounds
    \begin{equation}\label{eq:acc_I_avg_bounds}
    \frac{\sigma_0}{a + \sigma_0}K + \frac{a\sigma_0}{2(a + \sigma_0)^2}K^2 \leqslant T_{K} \leqslant K + \frac{a}{2\sigma_0}K^2.
    \end{equation}

    \item[(d)] [Convergence properties] The sequence of iterates $\{(\bx_{k},\by_{k}^{*})\}_{k=1}^{+\infty}$ is bounded and the individual sequence $\{\bx_{k}\}_{k=1}^{+\infty}$ converges strongly to the unique solution of the primal problem~\eqref{eq:primal_prob}. Moreover, the sequence of averages $\{(\bX_{K},\bY_{K})\}_{K=1}^{+\infty}$ is bounded, it has subsequence that converges weakly to a saddle point of the Lagrangian~\eqref{eq:lagrangian}, and the individual sequence $\{\bX_{K}\}_{K=1}^{+\infty}$ converges strongly to the unique solution of the primal problem~\eqref{eq:primal_prob}. If, in addition, the space $\ycals$ is finite-dimensional, then the individual sequences $\{\by_{k}^{*}\}_{k=1}^{+\infty}$ and $\{\bY_{k}^{*}\}_{k=1}^{+\infty}$ each have a subsequence that converges strongly to a solution $\by_{s}^{*}$ of the dual problem~\eqref{eq:dual_prob}.
    
\end{enumerate}
\end{prop}
\begin{proof}
See Appendix~\ref{app:C}.
\end{proof}

\begin{rem}[Choice of the free parameter $\sigma_0$]\label{rem:param_choice}
The accelerated nonlinear PDHG method~\eqref{eq:accI_pdhg_alg} converges at a rate determined by the average quantity $T_{K}$, which depends on the stepsize parameter $\sigma_{0}$. One possible choice for $\sigma_0$ is to choose it so as to maximize the coefficient multiplying $K^2$ in the lower bound~\eqref{eq:acc_I_avg_bounds} of $T_{K}$. This coefficient is maximized for the choice of $\sigma_0 = \gamma_g/(2\normop{\matr{A}}^2)$.
\end{rem}

\subsection{Accelerated nonlinear PDHG methods for smooth convex problems}\label{subsec:accII}

We present a variant of the first accelerated nonlinear PDHG method. It requires statement (A7) to hold with $\gamma_{h^*} > 0$, and it is similar to method~\eqref{eq:accI_pdhg_alg}; it takes two free parameters $\theta_{0} \in (0,1]$ and $\tau_{0} > 0$, a set parameter $\sigma_{0} = 1/(\normop{\matr{A}}^2\tau_0)$, an initial point $\by_{0}^{*} \in \dom \partial h^*$, and the initial points $\bx_{-1} = \bx_{0} \in \dom \partial g$ to generate the iterates
\begin{equation}\label{eq:accII_pdhg_alg}
\begin{dcases}
    \by_{k+1}^{*} &= \argmax_{\by^{*} \in \ycal^*} \left\{-h^{*}(\by^{*}) + \left\langle \by^{*}, \matr{A}(\bx_{k} + \theta_{k}(\bx_{k} - \bx_{k-1}))\right\rangle - \frac{1}{\sigma_{k}}D_{\phiys}(\by^{*},\by_{k}^{*})\right\}, \\
    \bx_{k+1} &= \argmin_{\bx \in \xcal} \left\{g(\bx) + \left\langle \by_{k+1}^{*}, \matr{A}\bx\right\rangle + \frac{1}{\tau_{k}}D_{\phix}(\bx,\bx_{k})\right\} \\
\end{dcases}
\end{equation}
where the parameters $\tau_{k}, \sigma_{k}, \theta_{k}$ for $k \in \mathbb{N}$ satisfy the recurrence relations
\begin{equation}\label{eq:acc_II_sq_exact_iter}
\theta_{k+1} = \frac{1}{\sqrt{1 + \gamma_{h^*} \sigma_{k}}}, \quad \tau_{k+1} = \tau_{k}/\theta_{k+1}, \quad \mathrm{and} \quad \sigma_{k+1} = \theta_{k+1}\sigma_{k}.
\end{equation}
Under assumptions (A1)-(A4), Lemma~\ref{lem:descent_rule} applies to method~\eqref{eq:accII_pdhg_alg}, and the method generates points $(\bx_{k},\by_{k}^{*})$ that are contained in $\dom \partial g \times \dom \partial h^*$. As such, method~\eqref{eq:accII_pdhg_alg} is well-defined. If, in addition, assumptions (A5) and (A7) hold, then this method satisfies the following properties:
\begin{prop}\label{prop:acc_sq_algII}
Assume (A1)-(A5) and (A7) hold. Let $\theta_{0} \in (0,1]$, $\tau_{0} > 0$ and $\sigma_{0} = 1/(\normop{\matr{A}}^2\tau_{0})$, let $(\bx_0,\by_{0}^{*}) \in \dom \partial g \times \dom \partial h^*$, let $\bx_{-1}^{*} = \bx_{0}^{*}$, and let $(\bx_{s},\by_{s}^{*})$ denote a saddle point of the Lagrangian~\eqref{eq:lagrangian}. Consider the sequence of iterates $\{(\bx_{k},\by_{k}^{*})\}_{k=1}^{K}$ with $K \in \mathbb{N}$ generated by the accelerated nonlinear PDHG method~\eqref{eq:accII_pdhg_alg} and the recurrence relations~\eqref{eq:acc_II_sq_exact_iter} from the initial points $(\bx_0,\by_{0}^{*})$ and $\bx_{-1}$ and initial parameters $\tau_0$, $\sigma_0$, and $\theta_0$, and define the averages
\[
T_{K} = \sum_{k=1}^{K}\frac{\tau_{k-1}}{\tau_{0}}, \quad \bX_{K} = \frac{1}{T_{K}}\sum_{k=1}^{K} \frac{\tau_{k-1}}{\tau_{0}} \bx_{k} \quad \mathrm{and} \quad \bY_{K}^{*} = \frac{1}{T_{K}}\sum_{k=1}^{K} \frac{\tau_{k-1}}{\tau_{0}} \by_{k}^{*}
\]
and for $(\bx,\by^{*}) \in \dom g \times \dom h^*$, the quantity
\begin{equation*}
    \Delta_{k}(\bx,\by^{*}) = \frac{1}{\tau_{k}}D_{\phix}(\bx,\bx_{k}) + \frac{1}{\sigma_{k}}D_{\phiys}(\bys,\by_{k}^{*}) + \frac{1}{\tau_{k}}D_{\phix}(\bx_{k},\bx_{k-1}) + \theta_{k}\left\langle\by^{*}-\by_{k}^{*},\matr{A}(\bx_{k}-\bx_{k-1}) \right\rangle.
\end{equation*}
Then:
\begin{enumerate}
    \item[(a)] For every $(\bx,\bys) \in \dom g \times \dom h^*$ and nonnegative integer $k$, the output $(\bx_{k+1},\by_{k+1}^{*})$ of the accelerated nonlinear PDHG method~\eqref{eq:accII_pdhg_alg} satisfies the descent rule
    \begin{equation*}
    \begin{alignedat}{1}
    \mathcal{L}(\bx_{k+1},\bys) - \mathcal{L}(\bx,\by_{k+1}^{*}) &\leqslant \Delta_{k}(\bx,\by^{*}) - \Delta_{k+1}(\bx,\by^{*})/\theta_{k+1}.
    \end{alignedat}    
    \end{equation*}
    
    \item[(b)] For every $(\bx,\bys) \in \dom g \times \dom h^*$, we have the estimate
    \begin{equation*}
    T_{K}\left(\mathcal{L}(\bX_{K},\bys) - \mathcal{L}(\bx,\bY_{K}^{*})\right) \leqslant \Delta_{0}(\bx,\by^{*}) - \frac{\tau_{K}}{\tau_{0}}\Delta_{K}(\bx,\by^{*})
    \end{equation*}
    and, for the choice of the saddle point $(\bx,\by^{*}) = (\bx_{s},\by_{s}^{*})$, the global bound
    \begin{equation*}
        \frac{1}{\tau_{K}}D_{\phix}(\bx_{s},\bx_{K}) + \frac{\gamma_{h^{*}}}{1 + \gamma_{h^{*}}\sigma_{0}}D_{\phiys}(\by^{*}_{s},\by^{*}_{K}) \leqslant \Delta_{K}(\bx_{s},\by_{s}^{*}) \leqslant \frac{\tau_0}{\tau_{K}}\Delta_{0}(\bx_{s},\by_{s}). 
    \end{equation*}
    
    \item[(c)] The average quantity $T_K$ satisfies the formula
    \begin{equation*}
    T_{K} = \normop{\matr{A}}^2(\tau_{K}^{2} - \tau_{0}^{2})/(\gamma_{h^*}\tau_{0})
    \end{equation*}
    and, with $a = \gamma_{h^*}/(2\normop{\matr{A}}^2)$, the bounds
    \begin{equation*}
    \frac{\tau_0}{a + \tau_0}K + \frac{a\tau_0}{2(a + \tau_0)^2}K^2 \leqslant T_{K} \leqslant K + \frac{a}{2\tau_0}K^2.
    \end{equation*}
    
    \item[(d)] [Convergence properties] The sequence of iterates $\{(\bx_{k},\by_{k}^{*})\}_{k=1}^{+\infty}$ is bounded and the individual sequence $\{\by_{k}^{*}\}_{k=1}^{+\infty}$ converges strongly to the unique solution of the dual problem~\eqref{eq:dual_prob}. Moreover, the sequence of averages $\{(\bX_{K},\bY_{K})\}_{K=1}^{+\infty}$ is bounded, it has subsequence that converges weakly to a saddle point of the Lagrangian~\eqref{eq:lagrangian}, and the individual sequence $\{\bY_{K}^{*}\}_{K=1}^{+\infty}$ converges strongly to the unique solution of the dual problem~\eqref{eq:dual_prob}. If, in addition, the space $\xcal$ is finite-dimensional, then the individual sequences $\{\bx_{k}\}_{k=1}^{+\infty}$ and $\{\bX_{k}\}_{k=1}^{+\infty}$ each have a subsequence that converges strongly to a solution $\bx_{s}$ of the primal problem~\eqref{eq:dual_prob}.
    
\end{enumerate}
\end{prop}
\begin{proof}
The proof is essentially the same as for Proposition~\ref{prop:acc_sq_algI} and is omitted.
\end{proof}

\subsection{Accelerated nonlinear PDHG method for smooth and strongly convex problems I}\label{subsec:accIII}

The second accelerated nonlinear PDHG method requires statements (A6) and (A7) to hold with $\gamma_g > 0$ and $\gamma_{h^*} > 0$. It takes the parameters
\begin{equation}\label{eq:accIII_params}
    \theta = 1  - \frac{\gamma_g\gamma_{h^*}}{2\normop{\matr{A}}^2}\left(\sqrt{1+\frac{4\normop{\matr{A}}^2}{\gamma_g\gamma_{h^*}}}-1\right), \quad \tau = \frac{1-\theta}{\gamma_g\theta}, \quad \mathrm{and} \quad \sigma = \frac{1-\theta}{\gamma_{h^*}\theta},
\end{equation}
an initial point $\bx_{0} \in \dom \partial g$, and the initial points $\by_{-1}^{*} = \by_{0}^{*} \in \dom \partial h^*$ to generate the iterates
\begin{equation}\label{eq:accIII_pdhg_alg}
\begin{dcases}
    \bx_{k+1} &= \argmin_{\bx \in \xcal} \left\{g(\bx) + \left\langle \by_{k}^{*} + \theta(\by_{k}^{*} - \by_{k-1}^{*}), \matr{A}\bx\right\rangle + \frac{1}{\tau}D_{\phix}(\bx,\bx_{k})\right\}, \\
    \by_{k+1}^{*} &= \argmax_{\by^{*} \in \ycal^*} \left\{-h^{*}(\by^{*}) + \left\langle \by^{*}, \matr{A}\bx_{k+1}\right\rangle - \frac{1}{\sigma}D_{\phiys}(\by^{*},\by_{k}^{*})\right\}. \\
\end{dcases}
\end{equation}
Under assumptions (A1)-(A4), Lemma~\ref{lem:descent_rule} applies to method~\eqref{eq:accIII_pdhg_alg}, and the method generates points $(\bx_{k},\by_{k}^{*})$ that are contained in $\dom \partial g \times \dom \partial h^*$. As such, method~\eqref{eq:accIII_pdhg_alg} is well-defined. If, in addition, assumptions (A5)-(A7) hold, then this method satisfies the following properties:

\begin{prop}\label{prop:acc_sq_algIII}
Assume (A1)-(A7) hold. Let $(\bx_{0},\by_{0}^{*}) \in \dom \partial g \times \dom \partial h^*$, let $\by_{-1}^{*} = \by_{0}^{*}$, and let $(\bx_{s},\by_{s}^{*})$ denote the unique saddle point of the Lagrangian~\eqref{eq:lagrangian}. Consider the sequence of iterates $\{(\bx_{k},\by_{k}^{*})\}_{k=1}^{K}$ with $K \in \mathbb{N}$ generated by the accelerated nonlinear PDHG method~\eqref{eq:accIII_pdhg_alg} from the initial points $\bx_0$, $\by_{0}^{*}$ and $\by_{-1}^{*}$, and the parameters $\theta$, $\tau$, and $\sigma$ defined in~\eqref{eq:accIII_params}. Define the averages
\[
T_{K} = \sum_{k=1}^{K} \frac{1}{\theta^{k-1}} = \frac{1-\theta^{K}}{(1-\theta)\theta^{K-1}}, \quad \bX_{K} = \frac{1}{T_K}\sum_{k=1}^{K}\frac{1}{\theta^{k-1}}\bx_{k} \quad \text{and} \quad \bY_{K}^{*} = \frac{1}{T_K}\sum_{k=1}^{K}\frac{1}{\theta^{k-1}} \by_{k}^{*},
\]
and for $(\bx,\by^{*}) \in \dom g \times \dom h^*$, the quantity
\begin{equation}\label{eq:app_e_symbol}
\Delta_{k}(\bx,\by^{*}) = \frac{1}{\tau}D_{\phix}(\bx,\bx_{k}) + \frac{1}{\sigma}D_{\phiys}(\by^{*},\by_{k}^{*}) + \frac{\theta}{\sigma}D_{\phiys}(\by_{k}^{*},\by_{k-1}^{*}) + \theta\left\langle\by_{k}^{*}-\by_{k-1}^{*},\matr{A}(\bx-\bx_{k}) \right\rangle.
\end{equation}
Then:
\begin{enumerate}
    \item[(a)] For every $(\bx,\bys) \in \dom g \times \dom h^*$ and nonnegative integer $k$, the output $(\bx_{k+1},\by_{k+1}^{*})$ of the accelerated nonlinear PDHG method~\eqref{eq:accIII_pdhg_alg} satisfies the descent rule
    \begin{equation}\label{eq:descent_rule_s_acc_II}
        \mathcal{L}(\bx_{k+1},\by^{*}) - \mathcal{L}(\bx,\by_{k+1}^{*}) \leqslant \Delta_{k}(\bx,\by^{*}) -  \Delta_{k+1}(\bx,\by^{*})/\theta.
    \end{equation}
    
    \item[(b)] For every $(\bx,\bys) \in \dom g \times \dom h^*$, we have the estimate
    \begin{equation}\label{eq:acc_sq__rateIII}
        T_{K}\left(\mathcal{L}(\bX_{K},\by^{*}) - \mathcal{L}(\bx,\bY_{K}^{*})\right) \leqslant \Delta_{0}(\bx,\by^{*}) - \frac{1}{\theta^{K}}\Delta_{K}(\bx,\by^{*})
    \end{equation}
    and, for the choice of the saddle point $(\bx,\by^{*}) = (\bx_{s},\by_{s}^{*})$, the global bound
    \begin{equation}\label{eq:bound_seq_acc_II}
    \frac{1}{\sigma}D_{\phiys}(\by^{*}_{s},\by^{*}_{K})\leqslant \Delta_{K}(\bx_{s},\by^{*}_{s}) \leqslant \theta^{K}\Delta_{0}(\bx_{s},\by^{*}_{s}).
    \end{equation}
    
    \item [(c)] [Convergence properties] The sequences $\{(\bx_{k},\by_{k}^{*})\}_{k=1}^{+\infty}$ and $\{(\bX_{K},\bY_{K})\}_{K=1}^{+\infty}$ both converge strongly to the unique saddle point $(\bx_{s},\by_{s}^{*})$ of the Lagrangian~\eqref{eq:lagrangian}.
\end{enumerate}
\end{prop}
\begin{proof}
See Appendix~\ref{app:D} for the proof.
\end{proof}


\subsection{Accelerated nonlinear PDHG method for smooth and strongly convex problems II}\label{subsec:accII_var}

We present a variant of the accelerated nonlinear PDHG method~\eqref{eq:accIII_pdhg_alg}. It requires statements (A6) and (A7) to hold with $\gamma_g > 0$ and $\gamma_{h^*} > 0$ and it takes the parameters \begin{equation}\label{eq:accII_params_var}
    \theta = 1  - \frac{\gamma_g\gamma_{h^*}}{2\normop{\matr{A}}^2}\left(\sqrt{1+\frac{4\normop{\matr{A}}^2}{\gamma_g\gamma_{h^*}}}-1\right), \quad \tau = \frac{1-\theta}{\gamma_g\theta}, \quad \mathrm{and} \quad \sigma = \frac{1-\theta}{\gamma_{h^*}\theta},
\end{equation}
the initial points $\bx_{-1} = \bx_{0} \in \dom \partial g$ and an initial point $\by_{0}^{*} \in \dom \partial h^*$ to generate the iterates
\begin{equation}\label{eq:accII_pdhg_alg_var}
\begin{dcases}
    \by_{k+1}^{*} &= \argmax_{\by^{*} \in \ycal^*} \left\{-h^{*}(\by^{*}) + \left\langle \by^{*}, \matr{A}(\bx_k + \theta(\bx_k - \bx_{k-1}))\right\rangle - \frac{1}{\sigma}D_{\phiys}(\by^{*},\by_{k}^{*})\right\}\\
    \bx_{k+1} &= \argmin_{\bx \in \xcal} \left\{g(\bx) + \left\langle \by_{k+1}^{*}, \matr{A}\bx\right\rangle + \frac{1}{\tau}D_{\phix}(\bx,\bx_{k})\right\}. \\
\end{dcases}
\end{equation}
method~\eqref{eq:accII_pdhg_alg_var} and~\eqref{eq:accIII_pdhg_alg} differ only in that we update $\by^{*}$ first. This change nonetheless yields different global bounds and convergence estimates. Under assumptions (A1)-(A4), Lemma~\ref{lem:descent_rule} applies to method~\eqref{eq:accII_pdhg_alg_var}, and the method generates points $(\bx_{k},\by_{k}^{*})$ that are contained in $\dom \partial g \times \dom \partial h^*$. As such, method~\eqref{eq:accII_pdhg_alg_var} is well-defined. If, in addition, assumptions (A5)-(A7) hold, then this method satisfies the following properties:

\begin{prop}\label{prop:acc_sq_algII_var}
Assume (A1)-(A7) hold. Let $(\bx_{0},\by_{0}^{*}) \in \dom \partial g \times \dom \partial h^*$, let $\bx_{-1} = \bx_{0}$, and let $(\bx_{s},\by_{s}^{*})$ denote the unique saddle point of the Lagrangian~\eqref{eq:lagrangian}. Consider the sequence of iterates $\{(\bx_{k},\by_{k}^{*})\}_{k=1}^{K}$ with $K \in \mathbb{N}$ generated by the accelerated nonlinear PDHG method~\eqref{eq:accII_pdhg_alg_var} from the initial points $\bx_0$, $\by_{0}^{*}$ and $\bx_{-1}^{*}$, and the parameters $\theta$, $\tau$, and $\sigma$ defined in~\eqref{eq:accIII_params}. Define the averages
\[
T_{K} = \sum_{k=1}^{K} \frac{1}{\theta^{k-1}} = \frac{1-\theta^{K}}{(1-\theta)\theta^{K-1}}, \quad \bX_{K} = \frac{1}{T_K}\sum_{k=1}^{K}\frac{1}{\theta^{k-1}}\bx_{k} \quad \text{and} \quad \bY_{K}^{*} = \frac{1}{T_K}\sum_{k=1}^{K}\frac{1}{\theta^{k-1}} \by_{k}^{*},
\]
and for $(\bx,\by^{*}) \in \dom g \times \dom h^*$, the quantity
\begin{equation*}
\Delta_{k}(\bx,\by^{*}) = \frac{1}{\tau}D_{\phix}(\bx,\bx_{k}) + \frac{1}{\sigma}D_{\phiys}(\by^{*},\by_{k}^{*}) + \frac{\theta}{\sigma}D_{\phix}(\bx_{k}^{*},\bx_{k-1}^{*}) + \theta\left\langle\by^{*}-\by_{k}^{*},\matr{A}(\bx_{k}-\bx_{k-1}) \right\rangle.
\end{equation*}
Then:
\begin{enumerate}
    \item[(a)] For every $(\bx,\bys) \in \dom g \times \dom h^*$ and nonnegative integer $k$, the output $(\bx_{k+1},\by_{k+1}^{*})$ of the accelerated nonlinear PDHG method~\eqref{eq:accIII_pdhg_alg} satisfies the descent rule
    \begin{equation*}
        \mathcal{L}(\bx_{k+1},\by^{*}) - \mathcal{L}(\bx,\by_{k+1}^{*}) \leqslant \Delta_{k}(\bx,\by^{*}) -  \Delta_{k+1}(\bx,\by^{*})/\theta.
    \end{equation*}
    
    \item[(b)] For every $(\bx,\bys) \in \dom g \times \dom h^*$, we have the estimate
    \begin{equation*}
        T_{K}\left(\mathcal{L}(\bX_{K},\by^{*}) - \mathcal{L}(\bx,\bY_{K}^{*})\right) \leqslant \Delta_{0}(\bx,\by^{*}) - \frac{1}{\theta^{K}}\Delta_{K}(\bx,\by^{*})
    \end{equation*}
   and, for the choice of the saddle point $(\bx,\by^{*}) = (\bx_{s},\by_{s}^{*})$, the global bound
    \begin{equation*}
    \frac{1}{\tau}D_{\phix}(\bx_{s},\bx_{K}) \leqslant \Delta_{K}(\bx_{s},\by^{*}_{s}) \leqslant \theta^{K}\Delta_{0}(\bx_{s},\by^{*}_{s}).
    \end{equation*}
    
    \item [(c)] [Convergence properties] The sequences $\{(\bx_{k},\by_{k}^{*})\}_{k=1}^{+\infty}$ and $\{(\bX_{K},\bY_{K})\}_{K=1}^{+\infty}$ both converge strongly to the unique saddle point $(\bx_{s},\by_{s}^{*})$ of the Lagrangian~\eqref{eq:lagrangian}.
\end{enumerate}
\end{prop}
\begin{proof}
The proof is essentially the same as for Proposition~\ref{prop:acc_sq_algIII} and is omitted.
\end{proof}

%% file: 5_applications.tex
\section{Applications to machine learning}\label{sec:applications}

We describe here applications of the accelerated nonlinear PDHG methods presented in Section~\ref{sec:acc_pd_algs} to two supervised learning tasks in machine learning: $\ell_{1}$-constrained logistic regression and entropy regularized zero-sum matrix games. In both problems, the real reflexive Banach spaces $\xcal$ and $\ycal$ are taken to be $\Rn$ and $\Rm$ with norms $\normx{\cdot}$ and $\normy{\cdot}$ chosen suitably for each example. As Table~\eqref{tab:op_norm_complexity} illustrates, for certain combinations of norms the induced operator norm can be computed in $O(mn)$ operations, making it simple and efficient to compute. In addition, we can choose the Bregman functions $\phix$ and $\phiys$ in conjunction with these norms to ensure that assumption (A5) holds. This is our strategy; in each forthcoming example we will choose the norms $\normx{\cdot}$ and $\normy{\cdot}$ and Bregman functions $\phix$ and $\phiys$ to obtain an explicit accelerated nonlinear PDHG method for which the stepsize parameters and updates in the method can be computed in $O(mn)$ operations. These choices will lead to significantly faster and more efficient methods compared to other competing methods. 

The following two subsections describe the problems of $\ell_{1}$-constrained logistic regression and entropy regularized zero-sum matrix games and give an explicit accelerated nonlinear PDHG method for each problem. Section~\ref{sec:numerics} after this one presents some numerical experiments on randomized synthetic data to compare the running times of our methods to other commonly-used optimization methods.

\begin{table}[ht]
    \centering
    \begin{tabular}{|c|c|c|c|c|}
    \hline
          \multicolumn{2}{|c|}{ } & \multicolumn{3}{c|}{Codomain}\\\cline{3-5}
          
          \multicolumn{2}{|c|}{ } & $(\Rm,\normone{\cdot})$ & $(\Rm,\normtwo{\cdot})$ & $(\Rm,\norminf{\cdot})$\\\hline
          
          \multirow{6}{*}{Domain}& \multirow{2}{*}{$(\Rn,\normone{\cdot})$} & Maximum $\ell_{1}$ norm& Maximum $\ell_{2}$ norm & Maximum $\ell_{\infty}$ norm\\
          
          & & of a column ($\sim O(mn)$) & of a column ($\sim O(mn)$) & of a column ($\sim O(mn)$)\\\cline{2-5}
          
          & \multirow{2}{*}{$(\Rn,\normtwo{\cdot})$} & \multirow{2}{*}{NP-hard} & Largest singular value & Maximum $\ell_{2}$ norm\\
         
          & & & ($\sim O(\min{(m^2n,mn^2)})$) & of a row ($\sim O(mn)$)\\\cline{2-5}
         
          & \multirow{2}{*}{$(\Rn,\norminf{\cdot})$} & \multirow{2}{*}{NP-hard} & \multirow{2}{*}{NP-hard} & Maximum $\ell_{1}$ norm\\
          
          & & & & of a row ($\sim O(mn)$)\\
         \hline
    \end{tabular}
    \caption{Table of some operator norms of $\matr{A}$ with their associated computational complexity. Table extracted from~\cite[Section~4.3.1]{tropp2004topics}.}
    \label{tab:op_norm_complexity}
\end{table}

\subsection{\texorpdfstring{$\ell_{1}$}--constrained logistic regression}\label{subsec:l1-logreg}
$\ell_{1}$-constrained logistic regression is a supervised learning algorithm for classification and to identify important features in data sets. Concretely, suppose we receive $m$ independent samples $\{\bu_{i},b_{i}\}_{i=1}^{m}$, each comprising a $d$-dimensional vector of features $\bu_{i}$ and a label $b_{i} \in \{-1,+1\}$. The $\ell_{1}$-constrained logistic regression problem is then
\begin{equation}\label{eq:l1-logreg}
\inf_{\substack{\bv \in \Rd \\ \normone{\bv} \leqslant \lambda}} \frac{1}{m}\sum_{i=1}^{m} \log\left(1 + e^{-b_{i}\left\langle\bu_{i},\bv\right\rangle}\right)
\end{equation}
where $\lambda > 0$ is a tuning parameter. The constraint on the $\ell_{1}$ norm regularizes the logistic model; it promotes solutions to have a number of entries that are identically zero~\cite{Foucart2013,el2020comparative,zanon2020sparse}. The non-zero entries are identified as the important features, and the zero entries are discarded. The number of non-zero entries itself depends on the value of the tuning parameter $\lambda$.

To derive an appropriate accelerated nonlinear PDHG algorithm for $\ell_{1}$-constrained logistic regression, we will express problem~\eqref{eq:l1-logreg} as an minimization problem over the unit simplex $\Delta_{2d}$. We can do so because every polytope, including the $\ell_{1}$-ball, can be represented as a convex hull of its vertices in barycentric coordinates~\cite{grunbaum1967convex,jaggi2013equivalence}. Here this means for every $\bv$ inside the $\ell_{1}$-ball of radius $\lambda$, there exists a point $\bx$ in the unit simplex $\Delta_{2d}$ for which
\begin{equation}\label{eq:l1ball-to-simplex}
\bv = \lambda (\matr{I}_{d\times d} \mid -\matr{I}_{d\times d})\bx,
\end{equation}
where $(\matr{I}_{d\times d} \mid -\matr{I}_{d\times d})$ denotes the horizontal concatenation of the identity matrices $\matr{I}_{d\times d}$ and $-\matr{I}_{d\times d}$. 

We now apply the change of variables~\eqref{eq:l1ball-to-simplex} to problem~\eqref{eq:l1-logreg}. Let $\matr{B}$ denote the $m \times d$ matrix $\matr{B}$ whose rows are the elements $-b_{i}\bu_{i}$, let $\matr{A} = \lambda(\matr{B} \mid -\matr{B})$, and let $n = 2d$. Then problem~\eqref{eq:l1-logreg} becomes equivalent to
\begin{equation}\label{eq:l1-logreg_primal}
\inf_{\bx \in \Delta_{n}} \frac{1}{m}\sum_{i=1}^{m} \log\left(1 + e^{[\matr{A}\bx]_{i}}\right).
\end{equation}
This is the primal problem of interest. Its associated convex-concave saddle-point problem is
\begin{equation}\label{eq:l1-logreg_saddle}
\inf_{\bx \in \Delta_{n}}\sup_{\by^{*} \in \Rm} \left\{\left\langle \by^{*},\matr{A}\bx\right\rangle - \psi(\by^{*})\right\}
\end{equation}
where $\psi\colon [0,1/m]^{m} \to \R$ denotes the average negative sum of $m$ binary entropy terms,
\begin{equation}\label{eq:binary-entropy}
\psi(\by^{*}) = 
\begin{dcases}
    &\frac{1}{m}\sum_{i=1}^{m}m[\by^{*}]_{i}\log\left(m[\by^{*}]_{i}\right) + (1-m[\by^{*}]_{i})\log\left(1-m[\by^{*}]_{i}\right) \quad \text{if}\, \bs \in [0,1/m]^m, \\
    & +\infty,\quad\text{otherwise}.
\end{dcases}
\end{equation}
The dual problem is
\begin{equation}\label{eq:l1-logreg_dual}
\sup_{\by^{*} \in \Rm}\left\{\vecmax{(-\matr{A}^{*}\by^{*})} - \psi(\by^{*})\right\}
\end{equation}
where $\vecmax{(\by)} = \max{([\by]_{1},\dots,[\by]_{m})}$ for $\by \in \Rm$. Due to the strong concavity of the dual problem~\eqref{eq:l1-logreg_dual}, the convex-concave saddle-point problem~\eqref{eq:l1-logreg_saddle} has at least one saddle point $(\bx_{s},\by^{*}_{s}) \in \Delta_{n} \times \Rm$, where $\bx_{s}$ is a global solution to the primal problem~\eqref{eq:l1-logreg_primal} and $\by_{s}^{*}$ is the unique solution to the dual problem~\eqref{eq:l1-logreg_dual}. They satisfy the optimality conditions
\begin{equation}\label{eq:l1-logreg_saddle-opt}
\bx_{s} \in \partial\vecmax{(-\matr{A}^{*}\by_{s}^{*})} \quad \text{and} \quad [\by_{s}^{*}]_{i} = \frac{1}{m + me^{-[\matr{A}\bx_{s}]_{i}}} \; \text{for}\; i \in \{1,\dots,m\}.
\end{equation}
The solution $\bv_{s}$ of the original problem~\eqref{eq:l1-logreg} follows from $\bx_{s}$ and the change of variables formula~\eqref{eq:l1ball-to-simplex}. In addition, the first optimality condition in~\eqref{eq:l1-logreg_saddle-opt} can be used to identify the zero entries of $\bx_{s}$ as follows~\cite{rockafellar2009variational}: Let $J(-\matr{A}^{*}\by_{s}^{*})$ denote the set of indices $j \in \{1,\dots,n\}$ with $\vecmax{(-\matr{A}^{*}\by_{s}^{*})} = [-\matr{A}^{*}\by_{s}^{*}]_{j}$. Then  $[\bx_{s}]_{j} = 0$ whenever $j \not\in J(-\matr{A}^{*}\by_{s}^{*})$.

\subsubsection*{Accelerated nonlinear PDHG method}
We propose to solve the $\ell_{1}$-constrained logistic regression problem~\eqref{eq:l1-logreg} through~\eqref{eq:l1-logreg_primal} and~\eqref{eq:l1ball-to-simplex} using the accelerated nonlinear PDHG method~\eqref{eq:accII_pdhg_alg} with the following choice of norms and Bregman functions:
\[
\normx{\cdot} = \normone{\cdot}, \quad \normys{\cdot} = \normtwo{\cdot}, \quad \phix = \mathcal{H}_{n}, \quad \text{and} \quad \phiys = \frac{1}{4m}\psi,
\]
where $\mathcal{H}_{n} \colon \Delta_{n} \to (-\infty,0]$ denotes the negative entropy function,
\[
\mathcal{H}_{n}(\bx) = \sum_{j=1}^{n}[\bx]_{j}\log([\bx]_{j}).
\]
The negative entropy function induces the Bregman divergence $D_{\mathcal{H}_{n}} \colon \Delta_{n} \times \interior{~\Delta_{n}} \to [0,+\infty)$ given by
\[
D_{\mathcal{H}_{n}}(\bx,\bxb) = \sum_{j=1}^{n}[\bx]_{j}\log\left([\bx]_{j}/[\bxb]_{j}\right).
\]
This Bregman divergence is the so-called Kullback--Leibler divergence or relative entropy. The Bregman function $\phiys$ is, up to a factor of $1/4m$, the average negative sum of $m$ binary entropy terms~\eqref{eq:binary-entropy}. It induces the Bregman divergence $D_{\psi/4m}\colon [0,1/m]^{m} \times (0,1/m)^{m} \to [0,+\infty)$ given by
\begin{equation*}
D_{\psi/4m}(\by^{*},\bysb) = 
\frac{1}{4m^2}\sum_{i=1}^{m} m[\by^{*}]_i\log\left(\frac{[\by^{*}]_i}{[\bysb]_{i}}\right) + (1 - m[\by^{*}]_i)\log\left(\frac{1-m[\by^{*}]_i}{1-m[\bysb]_{i}}\right).
\end{equation*}
With these choices, assumptions (A1)-(A5) and (A7) hold with $\gamma_{h^{*}} = 4m$. In particular, assumption (A5) holds because $\mathcal{H}_{n}$ is $1$-strongly convex with respect to the $\ell_{1}$ norm over the unit simplex $\Delta_{n}$. This fact is a direct consequence of a fundamental result in information theory known as Pinsker's inequality~\cite{beck2003mirror,csiszar1967information,kemperman1969optimum,kullback1967lower,pinsker1964information}. Moreover, the induced operator norm is the maximum $\ell_2$ norm of the columns of $\matr{A}$, i.e.,
\[
\normop{\matr{A}} = \norm{\matr{A}}_{1,2} = \sup_{\normone{\bx} = 1}\normtwo{\matr{A}\bx} = \max_{j \in \{1,\dots,n\}} \sqrt{\sum_{i=1}^{m}A_{ij}^{2}}.
\]
For this algorithm, we set the initial stepsize parameters to be $\theta_{0} = 0$, $\tau_{0} > 0$ and $\sigma_{0} = 1/(\norm{\matr{A}}_{1,2}^{2}\tau_{0})$. Given $\bx_{-1} = \bx_{0} \in \interior{~\Delta_{n}}$ and $\by_{0}^{*} \in (0,1/m)^{m}$, the corresponding accelerated nonlinear PDHG algorithm for problem~\eqref{eq:l1-logreg_primal} consists of the iterations
\[
\begin{alignedat}{1}
\by^{*}_{k+1} &= \argmax_{\by^{*} \in \Rm} \left\{-\psi(\by^{*}) + \left\langle \by^{*},\matr{A}(\bx_{k} + \theta_{k}[\bx_{k} - \bx_{k-1}])  \right\rangle - \frac{1}{\sigma_{k}}D_{\psi/4m}(\by^{*},\by^{*}_{k})\right\},\\
\bx_{k+1} &= \argmin_{\bx \in \Delta_{n}}\left\{\left\langle\matr{A}^{*}\by_{k+1}^{*},\bx \right\rangle + \frac{1}{\tau_{k}}D_{\mathcal{H}_{n}}(\bx,\bx_{k})\right\}\\
\theta_{k+1} &= 1/\sqrt{1 + 4m\sigma_{k}}, \quad \tau_{k+1} = \tau_{k}/\theta_{k+1}, \quad \text{and} \quad \sigma_{k+1} = \theta_{k+1}\sigma_{k}.
\end{alignedat}
\]

The updates $\by^{*}_{k+1}$ and $\bx_{k+1}$ can be both computed explicitly. For the first update, define the auxiliary variable
\[
[\bw_{k}]_{i} = \log\left(m[\by_{k}^{*}]_{i}/(1-m[\by_{k}^{*}]_{i})\right) \; \text{for} \; i \in \{1,\dots m\}.
\]
Then we can update $\by_{k+1}^{*}$ in two steps:
\[
\bw_{k+1} = \left(4m\sigma_{k}\bx_{k} + 4m\sigma_{k}\theta_{k}\left(\bx_{k} - \bx_{k-1}\right) + \bw_{k}\right)/(1+4m\sigma_{k})
\]
and
\[
[\by^{*}_{k+1}]_{i} = \frac{1}{m + me^{-[\bw_{k+1}]_{i}}} \quad \mathrm{for}\,i\in\left\{1,\dots,m\right\}.
\]
For the second update, a straightforward calculation gives
\[
[\bx_{k+1}]_{j} = \frac{[\bx_{k}]_{j}e^{-\tau_{k}[\matr{A}^{*}\by_{k+1}^{*}]_{j}}}{\sum_{j=1}^{m}[\bx_{k}]_{j}e^{-\tau_{k}[\matr{A}^{*}\by_{k+1}^{*}]_{j}}}
\]
for $j \in \{1,\dots,n\}$. Hence the iterations are given by
\begin{equation}\label{eq:l1-logreg_explicit}
    \begin{alignedat}{1}
    \bw_{k+1} &= \left(4m\sigma_{k}\bx_{k} + 4m\sigma_{k}\theta_{k}\left(\bx_{k} - \bx_{k-1}\right) + \bw_{k}\right)/(1+4m\sigma_{k})\\
    [\by_{k+1}^{*}]_{i} &= [\by^{*}_{k+1}]_{i} = \frac{1}{m + me^{-[\bw_{k+1}]_{i}}} \quad \text{for}\;i\in\left\{1,\dots,m\right\}\\
    [\bx_{k+1}]_{j} &= \frac{[\bx_{k}]_{j}e^{-\tau_{k}[\matr{A}^{*}\by_{k+1}^{*}]_{j}}}{\sum_{j=1}^{m}[\bx_{k}]_{j}e^{-\tau_{k}[\matr{A}^{*}\by_{k+1}^{*}]_{j}}} \quad \text{for}\;j\in\left\{1,\dots,n\right\}\\
    \theta_{k+1} &= 1/\sqrt{1 + 4m\sigma_{k}}, \quad \tau_{k+1} = \tau_{k}/\theta_{k+1}, \quad \text{and} \quad \sigma_{k+1} = \theta_{k+1}\sigma_{k}.
    \end{alignedat}
\end{equation}
All parameter calculations and updates can be performed in $O(mn)$ operations. According to Prop.~\ref{prop:acc_sq_algII} and the optimality conditions~\eqref{eq:l1-logreg_saddle}, we have the strong limits
\[
\lim_{k \to +\infty} \by_{k}^{*} = \by_{s}^{*} \quad \text{and} \quad \lim_{k \to +\infty} [\by_{k}^{*}]_{i} = \frac{1}{m + me^{-[\matr{A}\bx_{s}]_{i}}} \; \text{for}\; i \in \{1,\dots,m\}.
\]

\subsection{Zero-sum matrix games with entropy regularization}\label{subsec:mgames}
Two-player zero-sum matrix games are a class of saddle-point optimization problems that model one of the basic forms of constrained competitive games~\cite{cen2021fast}. We focus here on zero-sum matrix games with entropy regularization, the latter which models the imperfect knowledge of the payoff matrix $\matr{A}$ by the two players~\cite{mckelvey1995quantal}. Let $\Delta_{m}$ and $\Delta_{n}$ denote the unit simplices on $\Rm$ and $\Rn$, and let $\matr{A}$ denote an $m \times n$ matrix, called the payoff matrix. Zero-sum matrix games with entropy regularization are formulated as follow:
\begin{equation}\label{eq:matrix-games_saddle}
    \min_{\bx \in \Delta_{n}}\max_{\by^{*} \in \Delta_{m}} \left\{\lambda \mathcal{H}_{n}(\bx) + \left\langle\by^{*},\matr{A}\bx\right\rangle - \lambda \mathcal{H}_{m}(\by^{*})\right\},
\end{equation}
where $\lambda > 0$ and $\mathcal{H}_{n}(\bx) = \sum_{j=1}^{n}[\bx]_{j}\log([\bx]_{j})$ and $\mathcal{H}_{m}(\by^{*}) = \sum_{i=1}^{m}[\by^{*}]_{i}\log([\by^{*}]_{i})$ denote the negative entropies of the probability distributions $\bx$ and $\by^{*}$.

The primal and dual problems associated to the entropy regularized zero-sum matrix game~\eqref{eq:matrix-games_saddle} are given by
\begin{equation}\label{eq:matrix-games_primal}
    \min_{\bx \in \Delta_{n}} \left\{\lambda \mathcal{H}_{n}(\bx) + \lambda \log\left(\sum_{i=1}^{m}e^{[\matr{A}\bx]_{i}/\lambda}\right)\right\}
\end{equation}
and
\begin{equation}\label{eq:matrix-games_dual}
        \max_{\by^{*} \in \Delta_{n}} \left\{- \lambda \log\left(\sum_{j=1}^{n}e^{-[\matr{A}^{*}\by^{*}]_{j}/\lambda}\right) -\lambda \mathcal{H}_{m}(\by^{*})\right\}
\end{equation}
Due to the strong convexity of the primal problem~\eqref{eq:matrix-games_primal} and strong concavity of the dual problem~\eqref{eq:matrix-games_dual}, the saddle-point problem~\eqref{eq:matrix-games_saddle} has a unique saddle point $(\bx_{s},\by_{s}^{*}) \in \Rn \times \Rm$, which are also the unique solutions to the primal and dual problems above. They satisfy the optimality conditions
\begin{equation}\label{eq:matrix-games_optcond}
    -[\matr{A}^{*}\by_{s}^{*}]_{j} = \lambda(1 + \log([\bx_{s}]_{j})) \quad \text{and} \quad [\by_{s}^{*}]_{i} = \frac{e^{[\matr{A}\bx_{s}]_{i}/\lambda}}{\sum_{i = 1}^{m}e^{[\matr{A}\bx_{s}]_{i}/\lambda}}.
\end{equation}

\subsubsection*{Accelerated nonlinear PDHG method}
We propose to solve the zero-sum matrix game with entropy regularization~\eqref{eq:matrix-games_saddle} using the accelerated PDHG method~\eqref{eq:accII_pdhg_alg_var} with the following choice of norms and Bregman functions:
\[
\normx{\cdot} = \normone{\cdot}, \quad \normy{\cdot} = \norminf{\cdot} \implies \normys{\cdot} = \normone{\cdot}, \quad \phix = \mathcal{H}_{n}, \quad \text{and} \quad \phiys = \mathcal{H}_{m}. 
\]
The Bregman divergences induced by $\phix$ and $\phiys$ are the Kullback--Leibler divergences
\[
D_{\mathcal{H}_{n}}(\bx,\bxb) = \sum_{j=1}^{n}[\bx]_{j}\log\left([\bx]_{j}/[\bxb]_{j}\right)\quad \text{and} \quad D_{\mathcal{H}_{m}}(\by^{*},\bysb) = \sum_{i=1}^{m}[\by^{*}]_{i}\log\left([\bys]_{i}/[\bysb]_{i}\right)
\]
where $\bx \in \Delta_{n}$, $\bxb \in \interior{~\Delta_{n}}$, $\by \in \Delta_{m}$ and $\bysb \in \interior{~\Delta_{m}}$. With these choices, assumptions (A1)-(A7) hold with the strong convexity parameters $\gamma_{g} = \gamma_{h^{*}} = \lambda$. In particular, assumption (A5) holds because both $\mathcal{H}_{n}$ and $\mathcal{H}_{m}$ are $1$-strongly convex with respect to the $l_1$ norm over their respective unit simplices, due to Pinsker's inequality~\cite{beck2003mirror,csiszar1967information,kemperman1969optimum,kullback1967lower,pinsker1964information}. Moreover, the induced operator norm is the entry of the payoff matrix $\matr{A}$ with largest magnitude:
\[
\normop{\matr{A}} = \norm{\matr{A}}_{1,\infty} = \sup_{\normone{\bx} = 1}\norminf{\matr{A}\bx} = \max_{\substack{i \in \{1,\dots,m\} \\ j \in \{1,\dots,n\}}}|A_{ij}|.
\]
The stepsize parameters $\theta$, $\tau$, and $\sigma$ are accordingly
\[
\theta = 1 - \frac{\lambda^2}{2\norm{\matr{A}}_{1,\infty}^2}\left(\sqrt{1 + \frac{4\norm{\matr{A}}_{1,\infty}^2}{\lambda^2}} - 1\right) \quad \text{and} \quad \tau = \sigma = \frac{1-\theta}{\lambda\theta}.
\]
Given $\by^{*}_{0} \in \Rm$ and $\bx^{*}_{-1} = \bx_{0}^{*} \in \Rn$, the corresponding accelerated nonlinear PDHG method for the matrix game~\eqref{eq:matrix-games_saddle} consists of the iterations
\[
\begin{alignedat}{1}
\by_{k+1} &= \argmax_{\by^{*} \in \Delta_{m}}\left\{-\lambda \mathcal{H}_{m}(\by^{*}) + \left\langle \by^{*},\matr{A}(\bx_{k} - \theta(\bx_{k}-\bx_{k-1})) \right\rangle - \frac{1}{\sigma}D_{\mathcal{H}_{m}}(\by^{*},\by^{*}_{k})\right\}, \\
\bx_{k+1} &= \argmin_{\bx \in \Delta_{n}}\left\{\lambda\mathcal{H}_{n}(\bx) + \left\langle \by_{k+1}^{*}, \matr{A}\bx \right\rangle + \frac{1}{\tau}D_{\mathcal{H}_{n}}(\bx,\bx_{k})\right\}.
\end{alignedat}
\]

The updates $\bx_{k+1}$ and $\by^{*}_{k+1}$ can be both computed explicitly. A straightforward calculation gives the updates
\begin{equation}\label{eq:alg-matrix-games}
\begin{alignedat}{1}
[\by_{k+1}^{*}]_{i} &= \frac{\left([\by_{k}^{*}]_{i}e^{-\tau[\matr{A}\left(\bx_{k} - \theta(\bx_{k}-\bx_{k-1})\right)]_{i}}\right)^{1/(1+\lambda\sigma)}}{\sum_{i=1}^{m}\left([\by_{k}^{*}]_{i}e^{-\tau[\matr{A}\left(\bx_{k} - \theta(\bx_{k}-\bx_{k-1})\right)]_{i}}\right)^{1/(1+\lambda\sigma)}}\\
[\bx_{k+1}]_{j} &=  \frac{\left([\bx_{k}^{*}]_{j}e^{-\tau[\matr{A}^{*}\by_{k+1}^{*}]_{j}}\right)^{1/(1+\lambda\tau)}}{\sum_{j=1}^{n}\left([\bx_{k}^{*}]_{j}e^{-\tau[\matr{A}^{*}\by_{k+1}^{*}]_{j}}\right)^{1/(1+\lambda\tau)}}
\end{alignedat}
\end{equation}
for $i \in \{1,\dots,m\}$ and $j \in \{1,\dots,n\}$. All parameter calculations and updates can be performed in $O(mn)$ operations. According to Prop.~\ref{prop:acc_sq_algII_var} and the optimality conditions~\eqref{eq:matrix-games_optcond}, we have the strong limits
\[
\lim_{k \to +\infty}\bx_{k} = \bx_{s},\quad \lim_{k \to +\infty} \by_{k}^{*} = \by_{s}^{*},
\]
\[
 \lim_{k \to +\infty} -[\matr{A}^{*}\by_{s}^{*}]_{j} = \lambda(1 + \log([\bx_{s}]_{j})), \quad \text{and} \quad \lim_{k \to +\infty} [\by_{k}^{*}]_{i}= \frac{e^{([\matr{A}\bx_{s}]_{i}/\lambda)}}{\sum_{i = 1}^{m}e^{([\matr{A}\bx_{s}]_{i}/\lambda)}}.
\]

%% file: 6_Numerical_experiments.tex
\section{Numerical experiments}\label{sec:numerics}
This section presents some simulations to compare the running times of the accelerated nonlinear PDHG methods proposed in Section~\ref{sec:applications} to other commonly-used first-order optimization methods. These methods include the accelerated linear PDHG method~\cite{chambolle2011first,chambolle2016ergodic} for
both the $\ell_{1}$-constrained logistic regression problems and entropy-regularized matrix games and the forward-backward splitting method~\cite{beck2009fast,chambolle2016introduction} for the $\ell_{1}$-constrained logistic regression problem. The accelerated linear PDHG and forward-backward splitting methods for these examples are described below and were implemented in MATLAB. All numerical experiments were performed on a single core Intel(R) Core(TM) i7-10750H CPU @ 2.60 GHz.

\subsection{\texorpdfstring{$\ell_{1}$}--constrained logistic regression}
\subsubsection{Data generation and optimization methods} \label{subsec:num-setup}
We consider the setting where the $m$ vectors of features $(\bu_{1},\dots,\bu_{m})$ are independent and the true solution is sparse. Specifically, we draw $m$ independent samples $(\bu_{1},\dots,\bu_{m})$ from a $d$-dimensional Gaussian distribution with zero mean and unit variance. Letting $\bv \in \Rn$ denote the true solution to be estimated, we set 1\% of the coefficients of $\bv$ to be equal to $10$ and the other coefficients to be zero. Finally, letting $\bxi$ denote $n$-dimensional Gaussian distribution with zero mean and unit variance, we define the response model as
\[
[\bb]_{i} = 
\begin{dcases}
+1 &\; \text{if} \; \left\langle[\bu]_{i},\bv\right\rangle + [\boldsymbol{\xi}]_{i} \geqslant 0, \\
-1 &\; \text{otherwise}.
\end{dcases}
\]
This setting allows us to process dense, large-scale data sets with sparsity structure. We choose the number of samples to be smaller than then number of features, with $m = 10000$, and $d = 10,000$, $25,000$, $50,000$, $75,000$, $100,000$, $125,000$ and $150,000$. We set the tuning parameter to be $\lambda = 100$.

We perform simulations using the accelerated nonlinear PDHG method~\eqref{eq:l1-logreg_explicit}, the accelerated linear PDHG method~\eqref{eq:intro_l1-logreg_alg} described in the introduction, and the forward-backward splitting method as applied to problem~\eqref{eq:l1-logreg}. The initial values, parameters and numerical criteria for convergence of each method are described below.

\paragraph{Accelerated nonlinear PDHG method~\eqref{eq:l1-logreg_explicit}.}
We set $[\by_{0}^{*}]_{i} = 1/2m$ for each $i \in \{1,\dots,m\}$, we set $[\bx_{-1}]_{j} = [\bx_{0}]_{j} = 1/n$ for each $j \in \{1,\dots,n\}$, and we set $\tau_{0} = 2m/\norm{\matr{A}}_{1,2}^{2}$, $\sigma_{0} = 1/2m$ and $\theta_{0} = 0$. We compute the time required for convergence in the dual variable $\by_{k}^{*}$ and also the time required for convergence in the average dual variable $\bY^{*}$ as defined in Prop.~\ref{prop:acc_sq_algIII}. The iterations were stopped once $\normtwo{\by^{*}_{k+1} - \by^{*}_{k}} \leqslant 10^{-4}\normtwo{\by^{*}_{k+1}}$ and $\normtwo{\bY^{*}_{K+1} - \bY^{*}_{K}} \leqslant 10^{-4}\normtwo{\bY^{*}_{K+1}}$.

\paragraph{Accelerated PDHG method~\eqref{eq:intro_l1-logreg_alg}.}
We set $[\by_{0}^{*}]_{i} = 1/2m$ for each $i \in \{1,\dots,m\}$, we set $[\bv_{-1}]_{j} = [\bv_{0}]_{j} = 1/d$ for each $j \in \{1,\dots,n\}$, and we set $\tau_{0} = 4m/2\norm{\matr{A}}_{2,2}^{2}$, $\sigma_{0} = 1/2m$ and $\theta_{0} = 0$. We evaluate the update in $\bw_{k+1}$ using the forward-backward splitting method and we evaluate the update $\bv_{k+1}$ using the $\ell_{1}$-ball projection algorithm described in~\citet{condat2016fast}. We compute the time required for convergence in the dual variable $\by_{k}^{*}$ and also the time required for convergence in the average dual variable $\bY^{*}$ as defined in Prop.~\ref{prop:acc_sq_algIII}. The iterations were stopped once $\normtwo{\by^{*}_{k+1} - \by^{*}_{k}} \leqslant 10^{-4}\normtwo{\by^{*}_{k+1}}$ and $\normtwo{\bY^{*}_{K+1} - \bY^{*}_{K}} \leqslant 10^{-4}\normtwo{\bY^{*}_{K+1}}$.

\paragraph{Forward-backward splitting method.} 
We compute the iterates
\[
\begin{alignedat}{1}
\bw_{k} &= \bv_{k} + \beta_{k}(\bv_{k} - \bv_{k-1}),\\
\bv_{k+1} &= \argmin_{\normone{\bv} \leqslant \lambda} \left\{\bv - \left(\bw_{k} - \tau\matr{B}^{*}/(m + me^{-\matr{B}\bw_{k}})\right)\right\}, \\
t_{k+1} &= \frac{1 + \sqrt{(1 + 4t_{k}^{2}}}{2} \quad \text{and} \quad \beta_{k+1} = \frac{(t_{k}-1)}{t_{k+1}},\\
\end{alignedat}
\]
where $\tau = 4m/\norm{\matr{B}}_{2,2}^{2}$, $q = \lambda\tau/(1+\lambda\tau)$, $[\bv_{-1}]_{j} = [\bv_{0}]_{j} = 1/d$ for $j \in \{1,\dots,d\}$ and $t_{0} = \beta_{0} = 0$. We evaluate the update $\bv_{k+1}$ using the $\ell_{1}$-ball projection algorithm described in~\citet{condat2016fast}. We compute the time required for convergence in the variable $\bv_{k}$. The iterations were stopped once $\normone{\bv_{k+1} - \bv_{k}} \leqslant 10^{-4}\normone{\bv_{k+1}}$.

\subsubsection{Numerical results}
Table~\ref{tab:numres1} shows the time results for the forward-backward splitting, linear PDHG and nonlinear PDHG methods. For the linear and nonlinear PDHG methods, we also show the time results for convergence with the regular and ergodic sequences as described before. We observe that the nonlinear PDHG method is considerably faster than both the forward-backward splitting and linear PDHG methods; the nonlinear PDHG method achieves a speedup of about 4 to 6.

\begin{table}[ht]
    \centering
    \begin{tabular}{|c|c|c|c|c|c|c|c|}
    \hline
          \multicolumn{1}{|c|}{ } & \multicolumn{7}{c|}{Number of features $n$}\\\cline{2-8}
          
          \multicolumn{1}{|c|}{} & 10 000 & 25 000 & 50 000 & 75 000 & 100 000 & 125 000 & 150 000\\\hline
          
          \multicolumn{1}{|c|}{Optimization methods} & \multicolumn{7}{c|}{Timings (s)}\\\hline
          
          Forward-backward splitting & 40.55 & 114.82 & 269.25 & 437.52 & 725.81 & 839.24 & 1281.77\\
          Linear PDHG (Regular) & 40.56 & 111.60 & 254.41 & 408.79 & 670.57 & 739.37 & 1122.36\\
          Linear PDHG (Ergodic) & 46.96 & 126.51 & 284.25 & 447.06 & 717.33 & 810.46 & 1180.35\\
          Nonlinear PDHG (Regular) & 9.72 & 26.23 & 58.60 & 87.67 & 112.95 & 177.50 & 203.10\\
          Nonlinear PDHG (Ergodic) & 13.52 & 32.47 & 64.71 & 93.97 & 125.65 & 190.25 & 193.36\\
         
         \hline
    \end{tabular}
    \caption{Time results (in seconds) for solving the $\ell_{1}$-restricted logistic regression problem~\eqref{eq:l1-logreg} with the forward-backward and linear PDHG methods and time results for solving the equivalent problem~\eqref{eq:l1-logreg_primal} with the nonlinear PDHG method.}
    \label{tab:numres1}
\end{table}

\subsection{Entropy regularized zero-sum matrix games}
\subsubsection{Data generation and optimization methods}
Following the methodology described in~\cite[Section 2.3]{cen2021fast}, we generate each entry of the payoff matrix $\matr{A}$ from the uniform distribution on $[-1,1]$ and we set $\lambda = 0.1$. Here, we set $m = n$, with $n = 10,000$, $15,000$, $20,000$, $25,000$, $30,000$, $35,000$ and $40,000$.

We perform simulations using the accelerated nonlinear PDHG method~\eqref{eq:alg-matrix-games}, the accelerated linear PDHG method, and the Predictive Update (PU) and Optimistic Multiplicative Weights Update (OMWU) methods from~\citet{cen2021fast}. The initial values, parameters and numerical criteria for convergence of each method are described below.

\paragraph{Accelerated nonlinear PDHG method~\eqref{eq:alg-matrix-games}.}
We generate the entries of the initial vectors $\by^{*}_{0}$ and $[\bx_{-1}]_{j} = [\bx_{0}]_{j}$ for $j \in \{1,\dots,n\}$ uniformly at random in $(0,1/m)$ and $(0,1/n)$, respectively, and normalized their entries so that $\sum_{i=1}^{m} [\by^{*}_{0}]_{i} = 1$ and $\sum_{j=1}^{n} [\bx_{0}]_{j} = 1$. For the parameters, we set
\[
\theta = 1 - \frac{\lambda^2}{2\norm{\matr{A}}_{1,\infty}^2}\left(\sqrt{1 + \frac{4\norm{\matr{A}}_{1,\infty}^2}{\lambda^2}} - 1\right) \quad \text{and} \quad \tau = \sigma = \frac{1-\theta}{\lambda\theta}.
\]
We compute the time required for convergence in the dual variable $\by_{k}^{*}$ and also the time required for convergence in the average dual variable $\bY^{*}$ as defined in Prop.~\ref{prop:acc_sq_algII_var}. The iterations were stopped once $\normtwo{\by^{*}_{k+1} - \by^{*}_{k}} \leqslant 10^{-4}\normtwo{\by^{*}_{k+1}}$ and $\normtwo{\bY^{*}_{K+1} - \bY^{*}_{K}} \leqslant 10^{-4}\normtwo{\bY^{*}_{K+1}}$.

\paragraph{Accelerated linear PDHG method.} We compute the iterates
\[
\begin{alignedat}{1}
\by^{*}_{k+1} &= \argmax_{\by^{*} \in \Delta_{m}}\left\{-\lambda \mathcal{H}_{m}(\by^{*}) + \left\langle \by^{*},\matr{A}(\bx_{k} + \theta(\bx_{k}-\bx_{k-1})) \right\rangle - \frac{1}{2\sigma}\normsq{\by^{*}-\by^{*}_{k}}\right\}, \\
\bx_{k+1} &= \argmin_{\bx \in \Delta_{n}}\left\{\lambda\mathcal{H}_{n}(\bx) + \left\langle \by_{k+1}^{*}, \matr{A}\bx \right\rangle + \frac{1}{2\tau}\normsq{\bx-\bx_{k}}\right\}.
\end{alignedat}
\]
To compute these iterates, we use Moreau's identity~\cite{moreau1965proximite} to express them as follows:
\[
\begin{alignedat}{1}
\bv_{k} &= \by_{k}^{*} + \sigma\matr{A}(\bx_{k} + \theta[\bx_{k}-\bx_{k-1}]),\\
\by_{k+1}^{*} &= \bv_{k} - \argmin_{\bz \in \Rm}\left\{\frac{1}{2}\normsq{\bz-\bv_{k}} + \lambda\sigma \log\left(\sum_{i=1}^{m} e^{[\bz]_{i}/\lambda\sigma}\right)\right\}, \\
\bw_{k} &= \bx_{k} - \tau\matr{A}^{*}\by_{k+1}^{*}, \\
\bx_{k+1} &= \bw_{k} - \argmin_{\bz \in \Rn}\left\{\frac{1}{2}\normsq{\bz - \bw_{k}} + \lambda\tau\log\left(\sum_{i=1}^{m} e^{[\bz]_{i}/\lambda\tau}\right) \right\}.
\end{alignedat}
\]
We use the forward-backward splitting method~\cite[Algorithm 5]{chambolle2016introduction} to compute the second and fourth line. Here we use the same initial values as for the accelerated nonlinear PDHG method~\eqref{eq:alg-matrix-games}, and for the parameters we set
\[
\theta = 1 - \frac{\lambda^2}{2\norm{\matr{A}}_{2,2}^2}\left(\sqrt{1 + \frac{4\norm{\matr{A}}_{2,2}^2}{\lambda^2}} - 1\right) \quad \text{and} \quad \tau = \sigma = \frac{1-\theta}{\lambda\theta}.
\]
We compute the time required for convergence in the dual variable $\by_{k}^{*}$ and also the time required for convergence in the average dual variable $\bY^{*}$ as defined in Prop.~\ref{prop:acc_sq_algII_var}. The iterations were stopped once $\normtwo{\by^{*}_{k+1} - \by^{*}_{k}} \leqslant 10^{-4}\normtwo{\by^{*}_{k+1}}$ and $\normtwo{\bY^{*}_{K+1} - \bY^{*}_{K}} \leqslant 10^{-4}\normtwo{\bY^{*}_{K+1}}$.

\paragraph{Predictive Update and Optimistic Multiplicative Weights Update methods.}
For the PU and OMWU, we use Algorithms 1 and 2 as described in~\cite{cen2021fast} with the learning rates
\[
\eta_{\text{PU}} = \frac{1}{2 + \norm{\matr{A}}_{1,\infty}} \quad \text{and} \quad \eta_{\text{OMWU}} = \min{\left\{\frac{1}{2 + 2\norm{\matr{A}}_{1,\infty}},\frac{1}{4\norm{\matr{A}}_{1,\infty}}\right\}}.
\]

\subsubsection*{Numerical results}
Table~\ref{tab:numres2} shows the time results for the PU, OMWU, and linear and nonlinear PDHG methods. For the linear and nonlinear PDHG methods, we also show the time results for convergence with the regular and ergodic sequences as described before. We observe that the nonlinear PDHG method is considerable faster than both the linear PDHG method and the state-of-the-art methods PU and OMWU for solving the entropy regularized zero-sum matrix game~\eqref{eq:matrix-games_saddle}; the nonlinear PDHG method achieves a speedup of 5 to 11 compared to linear PDHG method and a speedup of 3 to 5 compared to the state-of-the-art methods PU and OMWU.

\begin{table}[ht]
    \centering
    \begin{tabular}{|c|c|c|c|c|c|c|c|}
    \hline
          \multicolumn{1}{|c|}{ } & \multicolumn{7}{c|}{Numbers $m = n$}\\\cline{2-8}
          
          \multicolumn{1}{|c|}{} & 10 000 & 15 000 & 20 000 & 25 000 & 30 000 & 35 000 & 40 000\\\hline
          
          \multicolumn{1}{|c|}{Optimization methods} & \multicolumn{7}{c|}{Timings (s)}\\\hline
          
          PU & 34.16 & 55.01 & 88.48 & 137.14 & 197.66 & 289.49 & 353.61\\
          OMWU & 44.89 & 81.67 & 142.53 & 221.23 & 318.65 & 493.73 & 568.94\\
          Linear PDHG (Regular) & 42.47 & 100.65 & 218.26 & 366.52 & 601.11 & 938.18 & 1094.25\\
          Linear PDHG (Ergodic) & 44.43 & 104.61 & 225.52 & 379.71 & 608.78 & 949.36 & 1114.24\\
          Nonlinear PDHG (Regular) & 8.56 & 15.88 & 24.68 & 38.40 & 55.13 & 82.52 & 103.50\\
          Nonlinear PDHG (Ergodic) & 12.02 & 19.05 & 31.45 & 48.70 & 70.27 & 105.26 & 129.28\\
         
         \hline
    \end{tabular}
    \caption{Time results (in seconds) for solving the entropy regularized zero-sum matrix game~\eqref{eq:matrix-games_saddle} with the PU, OMWU, and linear and nonlinear PDHG methods.}
    \label{tab:numres2}
\end{table}

%% file: 7_discussion.tex
\section{Discussion} \label{sec:discussion}
We have introduced new accelerated nonlinear primal-dual hybrid gradient (PDHG) optimization methods to solve efficiently large-scale convex optimization problems with saddle-point structure. We proved rigorous convergence results, including results for strongly convex or smooth problems posed on infinite-dimensional reflexive Banach spaces. The new accelerated nonlinear PDHG methods are particularly useful to solve problems involving a logistic regression model or problems defined on the unit simplex or both. Indeed, for these problems, one may choose to use a Bregman divergence defined in terms of the average negative sum of binary entropy terms or the relative entropy to arrive at a straightforward and efficient optimization method. To illustrate this, we presented practical implementations of accelerated nonlinear PDHG methods for $\ell_{1}$-constrained logistic regression and zero-sum matrix games with entropy regularization. Numerical experiments showed that the nonlinear PDHG methods are considerably faster than competing methods.

The new nonlinear PDHG methods are advantageous because they can achieve an optimal convergence rate with stepsize parameters that are simple and efficient to compute. They can be typically computed on the order of $O(mn)$ operations where $m$ and $n$ denote the dimensions to the dual and primal problems at hand. In contrast, most first-order optimization methods, including the linear PDHG method, require on the order of $O(\min{(m^2n,mn^2)})$ operations to compute all the parameters required to achieve an optimal convergence rate. This gain in efficiency can be considerable: in our numerical experiments for $\ell_{1}$-constrained logistic regression and zero-sum matrix games with entropy regularization we were able to get a speedup of 5 to 10 compared to other competing optimization methods.

We expect the accelerated nonlinear PDHG methods described in this work to provide efficient methods for solving large-scale supervised machine learning. In particular, these applications to strongly convex and smooth problems defined on the unit simplex, such as $\nu$-support vector machines with squared loss, maximum entropy estimation problems, and boosting and structured prediction problems in machine learning, will be pursued in future work. It would be interesting to extend the accelerated nonlinear PDHG methods described here to the stochastic case for problems that are separable in the dual variable, and to the non-convex case to deal with large-scale non-convex problems, such as those arising in deep learning. These extensions will be pursued in future work as well.

%% file: app_definitions.tex

This appendix lists some basic definitions and facts from convex and functional analysis that are used in this work. It is not meant to be exhaustive, and we refer the reader to~\cite{brezis2010functional,ekeland1999convex,folland2013real,hiriart2013convexI,hiriart2013convexII,rockafellar1970convex} for comprehensive references.

In all definitions and facts below, the spaces $\xcal$ and $\ycal$ denote two real reflexive Banach spaces endowed with norms $\normx{\cdot}$ and $\normy{\cdot}$. The interior of a non-empty subset $C$ of $\xcal$ or $\ycal$ is denoted by $\interior{~C}$. The set of proper, convex and lower semicontinuous functions defined on $\xcal$ and $\ycal$ are denoted by $\gmxcal$ and $\gmycal$. The dual spaces of all continuous linear functionals defined on $\xcal$ and $\ycal$ are denoted by $\xcal^*$ and $\ycal^*$. For a linear functional $\bx^* \in \xcal^*$ and an element $\bx \in \xcal$, the bilinear form $\left\langle \bx^*,\bx \right\rangle$ gives the value of $\bx^*$ at $\bx$. Likewise, for a linear functional $\by^* \in \ycal^*$ and an element $\by \in \ycal$, the bilinear form $\left\langle \by^*,\by \right\rangle$ gives the value of $\by^*$ at $\by$. The norms associated to $\xcal^*$ and $\ycal^*$ are defined as
\[
\normxs{\bx^*} = \sup_{\normx{\bx} = 1} \left\langle \bx^*,\bx \right\rangle \quad \mathrm{and} \quad \normys{\by^*} = \sup_{\normy{\by} = 1} \left\langle \by^*,\by \right\rangle.
\]
Let $\matr{A}\colon \xcal \to \ycal$ denote a bounded linear operator. Its corresponding adjoint operator $\matr{A}^*\colon \ycal^* \to \xcal^*$ is defined so as to satisfy
\begin{equation*}
    \left\langle \matr{A}^*\by^*,\bx \right\rangle = \left\langle \by^*,\matr{A}\bx \right\rangle
\end{equation*}
for every $\bx \in \xcal$ and $\bys \in \ycal^*$. The operator norm associated to $\matr{A}$ is defined as
\begin{equation*}
    \normop{\matr{A}} = \sup_{\normx{\bx} = 1} \normy{\matr{A}\bx} = \normop{\matr{A}^*} = \sup_{\normys{\by^*} = 1}\normxs{\matr{A}^*\by^*}.
\end{equation*}
These definitions imply the Cauchy--Schwartz inequality
\[
\left|\left\langle \by^*,\matr{A}\bx \right\rangle\right| \leqslant \normop{\matr{A}}\normx{\bx}\normys{\by^*}.
\]

\subsection*{Definitions}
\begin{defn}[Convex sets] \label{def:convex_sets}
A subset $C\subset\xcal$ is convex if for every pair $(\bx,\bx') \in C \times C$ and every scalar $\lambda \in (0,1)$, the point $\lambda\bx+(1-\lambda)\bx'$ is contained in $C$.
\end{defn}
\begin{defn}[Proper functions] \label{def:prop_f}
A function $f$ defined on $\xcal$ is proper if its domain 
\[
\dom f = \left\{ \bx \in \xcal : f(\bx) < +\infty\right\}
\]
is non-empty and $f(\bx) > -\infty$ for every $\bx \in \dom f$.
\end{defn}
\begin{defn}[Lower semicontinuous functions] \label{def:lsc_f}
A proper function $f\colon\xcal\to\R\cup\{+\infty\}$ is lower semicontinuous at a point $\bx\in\xcal$ if for every sequence $\left\{ \bx_k\right\} _{k=1}^{+\infty}$ in $\xcal$ that converges to $\bx$, 
\[
\liminf_{k\to+\infty}f(\bx_k)\geqslant f(\bx).
\] 
We say that $f$ is lower semicontinuous if it is lower semicontinuous at every $\bx \in \dom f$.
\end{defn}
\begin{defn}[Convex functions] \label{def:convex} 
A proper function $f\colon\xcal\to\R\cup\{+\infty\}$ is convex if its domain $\dom f$ is convex and if for every pair $(\bx,\bx') \in \dom f \times \dom f$ and every scalar $\lambda\in[0,1]$,
\[
f(\lambda\bx+(1-\lambda)\bx')\leqslant\lambda f(\bx)+(1-\lambda)f(\bx').
\]
It is \textit{strictly convex} if the inequality above is strict whenever $\bx\neq\bx'$ and $\lambda \in (0,1)$, and it is \textit{$m$-strongly convex} (with $m>0$) if for every pair $(\bx,\bx') \in \dom f \times \dom f$ and every scalar $\lambda\in[0,1]$. 
\[
f(\lambda\bx+(1-\lambda)\bx')\leqslant\lambda f(\bx)+(1-\lambda)f(\bx')-\frac{m}{2}\lambda(1-\lambda)\normx{\bx-\bx'}^2.
\]
\end{defn}
\begin{defn}[Coercive functions]\label{def:coercive}
A proper function $f\colon\xcal\to\R\cup\{+\infty\}$ is coercive if for every sequence $\{\bx_k\}_{k=1}^{+\infty}$ in $\xcal$ such that $\lim_{k \to +\infty} \normx{\bx_k} = +\infty$,
\[
\lim_{k \to +\infty} f(\bx_k) = +\infty.
\]
A proper function $f\colon\xcal\to\R\cup\{+\infty\}$ is supercoercive if for every sequence $\{\bx_k\}_{k=1}^{+\infty}$ in $\xcal$ such that $\lim_{k \to +\infty} \normx{\bx_k} = +\infty$,
\[
\lim_{k \to +\infty} \frac{f(\bx_k)}{\normx{\bx_k}} = +\infty.
\]
\end{defn}
\begin{defn}[Weak convergence]\label{def:weak_conv}
A sequence $\{\bx_k\}_{k=1}^{+\infty}$ of points in $\xcal$ converges weakly to $\bx \in \xcal$ if for every linear functional $\bx^* \in \gmxcal$,
\[
\lim_{k \to +\infty} \left\langle \bx^*,\bx_k \right\rangle = \left\langle \bx^*,\bx \right\rangle.
\]
\end{defn}
\begin{defn}[Differentiability] \label{def:differentiability}
A proper function $f\colon \xcal \to \R\cup\{+\infty\}$ with $\interior{(\dom f)} \neq \varnothing$ is differentiable at a point $\bx \in \interior{(\dom f)}$ if there exists a linear functional $\bx^* \in \xcal^*$ such that for every $\bx' \in \xcal$,
\[
\lim_{\substack{\lambda\to0\\ \lambda>0}} \frac{f(\bx + \lambda\bx') - f(\bx)}{\lambda} = \left\langle \bxs,\bx'\right\rangle.
\]
This linear functional, when it exists, is unique. It is called the gradient of $f$ at $\bx$ and is denoted by $\nabla f(\bx)$.
\end{defn}
\begin{defn}[Subdifferentiability and subgradients]\label{def:subgrad}
A function $f \in \gmxcal$ is subdifferentiable at a point $\bx \in \xcal$ if there exists a linear functional $\bx^* \in \xcal^*$ such that for every $\bx' \in \dom f$,
\begin{equation}\label{eq:conv_subdiff_char}
f(\bx') - f(\bx) - \left\langle\bx^*,\bx'-\bx \right\rangle \geqslant 0.
\end{equation}
In this case, $\bx^*$ is called a subgradient of the function $f$ at $\bx$. The set of subgradients at $\bx \in \xcal$ is called the subdifferential of $f$ at $\bx$, and it is denoted by $\partial f(\bx)$. The set of points $\bx \in \dom f$ at which the subdifferential $\partial f(\bx)$ is non-empty is denoted by $\dom \partial f$.

If $f$ is strictly convex, then for $\bx \neq \bx'$ the inequality in~\eqref{eq:conv_subdiff_char} is strict. If $f$ is $m$-strongly convex and $\bx \in \dom \partial f$, then for every $\bx' \in \dom f$ the subgradients $\bx^* \in \partial f(\bx)$ satisfy the inequality
\begin{equation}\label{eq:sc_subdiff_char}
f(\bx') - f(\bx) - \left\langle \bx^*,\bx'-\bx\right\rangle \geqslant \frac{m}{2}\normx{\bx-\bx'}^2.
\end{equation}
\end{defn}
\begin{defn}[Convex conjugates] \label{def:conv_conj}
Let $f \in \gmxcal$. The convex conjugate $f^{*}\colon\xcal^*\to\mathbb{R}\cup\{+\infty\}$ of $f$ is defined by
\begin{equation*}
f^{*}(\bx^*)=\sup_{\bx \in \dom f}\left\{\left\langle \bx^*,\bx\right\rangle -f(\bx)\right\}. 
\end{equation*}
By definition, the function $f^*$ is in $\gmxcals$~\cite[Definition 4.1]{ekeland1999convex}.
\end{defn}
\begin{defn}[Saddle points]\label{def:saddle}
Let $\mathcal{L}\colon \xcal \times \ycal^* \to \R\cup\{+\infty\}$ be a proper function. A pair of points $(\bx_{s},\by_{s}^{*}) \in \xcal \times \ycal^*$ is a saddle point of $\mathcal{L}$ if for every $\bx \in \xcal$ and $\bys \in \ycal^*$,
\[
\mathcal{L}(\bx_{s},\bys) \leqslant \mathcal{L}(\bx_{s},\by_{s}^{*}) \leqslant \mathcal{L}(\bx,\by_{s}^{*}). 
\]
\end{defn}
\begin{defn}[Essential smoothness]\label{def:ess_smooth}
A function $f \in \gmxcal$ is essentially smooth if $\dom \partial f  \neq \varnothing$, $\dom \partial f  = \interior{(\dom f)}$, $f$ is differentiable on $\interior{(\dom f)}$, and $\normx{\nabla f(\bx_k)} \to +\infty$ for every sequence $\{\bx_k\}_{k=1}^{+\infty}$ in $\interior{(\dom f)}$ converging to some boundary point of $\dom f$.
\end{defn}
\begin{defn}[Essential strict convexity]\label{def:ess_conv}
A function $f \in \gmxcal$ is essentially strictly convex if $f$ is strictly convex on every convex subset of $\dom \partial f$ and the subdifferential mapping $\partial f^{*}$ is locally bounded on its domain.
\end{defn}
\begin{defn}[Bregman divergences]\label{def:bregman}
Let $\phi \in \gmxcal$ with $\interior{(\dom \phi)} \neq \varnothing$. The Bregman divergence of the function $\phi$ is the function $D_{\phi}\colon \xcal\times \interior{(\dom \phi)} \to [0,+\infty]$ defined as
\[
D_{\phi}(\bx,\bx') = \phi(\bx) - \phi(\bx') - \max_{\bx^* \in \partial \phi(\bx')}\left\{\left\langle \bx^*,\bx-\bx'\right\rangle\right\}.
\]
Note that Bregman divergences are sometimes defined differently in the convex analysis literature. Here, we use the definition of~\citet[Definition 7.1 and Lemma 7.3(i)]{bauschke2001essential}.
\end{defn}
\begin{defn}[Bregman proximity operators]\label{def:d_prox_operator}
Let $f,\phi \in \gmxcal$ with $\interior{(\dom \phi)} \neq \varnothing$ and let $t > 0$. The Bregman $D_{\phi}$-proximal operator $\mathrm{prox}_{\left(tf,D_{\phi}\right)}(\cdot)$ is a set-valued mapping defined for every $\bx' \in \interior{(\dom \phi)}$ as
\begin{equation}\label{eq:d_prox_operator}
    \mathrm{prox}_{\left(tf,D_{\phi}\right)}(\bx') = \left\{\bxh \in \dom f \cap \dom \phi : tf(\bxh) + D_{\phi}(\bxh,\bx') = \inf_{\bx \in \xcal} \left\{tf(\bx) + D_{\phi}(\bx,\bx')\right\} < +\infty  \right\}. 
\end{equation}
\end{defn}

\subsection*{Facts}
\begin{fact}\label{fact:cauchy_ineq_cont}
Let $\alpha > 0$ and let $\matr{A}\colon \xcal \to \ycal$ be a bounded linear operator. For every $(\bx,\bys), (\bx',\bys') \in \xcal \times \ycal^{*}$, the following auxiliary inequality holds:
\begin{equation}\label{eq:cauchy_ineq_cont}
    \left|\left\langle \bys-\bys', \matr{A}(\bx-\bx') \right\rangle\right| \leqslant \normop{\matr{A}}\left(\frac{\alpha}{2}\normx{\bx-\bx'}^2 + \frac{1}{2\alpha}\normys{\bys - \bys'}^2\right).
\end{equation}
\begin{proof}
From the Cauchy--Schwartz inequality,
\begin{equation*}
\begin{alignedat}{1}
\left |\left\langle \bys-\bys' ,\matr{A}(\bx-\bx'\right\rangle \right | &\leqslant \normop{\matr{A}}\normx{\bx-\bx'}\normys{\bys - \bys'} \\
&= \normop{\matr{A}}\left(\frac{\alpha}{2}\normx{\bx-\bx'}^2 + \frac{1}{2\alpha}\normys{\bys - \bys'}^2\right) \\
&\qquad - \normop{\matr{A}}\left(\sqrt{\frac{\alpha}{2}}\normx{\bx-\bx'} - \sqrt{\frac{1}{2\alpha}}\normys{\bys - \bys'}\right)^2 \\
&\leqslant \normop{\matr{A}}\left(\frac{\alpha}{2}\normx{\bx-\bx'}^2 + \frac{1}{2\alpha}\normys{\bys - \bys'}^2\right).
\end{alignedat}
\end{equation*}
\end{proof}
\end{fact}
\begin{fact}[Weighted averages of a convergent sequence]\label{fact:weighted_averages}
Let $\{\bx_{k}\}_{k=1}^{+\infty} \subset \xcal$ be a sequence converging strongly to some $\bx \in \xcal$, let $\{\lambda_{k}\}_{k=1}^{+\infty} \subset (0,+\infty)$ be a divergent sequence, i.e., $\sum_{k=1}^{+\infty} \lambda_{k} = +\infty$, and set $T_k = \sum_{j=1}^{k} \lambda_{j}$. Then
\[
\lim_{k \to +\infty} \normx{\frac{1}{T_k}\left(\sum_{j=1}^{k} \lambda_{j}\bx_{j}\right) - \bx} = 0.
\]
\end{fact}
\begin{proof}
Fix $\epsilon > 0$. Then there exists some $K_1 \in \mathbb{N}$ such that for every $k \geqslant K_1$, we have $\normx{\bx_{k} - \bx} < \epsilon/2$. Now, let $k \geqslant K_1$, take the difference between the weighted average $\frac{1}{T_k}\sum_{j=1}^{k} \lambda_{j}\bx_{j}$ and $\bx$, take the norm, use the triangle inequality and rearrange to get
\[
\begin{alignedat}{1}
\normx{\frac{1}{T_k}\sum_{j=1}^{k} \lambda_{j}\bx_{j} - \bx} &= \normx{\frac{1}{T_k}\sum_{j=1}^{k} \lambda_{j}\left(\bx_j - \bx\right)} \\
&\leqslant \frac{1}{T_k}\sum_{j=1}^{k} \lambda_{j}\normx{\bx_j - \bx} \\
&= \frac{1}{T_k}\sum_{j=1}^{K_1-1} \lambda_{j}\normx{\bx_j - \bx} + \frac{1}{T_k}\sum_{j=K_1}^{k} \lambda_{j}\normx{\bx_j - \bx} \\
&\leqslant \frac{1}{T_k}\sum_{j=1}^{K_1-1} \lambda_{j}\normx{\bx_j - \bx} + \frac{1}{T_k}\sum_{j=K_1}^{k} \lambda_{j} \frac{\epsilon}{2} \\
&\leqslant \frac{1}{T_k}\sum_{j=1}^{K_1-1} \lambda_{j}\normx{\bx_j - \bx} + \frac{\epsilon}{2}.
\end{alignedat}
\]
The first term on the right hand side of the last line depends on $k$ only through the term $T_k$. By assumption, $T_k \to +\infty$ as $k \to +\infty$, and therefore there exists some $K_2 \in \mathbb{N}$ such that for $k \geqslant K_{2}$,
\[
\frac{1}{T_k}\sum_{j=1}^{K_1-1} \lambda_{j}\normx{\bx_j - \bx} < \frac{\epsilon}{2}.
\] 
Taking $k \geqslant \max{(K_1,K_2)}$, we find
\[
\normx{\frac{1}{T_k}\sum_{j=1}^{k} \lambda_{j}\bx_{j} - \bx} < \epsilon.
\]
As $\epsilon$ was arbitrary positive number, we can take $\epsilon \to 0$ and obtain the desired result.
\end{proof}
%
%
\begin{fact}[Supercoercivity]\label{fact:supercoercivity}
Let $f \in \gmxcal$ and suppose that $f$ is supercoercive. Then for every $\alpha >0$, there exists $\beta \in \R$ such that $f(\bx) \geqslant \alpha\normx{\bx} + \beta$ for every $\bx \in \xcal$. In particular, a supercoercive function is always bounded from below.
\begin{proof}
See~\cite[Lemma 3.2]{bauschke2001essential} for a proof.
\end{proof}
\end{fact}
\begin{fact}[Bounded sequences and weak convergence]\label{fact:weak_conv_bdd_seq}
Let $\{\bx_k\}_{k=1}^{+\infty}$ be a bounded sequence in $\xcal$. Then this sequence has a subsequence $\{\bx_{k_l}\}_{l=1}^{+\infty}$ that converges weakly to some element in $\xcal$.
\end{fact}
\begin{proof}
See~\cite[Theorem 3.18]{brezis2010functional}.
\end{proof}
\begin{fact}[The primal problem and its dual problem]\label{fact:primal_dual_prob}
Let $g \in \gmxcal$, let $h \in \gmycal$, and let $\matr{A}\colon \xcal \to \ycal$ be a bounded linear operator. Assume the primal (minimization) problem
\begin{equation}\label{eq:fact:primal_prob}
\inf_{\bx \in \xcal} \left\{g(\bx) + h(\matr{A}\bx)\right\}
\end{equation}
has at least one solution and assume there exists $\bx \in \xcal$ such that $h$ is continuous at $\matr{A}\bx$. Then the dual (maximization) problem
\begin{equation}\label{eq:fact:dual_prob}
\sup_{\bys \in \ycal^*} \left\{-g^*(-\matr{A}^*\bys) - h^*(\bys)\right\}
\end{equation}
is finite and has at least one solution. Moreover, if $(\bx_{s},\by_{s}^{*})$ denotes a pair of solutions to the primal and dual problem then $(\bx_{s},\by_{s}^{*})$
satisfies the following optimality conditions
\[
-\matr{A}^*\by_{s}^{*} \in \partial g(\bx_{s}) \quad \mathrm{and} \quad \by_{s}^{*} \in \partial h(\matr{A}\bx_{s}).
\]
\begin{proof}
See~\cite[Theorem 4.1, Theorem 4.2, Equations (4.24)-(4.25)]{ekeland1999convex} for a proof. (Beware, in~\cite{ekeland1999convex} the notation used for the solution $\by_{s}^{*}$ is flipped by a minus sign.)
\end{proof}
\end{fact}
\begin{fact}[Convex-concave saddle point problems]\label{fact:saddle_prob}
Let $g \in \gmxcal$, let $h \in \gmycal$, let $\matr{A}\colon \xcal \to \ycal$ be a bounded linear operator, define the function $\mathcal{L}\colon \xcal \times \ycal^* \to \R\cup\{+\infty\}$ as
\[
\mathcal{L}(\bx,\bys) = g(\bx) + \left\langle \bys,\matr{A}\bx \right\rangle - h^*(\bys).
\]
Then the pair of points $(\bx_{s},\by_{s}^{*}) \in \xcal \times \ycal^*$ is a saddle point of $\mathcal{L}$ if and only if $\bx_{s}$ is a solution of the primal problem~\eqref{eq:fact:primal_prob} and $\by_{s}^{*}$ is a solution of the dual problem~\eqref{eq:fact:dual_prob}.
\begin{proof}
See~\cite[Proposition 3.1, page 57]{ekeland1999convex}.
\end{proof}
\end{fact}
\begin{fact}[Properties of Bregman divergences]\label{fact:breg_div_props}
Let $\phi \in \gmxcal$ with $\interior{(\dom \phi)} \neq \varnothing$, let $\bx \in \dom \phi$, and let $\bx',\bxh \in \interior{(\dom \phi)}$. Assume that $\phi$ is differentiable on $\interior{(\dom \phi)}$. Then the Bregman divergence $D_{\phi}$ of $\phi$ satisfies the following properties:
\begin{itemize}
    \item[(i)] The Bregman divergence can be written as $D_{\phi}(\bx,\bx') = \phi(\bx) - \phi(\bx') - \left\langle \nabla \phi(\bx'),\bx-\bx'\right\rangle$.
    
    \item[(ii)] The Bregman divergence $D_{\phi}$ satisfies the three-point identity
    \begin{equation}\label{eq:fact:three_point}
    D_{\phi}(\bx,\bx') = D_{\phi}(\bxh,\bx') + D_{\phi}(\bx,\bxh) + \left\langle \nabla\phi(\bx') - \nabla\phi(\bxh), \bxh - \bx \right\rangle.
    \end{equation}
    
    \item[(iii)] If $\phi$ is essentially strictly convex, then $D_{\phi}(\bx,\bx') = 0$ if and only if $\bx = \bx'$.
    
    \item[(iv)] If $\phi$ is essentially strictly convex, then the function $\bx \mapsto D_{\phi}(\bx,\bx')$ is coercive for every $\bx' \in \interior{(\dom \phi)}$.
    
    \item[(v)] If $\phi$ is supercoercive, then the function $\bx' \mapsto D_{\phi}(\bx,\bx')$ is coercive for every $\bx \in \interior{(\dom \phi)}$.
    
    \item[(vi)] If $\{\bx_k\}_{k = 1}^{+\infty}$ is a sequence in $\interior{(\dom \phi)}$ converging to a point $\bx \in \interior{(\dom \phi)}$, then 
    \begin{equation*}
    \lim_{k \to +\infty} D_{\phi}(\bx,\bx_k) = 0.
    \end{equation*}
    
    \item[(vii)] Assume that $\phi$ is essentially smooth. If $\{\bx_k\}_{k = 1}^{+\infty}$ is a sequence in $\interior{(\dom \phi)}$ converging to a point $\bx_{c} \in \interior{(\dom \phi)}$, then 
    \[
    \lim_{k \to +\infty} D_{\phi}(\bx,\bx_k) = D_{\phi}(\bx,\bx_{c}).
    \]
    
    \item[(viii)] If $\phi$ is $m$-strongly convex with respect to $\normx{\cdot}$, then
    \[
    D_{\phi}(\bx,\bx') \geqslant \frac{m}{2}\normx{\bx-\bx'}^2.
    \]
\end{itemize}
\begin{proof}
See~\cite[Lemma 7.3]{bauschke2001essential} for the proof of (i) and (iii)-(vi). Statement (ii) follows from (i) and a straightforward calculation. Statement (vii) follows from (i) and the continuity of both $\phi$ and $\nabla \phi$ over $\interior{(\dom \phi)}$. Statement (viii) follows from (i) and inequality~\eqref{eq:sc_subdiff_char}.
\end{proof}
\end{fact}
\begin{fact}[Properties of Bregman proximity operators]~\label{fact:breg_prox_prop}
Let $f,\phi \in \gmxcal$ be two functions such that $\dom f \cap \interior{(\dom \phi)} \neq \varnothing$, let $t > 0$, and assume $\phi$ is essentially smooth and essentially strictly convex. In addition, assume that either $f$ is bounded from below or $\phi$ is supercoercive. Then the following properties hold:
\begin{itemize}
    \item[(i)] The proximal operator $\bx' \mapsto \mathrm{prox}_{\left(tf,D_{\phi}\right)}(\bx')$ defined in~\eqref{eq:d_prox_operator} is single-valued on its domain $\interior{(\dom \phi)}$. That is, for every $\bx' \in \interior{(\dom \phi)}$,
    \begin{equation*}
    \mathrm{prox}_{\left(tf,D_{\phi}\right)}(\bx') = \argmin_{\bx \in \xcal}\left\{tf(\bx) + D_{\phi}(\bx,\bx')\right\}.
    \end{equation*}
    Moreover, $\mathrm{prox}_{\left(tf,D_{\phi}\right)}(\bx') \in \dom \partial f \cap \interior{(\dom \phi)}$.
    
    \item [(ii)] For every $\bx \in \dom f$ and $\bx' \in \interior{(\dom \phi)}$, the proximal point $\mathrm{prox}_{\left(tf,D_{\phi}\right)}(\bx')$ satisfies the characterization
    \begin{equation}\label{eq:char_prox_diff}
    f(\bx) - f(\mathrm{prox}_{\left(tf,D_{\phi}\right)}(\bx')) - \frac{1}{t}\left\langle \nabla\phi(\mathrm{prox}_{\left(tf,D_{\phi}\right)}(\bx')) - \nabla\phi(\bx'), \mathrm{prox}_{\left(tf,D_{\phi}\right)}(\bx') - \bx \right\rangle \geqslant 0.
    \end{equation}
    If, in addition, there exists $\gamma_f > 0$ such that the function $\bx \mapsto f(\bx) - \gamma_f \phi(\bx)$ is convex, then this characterization can be strengthened to
    \begin{equation}\label{eq:char_prox_diff_sc}
    \begin{alignedat}{1}
    f(\bx) - f(\mathrm{prox}_{\left(tf,D_{\phi}\right)}(\bx')) &- \frac{1}{t}\left\langle \nabla\phi(\mathrm{prox}_{\left(tf,D_{\phi}\right)}(\bx')) - \nabla\phi(\bx'), \mathrm{prox}_{\left(tf,D_{\phi}\right)}(\bx') - \bx \right\rangle \\
    &\qquad \geqslant \gamma_f D_{\phi}(\bx,\mathrm{prox}_{\left(tf,D_{\phi}\right)}(\bx')).
    \end{alignedat}
    \end{equation}
    
    
    \item[(iii)] For every $\bx \in \dom f$ and $\bx' \in \interior{(\dom \phi)}$,
    \begin{equation}\label{eq:char_prox_mapping}
    \begin{alignedat}{1}
        f(\bx) + \frac{1}{t}D_{\phi}(\bx,\bx') &\geqslant f(\mathrm{prox}_{\left(tf,D_{\phi}\right)}(\bx')) + \frac{1}{t} D_f(\mathrm{prox}_{\left(tf,D_{\phi}\right)}(\bx'),\bx') \\ 
        &\quad + \frac{1}{t}D_f(\bx,\mathrm{prox}_{\left(tf,D_{\phi}\right)}(\bx')).
        \end{alignedat}
    \end{equation}
    If, in addition, there exists $\gamma_f > 0$ such that the function $\bx \mapsto f(\bx) - \gamma_f\phi(\bx)$ is convex, then~\eqref{eq:char_prox_mapping} can be strengthened to
    \begin{equation}\label{eq:char_prox_mapping_sc}
    \begin{alignedat}{1}
        f(\bx) + \frac{1}{t}D_{\phi}(\bx,\bx') &\geqslant f(\mathrm{prox}_{\left(tf,D_{\phi}\right)}(\bx')) + \frac{1}{t}D_f(\mathrm{prox}_{\left(tf,D_{\phi}\right)}(\bx'),\bx') \\
        &\quad + \left(\frac{1}{t} + \gamma_f\right)D_f(\bx,\mathrm{prox}_{\left(tf,D_{\phi}\right)}(\bx')).
    \end{alignedat}
    \end{equation}
\end{itemize}
\begin{proof}
See~\cite[Proposition 3.21-3.23, Theorem 3.24, Corollary 3.25]{bauschke2003bregman} for the proof of statements (i). Statement (ii) follows directly from~\cite[Proposition 2.2, page 38]{ekeland1999convex}. To prove inequality~\eqref{eq:char_prox_mapping} in (iii), use the characterization~\eqref{eq:char_prox_diff} to write
\[
\begin{alignedat}{1}
f(\bx) + \frac{1}{t}D_{\phi}(\bx,\bx') &\geqslant f(\mathrm{prox}_{\left(tf,D_{\phi}\right)}(\bx')) + \frac{1}{t}D_{\phi}(\bx,\bx') \\
&\quad - \frac{1}{t}\left\langle \nabla\phi(\bx') - \nabla\phi(\mathrm{prox}_{\left(tf,D_{\phi}\right)}(\bx')), \mathrm{prox}_{\left(tf,D_{\phi}\right)}(\bx') - \bx \right\rangle.
\end{alignedat}
\]
Then use the three-point identity~\eqref{eq:fact:three_point} with $\bxh = \mathrm{prox}_{\left(tf,D_{\phi}\right)}(\bx'))$ to obtain~\eqref{eq:char_prox_mapping}. The proof of inequality~\eqref{eq:char_prox_mapping_sc} in (iii) is nearly identical, with the exception that the characterization~\eqref{eq:char_prox_diff_sc} is used in place of~\eqref{eq:char_prox_diff}.
\end{proof}
\end{fact}

%% file: app_proof_basic.tex
We divide the proof into two parts, first proving that the output $(\bxh,\bysh)$ is contained in the set $\dom \partial g \times \dom \partial h^*$ and then deriving the descent rule~\eqref{eq:descent_rule}.

\bigbreak
\noindent
\textbf{Part 1.} Consider the functions \[
f = g + \left\langle \matr{A}^*\byst, \cdot \right\rangle \quad \mathrm{and} \quad \phi = \phix.
\]
By assumptions (A1)-(A3), the function $\phi$ is essentially smooth and essentially strictly convex, we have that $\dom{f} \,\cap\, \interior{(\dom \phi)} \neq \varnothing$, and at least one of $g$ and $\phi$ is supercoercive. If $g$ is supercoercive, then an elementary calculation shows that the function $f$ is also supercoercive, and therefore bounded from below by Fact~\ref{fact:supercoercivity}. Hence we are guaranteed that $f$ is bounded from below or $\phi$ is supercoercive. In either case, we can invoke Fact~\ref{fact:breg_prox_prop}(i) to conclude that the minimization problem
\[
\argmin_{\bx \in \xcal} \left\{g(\bx) + \left\langle \matr{A}^*\byst, \bx\right\rangle + \frac{1}{\tau}D_{\phix}(\bx,\bxb)\right\}
\]
has a unique solution $\bxh$ that is contained in the set $\dom \partial g \,\cap\, \dom\partial \phix$.  A similar argument using assumptions (A1)-(A2) and (A4) shows that the minimization problem
\[
\argmin_{\bys \in \ycal^*} \left\{h^{*}(\bys) - \left\langle \bys, \matr{A}\bxt\right\rangle + \frac{1}{\sigma}D_{\phiys}(\bys,\bysb)\right\}
\]
has a unique solution $\bysh$ that is contained in the set $\dom \partial h^* \cap \dom \partial \phiys$.

\bigbreak
\noindent
\textbf{Part 2.} To derive the descent rule~\eqref{eq:descent_rule}, we apply inequality~\eqref{eq:char_prox_mapping} to each minimization problem in the iteration scheme~\eqref{eq:iteration_PG}. Note that inequality~\eqref{eq:char_prox_mapping} can be used here because we showed in Part 1 that the conditions of Fact~\ref{fact:breg_prox_prop} are satisfied by each minimization problem in~\eqref{eq:iteration_PG}. 

First, use inequality~\eqref{eq:char_prox_mapping} with the functions
\[
f = g + \left\langle \matr{A}^*\byst, \cdot \right\rangle \quad \mathrm{and} \quad \phi = \phix ,
\]
the parameter $t = \tau$, the element $\bx' = \bxb$, and the proximal point $\mathrm{prox}_{\left(tf,D_{\phi}\right)}(\bx') = \bxh$ to get
\begin{equation*}
g(\bx) + \left\langle \matr{A}^*\byst, \bx \right\rangle + \frac{1}{\tau}D_{\phix}(\bx,\bar{\bx})
\geqslant g(\hat{\bx}) + \left\langle \matr{A}^*\byst,\bxh \right\rangle + \frac{1}{\tau}\left(D_{\phix}(\bxh,\bxb) + D_{\phix}(\bx,\hat{\bx})\right).
\end{equation*}
Rearrange this inequality in terms of the difference $g(\bxh) - g(\bx)$ to get
\begin{equation}\label{eq:lem:ineq2}
g(\bxh) - g(\bx) \leqslant \frac{1}{\tau}\left(D_{\phix}(\bx,\bar{\bx}) - D_{\phix}(\bxh,\bxb) - D_{\phix}(\bx,\bxh)\right) + \left\langle \matr{A}^*\byst, \bx - \bxh \right\rangle.
\end{equation}
A similar application of inequality~\eqref{eq:char_prox_mapping} to the second line of the iteration scheme~\eqref{eq:iteration_PG} gives
\begin{equation}\label{eq:lem:ineq3}
h^*(\bysh) - h^{*}(\bys) \leqslant \frac{1}{\sigma}\left(D_{\phiys}(\bys,\bysb) - D_{\phiys}(\bysh,\bysb) - D_{\phiys}(\bys,\bysh)\right) - \left\langle \bys -\bysh,\matr{A}\bxt \right\rangle.
\end{equation}

Second, add the difference of bilinear forms
\[
\left\langle \bys, \matr{A}\bxh \right\rangle - \left\langle \bysh, \matr{A}\bx\right\rangle
\]
to both sides of inequality~\eqref{eq:lem:ineq3} and rearrange to get
\begin{equation}\label{eq:lem:ineq3_mod}
\begin{alignedat}{1}
\left(\left\langle \bys, \matr{A}\bxh \right\rangle - h^*(\bys)\right) - \left(\left\langle \bysh,\matr{A}\bx \right\rangle - h^*(\bysh)\right) &= \frac{1}{\sigma}\left(D_{\phiys}(\bys,\bysb) - D_{\phiys}(\bysh,\bysb) - D_{\phiys}(\bys,\bysh)\right) \\
&\quad + \left\langle \bys, \matr{A}\bxh \right\rangle - \left\langle \bysh, \matr{A}\bx\right\rangle - \left\langle \bys -\bysh,\matr{A}\bxt \right\rangle \\
&= \frac{1}{\sigma}\left(D_{\phiys}(\bys,\bysb) - D_{\phiys}(\bysh,\bysb) - D_{\phiys}(\bys,\bysh)\right) \\
&\quad + \left\langle \bys, \matr{A}(\bxh - \bxt) \right\rangle - \left\langle\bysh, \matr{A}(\bx - \bxt) \right\rangle.
\end{alignedat}
\end{equation}

Next, we combine inequalities~\eqref{eq:lem:ineq2} and~\eqref{eq:lem:ineq3_mod}. Add the left hand sides of inequalities~\eqref{eq:lem:ineq2} and~\eqref{eq:lem:ineq3_mod} and use the definition~\eqref{eq:lagrangian} of the Lagrangian function $\mathcal{L}(\cdot,\cdot)$ to get
\begin{equation} \label{eq:lem:ineq4_lhs}
\left(g(\bxh) + \left\langle \bys, \matr{A}\bxh \right\rangle - h^*(\bys)\right) - \left(g(\bx) + \left\langle \bysh,\matr{A}\bx \right\rangle - h^*(\bysh)\right) = \mathcal{L}(\bxh,\bys) - \mathcal{L}(\bx,\bysh).
\end{equation}
Thanks to~\eqref{eq:lem:ineq4_lhs}, the sum of inequalities~\eqref{eq:lem:ineq2} and~\eqref{eq:lem:ineq3_mod} give
\begin{equation}\label{eq:lem:ineq4}
\begin{alignedat}{1}
\mathcal{L}(\bxh,\bys) - \mathcal{L}(\bx,\bysh) &\leqslant \frac{1}{\tau}\left(D_{\phix}(\bx,\bar{\bx}) - D_{\phix}(\bxh,\bxb) - D_{\phix}(\bx,\bxh)\right) \\
&\quad + \frac{1}{\sigma}\left(D_{\phiys}(\bys,\bysb) - D_{\phiys}(\bysh,\bysb) - D_{\phiys}(\bys,\bysh)\right) \\
&\quad + \left\langle \matr{A}^*\byst, \bx - \bxh \right\rangle + \left\langle \bys, \matr{A}(\bxh - \bxt) \right\rangle - \left\langle\bysh, \matr{A}(\bx - \bxt) \right\rangle.
\end{alignedat}
\end{equation}
Now, write
\[
\begin{alignedat}{1}
\left\langle \matr{A}^*\byst, \bx - \bxh \right\rangle &= \left\langle \byst, \matr{A}(\bx - \bxh) \right\rangle \\
&= \left\langle \byst, \matr{A}(\bx  - \bxh + \bxt - \bxt) \right\rangle \\
&= \left\langle \byst, \matr{A}(\bx  - \bxt)\right\rangle - \left\langle \byst, \matr{A}(\bxh  - \bxt)\right\rangle
\end{alignedat}
\]
and use this to express the last line on the right hand side of inequality~\eqref{eq:lem:ineq4} as
\begin{equation}\label{eq:lem:ineq5}
\begin{alignedat}{1}
\left\langle \matr{A}^*\byst, \bx - \bxh \right\rangle + \left\langle \bys, \matr{A}(\bxh - \bxt) \right\rangle - \left\langle\bysh, \matr{A}(\bx - \bxt) \right\rangle &= \left\langle \byst, \matr{A}(\bx  - \bxt)\right\rangle - \left\langle \byst, \matr{A}(\bxh  - \bxt)\right\rangle \\
&\quad + \left\langle \bys, \matr{A}(\bxh - \bxt) \right\rangle - \left\langle\bysh, \matr{A}(\bx - \bxt) \right\rangle \\
&= \left\langle \byst - \bysh, \matr{A}(\bx  - \bxt)\right\rangle - \left\langle \bys - \byst, \matr{A}(\bxt - \bxh) \right\rangle.
\end{alignedat}
\end{equation}
Finally, combine inequalities~\eqref{eq:lem:ineq4} and~\eqref{eq:lem:ineq5} to find
\begin{equation*}
\begin{alignedat}{1}
\mathcal{L}(\bxh,\bys) - \mathcal{L}(\bx,\bysh) &\leqslant \frac{1}{\tau}\left(D_{\phix}(\bx,\bar{\bx}) - D_{\phix}(\bxh,\bxb) - D_{\phix}(\bx,\bxh)\right) \\
&\quad + \frac{1}{\sigma}\left(D_{\phiys}(\bys,\bysb) - D_{\phiys}(\bysh,\bysb) - D_{\phiys}(\bys,\bysh)\right) \\
&\quad + \left\langle \byst - \bysh, \matr{A}(\bx  - \bxt)\right\rangle - \left\langle \bys - \byst, \matr{A}(\bxt - \bxh) \right\rangle.
\end{alignedat}
\end{equation*}
which is the desired result.

%% file: app_proof_propI.tex
We divide the proofs into four parts, first deriving an auxiliary result, and then proving in turn the descent rule~\eqref{eq:descent_rule_s_I} (Proposition~\ref{prop:basic_algI}(a)), the estimate~\eqref{eq:basic_rateI} and the global bound~\eqref{eq:bound_seq} (Proposition~\ref{prop:basic_algI}(b)), and the convergence properties of the nonlinear PDHG method~\eqref{eq:basic_pdhg_alg} (Proposition~\ref{prop:basic_algI}(c)).

\bigbreak
\noindent
\textbf{Part 1.} We first show that for every $(\bx,\by^{*}) \in \dom g \times \dom h^*$ and nonnegative integer $k$, the quantity $\Delta_{k}(\bx,\by^{*})$ satisfies the bounds
\begin{equation}\label{eq:innerprod_param_ineq2}
\begin{alignedat}{1}
    0 \leqslant (1 - \sqrt{\tau\sigma}\normop{\matr{A}}) &\left(\frac{1}{\tau}D_{\phix}(\bx,\bx_k) + \frac{1}{\sigma}D_{\phiys}(\bys,\by_{k}^{*})\right) \\
    &\qquad \leqslant \Delta_{k}(\bx,\by^{*}) \leqslant (1 + \sqrt{\tau\sigma}\normop{\matr{A}})\left(\frac{1}{\tau}D_{\phix}(\bx,\bx_{k}) + \frac{1}{\sigma}D_{\phiys}(\bys,\by_{k}^{*})\right).
    \end{alignedat}
\end{equation}
To derive this, use fact~\ref{fact:cauchy_ineq_cont} with the choice of $\alpha = \sqrt{\sigma/\tau}$ and $(\bx',\bys') = (\bx_k,\by_{k}^{*})$ to get
\begin{equation}\label{eq:innerprod_param_ineq}
\begin{alignedat}{1}
\left |\left\langle \bys-\by_{k}^{*} ,\matr{A}(\bx-\bx_k\right\rangle \right | &\leqslant \normop{\matr{A}}\left(\frac{\sqrt{\sigma}}{2\sqrt{\tau}}\normx{\bx-\bx_k}^2 + \frac{\sqrt{\tau}}{2\sqrt{\sigma}}\normys{\bys - \by_{k}^{*}}^2\right) \\
&= \sqrt{\tau\sigma}\normop{\matr{A}}\left(\frac{1}{2\tau}\normx{\bx-\bx_k}^2 + \frac{1}{2\sigma}\normys{\bys - \by_{k}^{*}}^2\right) \\
&\leqslant \sqrt{\tau\sigma}\normop{\matr{A}}\left(\frac{1}{\tau}D_{\phix}(\bx,\bx_k) + \frac{1}{\sigma}D_{\phiys}(\bys,\by_{k}^{*})\right),
\end{alignedat}
\end{equation}
where in the last line we used assumption (A5) and Fact~\ref{fact:breg_div_props}(viii) with $ m= 1$. Inequality~\eqref{eq:innerprod_param_ineq2} then follows from equation~\eqref{eq:basic_d_symbol} and inequalities~\eqref{eq:innerprod_param_ineq} and~\eqref{eq:suff_param_ineq}.

\bigbreak
\noindent
\textbf{Part 2.} Let $(\bx,\bys) \in \dom g \times \dom h^*$. By assumption (A1)-(A4), Lemma~\ref{lem:descent_rule} holds, and we can apply the descent rule~\eqref{eq:descent_rule} to the $(k+1)^{\mathrm{th}}$ iterate given by~\eqref{eq:basic_pdhg_alg} with initial points $(\bxb,\bysb) = (\bx_{k},\by_{k}^{*})$ and intermediate points $(\bxt,\byst) = (2\bx_{k+1}-\bx_{k},\by_{k}^{*})$ to get
\begin{equation}\label{eq:prop1:ineq1}
\begin{alignedat}{1}
\mathcal{L}(\bx_{k+1},\bys) - \mathcal{L}(\bx,\by_{k+1}^{*}) &\leqslant \frac{1}{\tau}\left(D_{\phix}(\bx,\bx_{k}) - D_{\phix}(\bx_{k+1},\bx_{k}) - D_{\phix}(\bx,\bx_{k+1})\right)\\
&\qquad + \frac{1}{\sigma}\left(D_{\phiys}(\bys,\by_{k}^{*}) - D_{\phiys}(\by_{k+1}^{*},\by_{k}^{*}) - D_{\phiys}(\bys,\by_{k+1}^{*}) \right) \\
&\qquad + \left\langle \by_{k}^{*} - \by_{k+1}^{*}, \matr{A}(\bx  - (2\bx_{k+1}-\bx_{k}))\right\rangle \\
&\qquad - \left\langle \bys - \by_{k}^{*}, \matr{A}((2\bx_{k+1}-\bx_{k}) - \bx_{k+1}) \right\rangle.
\end{alignedat}    
\end{equation}
To proceed, we want to rewrite the last two lines of~\eqref{eq:prop1:ineq1} to simplify the analysis. First, write the penultimate line on the right hand side of~\eqref{eq:prop1:ineq1} as
\begin{equation}\label{eq:prop1:calc1}
\begin{alignedat}{1}
\left\langle \by_{k}^{*} - \by_{k+1}^{*}, \matr{A}(\bx  - (2\bx_{k+1}-\bx_{k}))\right\rangle &= \left\langle \by_{k}^{*} - \by_{k+1}^{*}, \matr{A}(\bx  - \bx_{k+1}) - \matr{A}(\bx_{k+1}  - \bx_{k})\right\rangle\\
&= \left\langle \by_{k}^{*} - \by_{k+1}^{*}, \matr{A}(\bx  - \bx_{k+1})\right\rangle \\
&\qquad + \left\langle \by_{k+1}^{*} - \by_{k}^{*}, \matr{A}(\bx_{k+1} - \bx_{k})\right\rangle \\
&= \left\langle \bys- \by_{k+1}^{*}, \matr{A}(\bx  - \bx_{k+1})\right\rangle\\
&\qquad - \left\langle \bys- \by_{k}^{*}, \matr{A}(\bx  - \bx_{k+1})\right\rangle\\
&\qquad + \left\langle \by_{k+1}^{*} - \by_{k}^{*}, \matr{A}(\bx_{k+1} - \bx_{k})\right\rangle \\
&= \left\langle \bys- \by_{k+1}^{*}, \matr{A}(\bx  - \bx_{k+1})\right\rangle \\
&\qquad + \left\langle \bys - \by_{k}^{*}, \matr{A}(\bx_{k+1}  - \bx_{k})\right\rangle \\
&\qquad - \left\langle \bys - \by_{k}^{*}, \matr{A}(\bx  - \bx_{k})\right\rangle\\
&\qquad + \left\langle \by_{k+1}^{*} - \by_{k}^{*}, \matr{A}(\bx_{k+1} - \bx_{k})\right\rangle,
\end{alignedat}
\end{equation}
The bilinear form on the last line of~\eqref{eq:prop1:ineq1} simplifies to
\begin{equation}\label{eq:prop1:calc2}
\left\langle \bys - \by_{k}^{*}, \matr{A}((2\bx_{k+1}-\bx_{k}) - \bx_{k+1}) \right\rangle = \left\langle \bys - \by_{k}^{*}, \matr{A}(\bx_{k+1} - \bx_{k}) \right\rangle.
\end{equation}
Combine equations~\eqref{eq:basic_d_symbol},~\eqref{eq:prop1:calc1} and~\eqref{eq:prop1:calc2} together to write inequality~\eqref{eq:prop1:ineq1} as
\begin{equation*}
\mathcal{L}(\bx_{k+1},\bys) - \mathcal{L}(\bx,\by_{k+1}^{*}) \leqslant \Delta_{k}(\bx,\by^{*}) - \Delta_{k+1}(\bx,\by^{*}) - \Delta_{k+1}(\bx_{k},\by_{k}^{*}).
\end{equation*}
Thanks to inequality~\eqref{eq:innerprod_param_ineq2}, 
\[
- \Delta_{k+1}(\bx_{k},\by_{k}^{*}) \leqslant 0,
\]
and hence
\begin{equation*}
\mathcal{L}(\bx_{k+1},\bys) - \mathcal{L}(\bx,\by_{k+1}^{*}) \leqslant \Delta_{k}(\bx,\by^{*}) - \Delta_{k+1}(\bx,\by^{*}).
\end{equation*}
This proves the descent rule~\eqref{eq:descent_rule_s_I}. 

\bigbreak
\noindent
\textbf{Part 3.} Sum inequality~\eqref{eq:descent_rule_s_I} from $k=1$ to $K$ on both sides to obtain
\begin{equation}\label{eq:prop1:ineq4}
\begin{alignedat}{1}
\sum_{k=1}^{K} \mathcal{L}(\bx_{k},\bys) - \mathcal{L}(\bx,\by_{k}^{*}) &\leqslant \Delta_{0}(\bx,\by^{*}) - \Delta_{K}(\bx,\by^{*}).
\end{alignedat}
\end{equation}
Use the averages
\[
\bX_{K} = \frac{1}{K}\sum_{k=1}^{K} \bx_{k} \quad \mathrm{and} \quad \bY_{K}^{*} = \frac{1}{K}\sum_{k=1}^{K} \by_{k}^{*},
\]
the convexity and concavity in the first and second arguments of the Lagrangian~\eqref{eq:lagrangian}, respectively, and inequality~\eqref{eq:prop1:ineq4} to bound the difference of Lagrangians $\mathcal{L}(\bX_{K}, \bys) - \mathcal{L}(\bx,\bY_{K}^{*})$ as follows:
\begin{equation}\label{eq:prop1:ave_ineq}
\begin{alignedat}{1}
\mathcal{L}(\bX_{K}, \bys) - \mathcal{L}(\bx,\bY_{K}^{*}) &\leqslant \frac{1}{K}\sum_{k=1}^{K}\left(\mathcal{L}(\bx_{k},\bys) - \mathcal{L}(\bx,\by_{k}^{*})\right) \\
&\leqslant \frac{1}{K}\left(\Delta_{0}(\bx,\by^{*}) - \Delta_{K}(\bx,\by^{*})\right).
\end{alignedat}
\end{equation}
Finally, use the lower and upper bounds~\eqref{eq:innerprod_param_ineq2} in~\eqref{eq:prop1:ave_ineq} to get 
\[
\begin{alignedat}{1}
\mathcal{L}(\bX_{K}, \bys) - \mathcal{L}(\bx,\bY_{K}^{*}) &\leqslant \frac{1+\sqrt{\tau\sigma}\normop{\matr{A}}}{K}\left(\frac{1}{\tau}D_{\phix}(\bx,\bx_{0}) + \frac{1}{\sigma}D_{\phiys}(\bys,\by_{0}^{*})\right) \\
&\qquad - \frac{1-\sqrt{\tau\sigma}\normop{\matr{A}}}{K}\left(\frac{1}{\tau}D_{\phix}(\bx,\bx_{K}) + \frac{1}{\sigma}D_{\phiys}(\bys,\by_{K}^{*})\right).
\end{alignedat}
\]
This proves the estimate~\eqref{eq:basic_rateI}.

Now, let $(\bx,\bys) = (\bx_{s},\by_{s}^{*})$ in estimate~\eqref{eq:basic_rateI} and use the saddle point property 
\[
\mathcal{L}(\bX_{K}, \by_{s}^{*}) - \mathcal{L}(\bx_{s},\bY_{K}^{*}) \geqslant 0
\]
and rearrange to get
\[
\begin{alignedat}{1}
(1-\sqrt{\tau\sigma}\normop{\matr{A}}) &\left(\frac{1}{\tau}D_{\phix}(\bx_{s},\bx_{K}) + \frac{1}{\sigma}D_{\phiys}(\by_{s}^{*},\by_{K}^{*})\right) \\
&\qquad \leqslant (1+\sqrt{\tau\sigma}\normop{\matr{A}})\left(\frac{1}{\tau}D_{\phix}(\bx_{s},\bx_{0}) + \frac{1}{\sigma}D_{\phiys}(\by_{s}^{*},\by_{0}^{*})\right).
\end{alignedat}
\]
Since $\tau\sigma\normop{\matr{A}}^2 < 1$, the number $(1-\sqrt{\tau\sigma}\normop{\matr{A}})$ is strictly positive and we can divide both sides of the previous inequality by $(1-\sqrt{\tau\sigma}\normop{\matr{A}})$ to get
\[
\frac{1}{\tau}D_{\phix}(\bx_{s},\bx_{K}) + \frac{1}{\sigma}D_{\phiys}(\by_{s}^{*},\by_{K}^{*}) \leqslant \frac{1+\sqrt{\tau\sigma}\normop{\matr{A}}}{1-\sqrt{\tau\sigma}\normop{\matr{A}}}\left(\frac{1}{\tau}D_{\phix}(\bx_{s},\bx_{0}) + \frac{1}{\sigma}D_{\phiys}(\by_{s}^{*},\by_{0}^{*})\right).
\]
This proves inequality~\eqref{eq:bound_seq}.

\bigbreak
\noindent
\textbf{Part 4.} First, note that the global bound~\eqref{eq:bound_seq} implies that the sequence of iterates $\{(\bx_{k},\by_{k}^{*})\}_{k=1}^{+\infty}$ is bounded. It follows immediately from the definitions of the averages $\bX_{K}$ and $\bY_{K}^{*}$ that the sequence of averages $\{(\bX_{K},\bY_{K}^{*})\}_{K=1}^{+\infty}$ is also bounded.

From Fact~\ref{fact:weak_conv_bdd_seq}, there is a subsequence $\{(\bX_{K_l},\bY_{K_l}^{*})\}_{l=1}^{+\infty}$ that converges weakly to some point $(\bX,\bYs) \in \xcal \times \ycals$. We claim that $(\bX,\bYs)$ is a saddle point of the Lagrangian~\eqref{eq:lagrangian}. To see this, use inequality~\eqref{eq:basic_rateI} with $(\bx,\bys) \in \dom g \times \dom h^*$ and take the infimum limit $l \to +\infty$ to get
\[
\liminf_{l \to +\infty} \left[\mathcal{L}(\bX_{K_l},\bys) - \mathcal{L}(\bx,\bY_{K_l}^{*})\right] \leqslant \liminf_{l \to +\infty} \left[ \frac{1+\sqrt{\tau\sigma}\normop{\matr{A}}}{K_l}\left( \frac{1}{\tau} D_{\phix}(\bx,\bx_{0}) + \frac{1}{\sigma} D_{\phiys}(\bys,\by_{0}^{*})\right) \right] = 0.
\]
The lower semicontinuity property of the functions $g$ and $h^*$ implies
\[
0 \geqslant \liminf_{l \to +\infty} \left(\mathcal{L}(\bX_{K_l},\bys) - \mathcal{L}(\bx,\bY_{K_l}^{*})\right) \geqslant \mathcal{L}(\bX,\bys) - \mathcal{L}(\bx,\bYs),
\]
from which we find
\[
\mathcal{L}(\bX,\bys) \leqslant \mathcal{L}(\bx,\bYs).
\]
As the pair of points $(\bx,\bys) \in \xcal\times\ycals$ was arbitrary, we conclude that $(\bX,\bYs)$ is a saddle point of the Lagrangian~\eqref{eq:lagrangian}.

Assume now that the spaces $\xcal$ and $\ycals$ are finite-dimensional. Since the sequence of iterates $\{(\bx_{k},\by_{k}^{*})\}_{k=1}^{+\infty}$ is bounded, by Fact~\ref{fact:weak_conv_bdd_seq} there is a subsequence that converges strongly to some point $(\bx_c,\by_{c}^{*})$. Note that $(\bx_c,\by_{c}^{*})$ is a fixed point of the nonlinear PDHG method~\eqref{eq:basic_pdhg_alg}, and therefore we can invoke Lemma~\ref{lem:descent_rule} to conclude that $(\bx_c,\by_{c}^{*}) \in \dom \partial g \times \dom \partial h^{*}$. 

We claim that $(\bx_c,\by_{c}^{*})$ is a saddle point of the Lagrangian~\eqref{eq:lagrangian}. To see this, consider the descent rule~\eqref{eq:descent_rule_s_I} with arbitrary $(\bx,\by^{*})$ and the subsequence $\{(\bx_{k_l},\by_{k_l}^{*})\}_{l=1}^{+\infty}$:
\begin{equation*}
\mathcal{L}(\bx_{k_l},\by^{*}) - \mathcal{L}(\bx,\by_{k_l}^{*}) \leqslant \Delta_{k_l}(\bx,\bys) - \Delta_{k_{l+1}}(\bx,\bys).
\end{equation*}
By the strong convergence of the subsequence $\{(\bx_{k_l},\by_{k_l}^{*})\}_{l=1}^{+\infty}$ to $(\bx_c,\by_{c}^{*})$ and Fact~\ref{fact:breg_div_props}(vii), we have the limits 
\[
\lim_{l\to +\infty} D_{\phix}(\bx,\bx_{k_{l}}) = D_{\phix}(\bx,\bx_c) \quad \mathrm{and} \quad \lim_{l\to +\infty} D_{\phiys}(\by^{*},\by_{k_{l}}^{*}) = D_{\phiys}(\by^{*},\by_{c}^{*}).
\]
Hence
\[
\lim_{l \to +\infty} \Delta_{k_l}(\bx,\by^{*}) = \frac{1}{\tau}D_{\phix}(\bx,\bx_c) + \frac{1}{\sigma} D_{\phiys}(\by^{*},\by_{c}^{*}) - \left\langle \bys - \by_{c}^{*}, \matr{A}(\bx  - \bx_{c})\right\rangle,
\]
and we conclude, from the completeness property of the real numbers, that
\[
\liminf_{l \to +\infty} \Delta_{k_l}(\bx,\by^{*}) - \Delta_{k_{l+1}}(\bx,\by^{*}) = 0.
\]
We therefore deduce the infimum limit
\[
\liminf_{l \to +\infty} \mathcal{L}(\bx_{k_l},\by^{*}) - \mathcal{L}(\bx,\by_{k_l}^{*}) \leqslant 0.
\]
The lower semicontinuity property of the functions $g$ and $h^*$ implies
\[
0 \geqslant \liminf_{l \to +\infty} \left(\mathcal{L}(\bx_{k_l},\bys) - \mathcal{L}(\bx,\by_{k_l}^{*})\right) \geqslant \mathcal{L}(\bx_c,\bys) - \mathcal{L}(\bx,\by_{c}^{*}),
\]
from which we find
\[
\mathcal{L}(\bx,\by_{c}^{*}) \leqslant \mathcal{L}(\bx_c,\by^{*}).
\]
As the pair of points $(\bx,\bys) \in \xcal\times\ycals$ was arbitrary, we conclude that $(\bx_c,\by_{c}^{*})$ is a saddle point of the Lagrangian~\eqref{eq:lagrangian}. 

It remains to prove that the sequence of iterates $\left\{(\bx_{k},\by_{k})\right\}_{k=1}^{+\infty}$ converges strongly to the saddle point $(\bx_c,\by_{c}^{*})$. To do so, consider the descent rule~\eqref{eq:descent_rule_s_I} with the choice of saddle point $(\bx,\by^{*}) = (\bx_c,\by_{c}^{*})$. From the saddle-point property $\mathcal{L}(\bx_{k},\by_{c}^{*}) - \mathcal{L}(\bx_{c},\by_{k}^{*}) \geqslant 0$ and inequality~\eqref{eq:innerprod_param_ineq2}, we have
\[
0 \leqslant \Delta_{k}(\bx_{c},\by_{c}^{*}) \leqslant \Delta_{k-1}(\bx_{c},\by_{c}^{*}).
\]
The sequence of real numbers $\{\Delta_{k}(\bx_{c},\by_{c}^{*})\}_{k=1}^{+\infty}$ is non-increasing in $l$, and as such, it has a limit. By Lemma~\ref{lem:descent_rule}, Fact~\ref{fact:breg_div_props}(vi), and the strong convergence of the subsequence $\{(\bx_{k_l},\by_{k_l}^{*})\}_{l=1}^{+\infty}$ to $(\bx_c,\by_{c}^{*})$, we have
\[
\lim_{l \to +\infty}D_{\phix}(\bx_c,\bx_{k_l}) = 0, \quad \lim_{l \to +\infty}D_{\phiys}(\by_{c}^{*},\by_{k_l}^{*}) = 0, \quad \mathrm{and} \quad \lim_{l \to +\infty} \left\langle \by_{c}^{*}- \by_{k_l}^{*}, \matr{A}(\bx_{c}  - \bx_{k_l})\right\rangle = 0.
\]
We deduce the limit
\[
\lim_{k \to +\infty} \Delta_{k}(\bx_{c},\by_{c}^{*}) = 0.
\]
Now, from this limit and the lower bound~\eqref{eq:innerprod_param_ineq2} with $(\bx,\by) = (\bx_{c},\by_{c}^{*})$, we have
\[
0 \leqslant \lim_{k \to +\infty}(1-\sqrt{\tau\sigma}\normop{\matr{A}})\left(\frac{1}{\tau}D_{\phix}(\bx_c,\bx_{k}) + \frac{1}{\sigma}D_{\phiys}(\by_{c}^{*},\by_{k}^{*})\right) \leqslant \lim_{k \to +\infty} \Delta_{k}(\bx_{c},\by_{c}^{*}) = 0.
\]
Since $(1-\sqrt{\tau\sigma}\normop{\matr{A}}) > 0$ and assumption (A5) holds, we deduce the limits
\[
0 \leqslant \lim_{k \to +\infty} \frac{1}{2}\normx{\bx_{c}-\bx_{k}}^2 \leqslant \lim_{k \to +\infty} D_{\phix}(\bx_c,\bx_{k}) = 0
\]
and
\[
0 \leqslant \lim_{k \to +\infty} \frac{1}{2}\normys{\by_{c}^{*}-\by_{k}^{*}}^2 \leqslant  \lim_{k \to +\infty} D_{\phiys}(\by_{c}^{*},\by_{k}^{*}) = 0.
\]
This proves the strong convergence of the sequence of iterates $\{(\bx_{k},\by_{k}^{*})\}_{k=1}^{+\infty}$ to the saddle point $(\bx_c,\by_{c}^{*})$. Finally, we deduce from Fact~\ref{fact:weighted_averages} that the sequence of averages $\{(\bX_{K},\bY_{K}^{*})\}_{K=1}^{+\infty}$ converges strongly to the same limit $(\bx_c,\by_{c}^{*})$. This concludes the proof.

%% file: app_proof_propacc1.tex
We divide the proof into five parts, first deriving an auxiliary result, and then proving in turn the descent rule~\eqref{eq:descent_rule_s_acc_I} (Proposition~\ref{prop:acc_sq_algI}(a)), the estimate~\eqref{eq:acc_sq__rateI} and global bound~\eqref{eq:bound_seq_acc_I} (Proposition~\ref{prop:acc_sq_algI}(b)), formula~\eqref{eq:acc_exact_t_avg} and the bounds~\eqref{eq:acc_I_avg_bounds} (Proposition~\ref{prop:acc_sq_algI}(c)), and the convergence properties of the accelerated nonlinear PDHG method~\eqref{eq:accI_pdhg_alg} (Proposition~\ref{prop:acc_sq_algI}(d)).

\bigbreak
\noindent
\textbf{Part 1.} We first show that for every $(\bx,\by^{*}) \in \dom g \times \dom h^*$ and $k \in \mathbb{N}$, the quantity $\Delta_{k}(\bx,\by^{*})$ satisfies the lower bound
\begin{equation}\label{eq:bound_delta_below}
\Delta_{k}(\bx,\by^{*}) \geqslant \frac{\gamma_{g}}{1 + \gamma_{g}\tau_{0}}D_{\phix}(\bx,\bx_{k}) + \frac{1}{\sigma_{k}}D_{\phiys}(\by^{*},\by_{k}^{*}).
\end{equation}
To do so, use Fact~\ref{fact:cauchy_ineq_cont} with $\alpha = \theta_{k}/(\tau_{k}\normop{\matr{A}})$ and assumption (A5) to get
\[
\left | \left\langle\by_{k}^{*}-\by_{k-1}^{*},\matr{A}(\bx-\bx_{k}) \right\rangle \right |  \leqslant \frac{\theta_{k}}{\tau_{k}}D_{\phix}(\bx,\bx_{k}) + \frac{\tau_{k}\normop{\matr{A}}^2}{\theta_{k}} D_{\phiys}(\by_{k}^{*},\by_{k-1}^{*}).
\]
By the second and third recurrence relations in~\eqref{eq:acc_sq_exact_iter}, we have the identity
\begin{equation}\label{eq:app_identity_rr}
    \tau_{k}\sigma_{k}\normop{\matr{A}}^2 = 1.
\end{equation}
Use this identity in the previous inequality to find
\begin{equation*}
\left | \left\langle\by_{k}^{*}-\by_{k-1}^{*},\matr{A}(\bx-\bx_{k}) \right\rangle \right |  \leqslant  \frac{\theta_{k}}{\tau_{k}}D_{\phix}(\bx,\bx_{k}) + \frac{1}{\sigma_{k}\theta_{k}} D_{\phiys}(\by_{k}^{*},\by_{k-1}^{*}).
\end{equation*}
Substitute in $\Delta_{k}(\bx,\by^{*})$ to get
\begin{equation*}
\begin{alignedat}{1}
\Delta_{k}(\bx,\by^{*}) &\geqslant \frac{1}{\tau_{k}}D_{\phix}(\bx,\bx_{k}) + \frac{1}{\sigma_{k}}D_{\phiys}(\by^{*},\by_{k}^{*}) + \frac{1}{\sigma_{k}}D_{\phiys}(\by_{k}^{*},\by_{k-1}^{*}) \\
&\qquad  - \theta_{k}\left(\frac{\theta_{k}}{\tau_{k}}D_{\phix}(\bx,\bx_{k}) + \frac{1}{\sigma_{k}\theta_{k}}D_{\phiys}(\by_{k}^{*},\by_{k-1}^{*})\right)\\
&=  \frac{1-\theta_{k}^{2}}{\tau_{k}}D_{\phix}(\bx,\bx_{k}) + \frac{1}{\sigma_{k}}D_{\phiys}(\by^{*},\by_{k}^{*}).
\end{alignedat}
\end{equation*}
The first and second recurrence relations in~\eqref{eq:acc_sq_exact_iter} imply
\[
\frac{1-\theta^{2}_{k}}{\tau_{k}} = \frac{1}{\tau_{k}}\left(1 - \frac{1}{1 + \gamma_{g}\tau_{k-1}}\right) = \frac{1}{\tau_{k}}\left(\frac{\gamma_{g}\tau_{k-1}}{1 + \gamma_{g}\tau_{k-1}}\right) = \frac{1}{\theta_{k}}\left(\frac{\gamma_{g}}{1 + \gamma_{g}\tau_{k-1}}\right) = \frac{\gamma_{g}}{\theta_{k} + \gamma_{g}\tau_{k}} \geqslant \frac{\gamma_{g}}{1 + \gamma_{g}\tau_{0}}.
\]
Hence
\[
\Delta_{k}(\bx,\by^{*}) \geqslant \frac{\gamma_{g}}{1 + \gamma_{g}\tau_{0}}D_{\phix}(\bx,\bx_{k}) + \frac{1}{\sigma_{k}}D_{\phiys}(\by^{*},\by_{k}^{*}),
\]
which proves the auxiliary result~\eqref{eq:bound_delta_below}.

\bigbreak
\noindent
\textbf{Part 2.} Let $(\bx,\by^{*}) \in \dom g \times \dom h^*$. By assumption (A1)-(A6), Lemma~\ref{lem:descent_rule} holds, and we can apply the improved descent rule~\eqref{eq:descent_rule_sc} to the $(k+1)^{\mathrm{th}}$ iterate given by the accelerated nonlinear PDHG method~\eqref{eq:accI_pdhg_alg} with the initial points $(\bxb, \bysb) = (\bx_{k},\by_{k}^{*})$, intermediate points $(\bxt,\byst) = (\bx_{k+1}, \by_{k}^{*} + \theta_{k}(\by_{k}^{*} - \by_{k-1}^{*}))$, output points $(\bxh,\bysh) = (\bx_{k+1},\by_{k+1}^{*})$, strong convexity constants $\gamma_g > 0$ and $\gamma_{h^*} = 0$, and parameters $\tau = \tau_{k}$, $\sigma = \sigma_{k}$ and $\theta = \theta_{k}$ to get
\begin{equation}\label{eq:prop3_descent_rule_sc}
    \begin{alignedat}{1}
    \mathcal{L}(\bx_{k+1},\by^{*}) - \mathcal{L}(\bx,\by_{k+1}^{*}) &\leqslant \frac{1}{\tau_{k}}\left(D_{\phix}(\bx,\bx_{k}) - D_{\phix}(\bx_{k+1},\bx_{k}) - \left(1 + \gamma_g\tau_{k}\right)D_{\phix}(\bx,\bx_{k+1})\right) \\
    &\quad + \frac{1}{\sigma_{k}}\left(D_{\phiys}(\bys,\by_{k}^{*}) - D_{\phiys}(\by_{k+1}^{*},\by_{k}^{*}) - D_{\phiys}(\bys,\by_{k+1}^{*})\right) \\
    &\quad + \left\langle \by_{k}^{*} + \theta_{k}(\by_{k}^{*} - \by_{k-1}^{*}) - \by_{k+1}^{*}, \matr{A}(\bx  - \bx_{k+1})\right\rangle. \\
    \end{alignedat}
\end{equation}

We wish to bound the last line on the right hand side of~\eqref{eq:prop3_descent_rule_sc} to eliminate the Bregman divergence term $D_{\phix}(\bx_{k+1},\bx_{k})$. To do so, first distribute the last line on the right hand side of~\eqref{eq:prop3_descent_rule_sc} as
\begin{equation}\label{eq:app_d_term1}
\theta_{k}\left\langle \by_{k}^{*} - \by_{k-1}^{*}, \matr{A}(\bx - \bx_{k+1})\right\rangle - \left\langle \by_{k+1}^{*} - \by_{k}^{*}, \matr{A}(\bx - \bx_{k+1}) \right\rangle.
\end{equation}
Write $\bx-\bx_{k+1} = (\bx-\bx_{k}) + (\bx_{k} - \bx_{k+1})$ and substitute in~\eqref{eq:app_d_term1} to get
\begin{equation}\label{eq:app_d_term2}
\begin{alignedat}{1}
\theta_{k}\left\langle \by_{k}^{*} - \by_{k-1}^{*}, \matr{A}(\bx - \bx_{k})\right\rangle &- \theta_{k}\left\langle \by_{k}^{*} - \by_{k-1}^{*}, \matr{A}(\bx_{k+1} - \bx_{k})\right\rangle \\
&\quad - \left\langle \by_{k+1}^{*} - \by_{k}^{*}, \matr{A}(\bx - \bx_{k+1}) \right\rangle.
\end{alignedat}
\end{equation}
Next, use Fact~\ref{fact:cauchy_ineq_cont} with $(\bx,\bys) = (\bx_{k+1},\by_{k}^{*})$, $(\bx',\bys') = (\bx_{k},\by_{k-1}^{*})$ and $\alpha = 1/(\theta_{k}\tau_{k}\normop{\matr{A}})$, and assumption (A5), identity~\eqref{eq:app_identity_rr} derived in Part 1 to bound the second bilinear form in~\eqref{eq:app_d_term2} as follows:
\begin{equation} \label{eq:app_d_term4}
\left |\left\langle \by_{k}^{*} - \by_{k-1}^{*}, \matr{A}(\bx_{k+1} - \bx_{k})\right\rangle \right| \leqslant \frac{1}{\theta_{k}\tau_{k}}D_{\phix}(\bx_{k+1},\bx_{k}) + \frac{\theta_{k}}{\sigma_{k}} D_{\phiys}(\by_{k}^{*},\by_{k-1}^{*}).
\end{equation}
Finally, use~\eqref{eq:app_d_term1},~\eqref{eq:app_d_term2}, and~\eqref{eq:app_d_term4} to eliminate the Bregman divergence term $D_{\phix}(\bx_{k+1},\bx_{k})$ on the right hand side of the descent rule~\eqref{eq:prop3_descent_rule_sc} and rearrange to get
\begin{equation}\label{eq:app_d_dr1}
\begin{alignedat}{1}
    &\mathcal{L}(\bx_{k+1},\bys) - \mathcal{L}(\bx,\by_{k+1}^{*}) + \left(\frac{1+\gamma_g\tau_{k}}{\tau_{k}}\right)D_{\phix}(\bx,\bx_{k+1}) + \frac{1}{\sigma_{k}}D_{\phiys}(\bys,\by_{k+1}^{*}) \\
    &\qquad + \frac{1}{\sigma_{k}}D_{\phiys}(\by_{k+1}^{*},\by_{k}^{*}) + \left\langle\by_{k+1}^{*}-\by_{k}^{*},\matr{A}(\bx-\bx_{k+1}) \right\rangle \\
    &\leqslant \frac{1}{\tau_{k}}D_{\phix}(\bx,\bx_{k}) + \frac{1}{\sigma_{k}}D_{\phiys}(\bys,\by_{k}^{*}) \\
    &\qquad + \frac{\theta_{k}^2}{\sigma_{k}}D_{\phiys}(\by_{k}^{*},\by_{k-1}^{*}) + \theta_{k}\left\langle\by_{k}^{*}-\by_{k-1}^{*},\matr{A}(\bx-\bx_{k}) \right\rangle
\end{alignedat}
\end{equation}

We now want to express both sides of inequality~\eqref{eq:app_d_dr1} in terms of $\Delta_{k}(\bx,\by^{*})$ and $\Delta_{k+1}(\bx,\by^{*})$, starting from the left hand side. Note that the recurrence relations~\eqref{eq:acc_sq_exact_iter} imply
\[
\frac{1+\gamma_g\tau_{k}}{\tau_{k}} = \frac{1}{\theta_{k+1}\tau_{k+1}} \quad \mathrm{and} \quad \frac{1}{\sigma_{k}} = \frac{1}{\theta_{k+1}\sigma_{k+1}}.
\]
As such, the left hand side of~\eqref{eq:app_d_dr1} admits the lower bound
\begin{equation}\label{eq:app_d_dr_lhs}
    \mathcal{L}(\bx_{k+1},\bys) - \mathcal{L}(\bx,\by_{k+1}^{*}) + \Delta_{k+1}(\bx,\by^{*})/\theta_{k+1}.
\end{equation}
Since $0 \leqslant \theta_{k} \leqslant 1$, we have
\[
\frac{\theta_{k}^2}{\sigma_{k}}D_{\phiys}(\by_{k}^{*},\by_{k-1}^{*}) \leqslant \frac{1}{\sigma_{k}}D_{\phiys}(\by_{k}^{*},\by_{k-1}^{*}).
\]
Hence the right hand side of~\eqref{eq:app_d_dr1} is bounded from above by $\Delta_{k}(\bx,\by^{*})$. In summary, we find
\begin{equation*}
\mathcal{L}(\bx_{k+1},\bys) - \mathcal{L}(\bx,\by_{k+1}^{*}) \leqslant \Delta_{k}(\bx,\by^{*}) - \Delta_{k+1}(\bx,\by^{*})/\theta_{k+1}.
\end{equation*}
This proves the descent rule~\eqref{eq:descent_rule_s_acc_I}. 

\bigbreak
\noindent
\textbf{Part 3.} Use~\eqref{eq:descent_rule_s_acc_I}, the third recurrence relation in~\eqref{eq:acc_sq_exact_iter}, and the averages $T_{K}$, $\bX_{K}$ and $\bY_{K}^{*}$ to compute the weighted sum
\begin{equation}\label{eq:app_d_dr3}
\begin{alignedat}{1}
T_{K}\left(\mathcal{L}(\bX_{K},\bys) - \mathcal{L}(\bx,\bY_{K}^{*})\right) & \leqslant \sum_{k=1}^{K}\frac{\sigma_{k-1}}{\sigma_{0}}\left(\mathcal{L}(\bx_{k},\bys) - \mathcal{L}(\bx,\by_{k}^{*})\right)\\
&\leqslant \sum_{k=1}^{K}\frac{\sigma_{k-1}}{\sigma_{0}}\left(\Delta_{k-1}(\bx,\by^{*}) -  \Delta_{k}(\bx,\by^{*})/\theta_{k}\right) \\
&= \sum_{k=1}^{K}\left(\frac{\sigma_{k-1}}{\sigma_{0}}\Delta_{k-1}(\bx,\by^{*}) - \frac{1}{\theta_{k}}\frac{\sigma_{k-1}}{\sigma_{0}}\Delta_{k}(\bx,\by^{*})\right) \\
&= \sum_{k=1}^{K}\left(\frac{\sigma_{k-1}}{\sigma_{0}}\Delta_{k-1}(\bx,\by^{*}) - \frac{\sigma_{k}}{\sigma_{0}}\Delta_{k}(\bx,\by^{*})\right) \\
&= \Delta_{0}(\bx,\by^{*}) - \frac{\sigma_{K}}{\sigma_{0}}\Delta_{K}(\bx,\by^{*}). \\
\end{alignedat}
\end{equation}
This proves the estimate~\eqref{eq:acc_sq__rateI}. Finally, substitute the saddle point $(\bx_{s},\by_{s}^{*})$ for $(\bx,\by^{*})$ in inequality~\eqref{eq:app_d_dr3} and use the saddle-point property $\mathcal{L}(\bx_{k},\by_{s}^{*}) - \mathcal{L}(\bx_{s},\by_{k}^{*}) \geqslant 0$ to get
\[
\Delta_{K}(\bx_{s},\by_{s}^{*}) \leqslant \frac{\sigma_0}{\sigma_{K}}\Delta_{0}(\bx_{s},\by_{s}). 
\]
The global bound~\eqref{eq:bound_seq_acc_I} follows from this upper bound and the lower bound~\eqref{eq:bound_delta_below} derived in Part 1.

\bigbreak
\noindent
\textbf{Part 4.} Substitute the identity~\eqref{eq:app_identity_rr} in the third recurrence relation of~\eqref{eq:acc_sq_exact_iter} and take the square to get the nonlinear recurrence relation
\begin{equation}\label{eq:app_d:taverage1}
    \sigma_{k+1}^{2} = \sigma_{k}^{2} + \gamma_g\sigma_{k}/\normop{\matr{A}}^2.
\end{equation}
Use this to express the average quantity $T_{K}$ as a telescoping sum:
\begin{equation*}
    T_{K} = \sum_{k=1}^{K} \frac{\sigma_{k}}{\sigma_{0}} = \sum_{k=1}^{K}\left( \normop{\matr{A}}^2\left(\sigma_{k}^{2} - \sigma_{k-1}^2\right)/\gamma_g \sigma_0\right) = \normop{\matr{A}}^2\left(\sigma_{K}^2 - \sigma_{0}^2\right)/(\gamma_g\sigma_0).
\end{equation*}
This proves formula~\eqref{eq:acc_exact_t_avg}.

We now compute the bounds in~\eqref{eq:acc_I_avg_bounds}, starting with the upper bound. Let $a = \gamma_g/(2\normop{\matr{A}}^2)$, and use this quantity in equation~\eqref{eq:app_d:taverage1} to derive a simple upper bound on $\sigma_{K}$:
\[
\begin{alignedat}{1}
\sigma_{K}^2 &= \sigma_{K-1}^2 + 2a\sigma_{K-1} \\
&\leqslant \sigma_{K-1}^2 + 2a\sigma_{K-1} + a^2 \\
&= (\sigma_{K-1} + a)^2.
\end{alignedat}
\]
Take the square root to find
\[
\sigma_{K} \leqslant \sigma_{K-1} + a.
\]
A simple calculation gives
\begin{equation*}
\sigma_{K} \leqslant \sigma_{0} + aK.
\end{equation*}
Hence
\[
T_{K} = \frac{\sigma_{K}^2 - \sigma_{0}^2}{2a\sigma_{0}} \leqslant \frac{(\sigma_{0} + aK)^2 - \sigma_{0}^2}{2a\sigma_{0}} = K + \frac{a}{2\sigma_{0}}K^2.
\]
which proves the upper bound in inequality~\eqref{eq:acc_I_avg_bounds}. 

The lower bound in inequality~\eqref{eq:acc_I_avg_bounds} is the same as derived in~\citet[Section 5.2, Equation (41)]{chambolle2016ergodic}, but here we give a different proof. Use~\eqref{eq:app_d:taverage1} to derive a lower bound on $\sigma_{K}$:
\[
\begin{alignedat}{1}
\sigma_{K}^2 &= \sigma_{K-1}^2 + 2a\sigma_{K-1} \\
&= \sigma_{K-1}^2 + 2a\sigma_{K-1}\left(\frac{\sigma_0}{a + \sigma_0} + \frac{a}{a+\sigma_0}\right) \\
&= \sigma_{K-1}^2 + \frac{2a\sigma_{0}\sigma_{K-1}}{a+\sigma_{0}} + \frac{2a^2\sigma_{K-1}}{a + \sigma_0} \\
& \geqslant \sigma_{K-1}^2 + \frac{2a\sigma_{0}\sigma_{K-1}}{a+\sigma_{0}} + \frac{2a^2\sigma_{K-1}}{(a + \sigma_0)}\frac{\sigma_0}{(a + \sigma_0)} \\
& \geqslant \sigma_{K-1}^2 + \frac{2a\sigma_{0}\sigma_{K-1}}{a+\sigma_{0}} + \frac{2a^2\sigma_{0}^2}{(a + \sigma_0)^2} \\
&= \left(\sigma_{K-1} + \frac{a\sigma_0}{a+\sigma_0}\right)^2,
\end{alignedat}
\]
where on the fifth line we used that $\sigma_{k} \geqslant \sigma_0$ for every nonnegative integer $k$, as per the third recurrence relation~\eqref{eq:acc_sq_exact_iter}. We have found
\[
\sigma_{K} \geqslant \sigma_{K-1} + \frac{a\sigma_0}{a + \sigma_0}
\]
which implies, after a simple calculation,
\[
\sigma_{K} \geqslant  \sigma_{0} + \frac{a\sigma_{0}}{a + \sigma_{0}}K.
\]
Hence
\[
\begin{alignedat}{1}
T_{K} = \frac{\sigma_{K}^{2} - \sigma_{0}^{2}}{2a\sigma_{0}} &\geqslant \frac{\sigma_{0}^{2}}{2a\sigma_{0} }\left[\left(1 + \frac{a}{a + \sigma_{0}}K\right)^2 - 1\right]\\
&= \frac{\sigma_{0}}{2a}\left[\frac{2a}{a + \sigma_{0}}K + \frac{a^2}{(a+\sigma_{0})^2}K^2\right] \\
&= \frac{\sigma_{0}}{a + \sigma_{0}}K + \frac{a\sigma_{0}}{2(a + \sigma_{0})^2}K^2,
\end{alignedat}
\]
which proves the lower bound in inequality~\eqref{eq:acc_I_avg_bounds}.

\bigbreak
\noindent
\textbf{Part 5.} First, combine the auxiliary result~\eqref{eq:bound_delta_below} and global bound~\eqref{eq:bound_seq_acc_I} to get the inequality.
\[
\frac{\gamma_{g}}{1 + \gamma_{g}\tau_{0}}D_{\phix}(\bx_{s},\bx_{K}) + \frac{1}{\sigma_{k}}D_{\phiys}(\by_{s}^{*},\by_{K}^{*}) \leqslant \frac{1}{\tau_{0}} D_{\phix}(\bx_{s},\bx_{0}) + \frac{1}{\sigma_{0}}D_{\phiys}(\by_{s}^{*},\by_{0}^{*}). 
\]
As a consequence, we have that
\begin{equation}\label{eq:app_bound_xk}
0 \leqslant \frac{\gamma_{g}}{1 + \gamma_{g}\tau_{0}}D_{\phix}(\bx_{s},\bx_{K}) \leqslant \frac{\sigma_{0}}{\sigma_{K}}\left(\frac{1}{\tau_{0}} D_{\phix}(\bx_{s},\bx_{0}) + \frac{1}{\sigma_{0}}D_{\phiys}(\by_{s}^{*},\by_{0}^{*})\right)
\end{equation}
and
\begin{equation}\label{eq:app_bound_yk}
0 \leqslant \frac{1}{\sigma_{0}}D_{\phiys}(\by_{s}^{*},\by_{K}^{*}) \leqslant \frac{1}{\tau_{0}} D_{\phix}(\bx_{s},\bx_{0}) + \frac{1}{\sigma_{0}}D_{\phiys}(\by_{s}^{*},\by_{0}^{*}). 
\end{equation}
These inequality immediately imply that the sequence of iterates $\{\bx_{K},\by_{K}^{*}\}_{K=1}^{+\infty}$ is bounded. It follows from the definitions of the averages $\bX_{K}$ and $\bY_{K}^{*}$ that the sequence of averages $\{(\bX_{K},\bY_{K}^{*})\}_{K=1}^{+\infty}$ is also bounded.

Now, thanks to Fact~\ref{fact:weak_conv_bdd_seq} there is a subsequence $\{(\bX_{K_l},\bY_{K_l}^{*})\}_{l=1}^{+\infty}$ that converges weakly to some point $(\bX,\bYs) \in \xcal \times \ycals$. We claim that $(\bX,\bYs)$ is a saddle point of the Lagrangian~\eqref{eq:lagrangian}. To see this, use inequality~\eqref{eq:acc_sq__rateI} with $(\bx,\by^*) \in \dom g \times \dom h^*$ and take the infimum limit $l \to +\infty$ to get
\[
\liminf_{l \to +\infty} \mathcal{L}(\bX_{K_{l+1}},\bys) - \mathcal{L}(\bx,\bY_{K_{l+1}}^{*}) \leqslant \liminf_{l \to +\infty} \frac{1}{T_{K_{l+1}}}\left(\frac{1}{\tau_{0}}D_{\phix}(\bx,\bx_{0}) + \frac{1}{\sigma_{0}}D_{\phiys}(\by^{*},\by_{0}^{*})\right) = 0.
\]
The lower semicontinuity property of the functions $g$ and $h^*$ implies
\[
0 \geqslant \liminf_{l \to +\infty} \left(\mathcal{L}(\bX_{K_l},\bys) - \mathcal{L}(\bx,\bY_{K_l}^{*})\right) \geqslant \mathcal{L}(\bX,\bys) - \mathcal{L}(\bx,\bYs),
\]
from which we find
\[
\mathcal{L}(\bX,\bys) \leqslant \mathcal{L}(\bx,\bYs).
\]
As the pair of points $(\bx,\bys) \in \xcal\times\ycals$ was arbitrary, we conclude that $(\bX,\bYs)$ is a saddle point of the Lagrangian~\eqref{eq:lagrangian}. Moreover, we deduce from Remark~\ref{rem:uniqueness} that $\bX$ coincides with the unique solution $\bx_{s}$ of the primal problem~\eqref{eq:primal_prob}, i.e., $\bX = \bx_{s}$.

Next, we show that the individual sequences $\{\bx_{k}\}_{k=1}^{+\infty}$ and $\{\bX_{K}\}_{K=1}^{+\infty}$ converge strongly to the unique solution $\bx_{s}$ of the primal problem~\eqref{eq:primal_prob}. The strong convergence of $\bx_{k}$ is evident from~\eqref{eq:app_bound_xk}, the limit $\lim_{K \to +\infty} \sigma_{K} = +\infty$ from~\eqref{eq:acc_exact_t_avg} and~\eqref{eq:acc_I_avg_bounds}, and assumption (A5):
\[
\begin{alignedat}{1}
0 \leqslant \lim_{K \to +\infty} \frac{1}{2}\normsq{\bx_{s}-\bx_{K}} &\leqslant \lim_{K \to +\infty} D_{\phix}(\bx_{s},\bx_{K}) \\
&\leqslant \lim_{K \to +\infty} \left(\frac{1 + \gamma_{g}\tau_{0}}{\gamma_{g}}\right)\frac{\sigma_{0}}{\sigma_{K}}\left(\frac{1}{\tau_{0}} D_{\phix}(\bx_{s},\bx_{0}) + \frac{1}{\sigma_{0}}D_{\phiys}(\by_{s}^{*},\by_{0}^{*})\right) \\
&= 0.
\end{alignedat}
\]
We further deduce from Fact~\ref{fact:weighted_averages} that the sequence $\{\bX_{K}\}_{K=1}^{+\infty}$ converges strongly to the same limit $\bx_{s}$.

Suppose now that $\ycals$ is finite-dimensional. Since the sequence $\{(\bx_{k},\by_{k}^{*})\}_{k=1}^{+\infty}$ is bounded and $\bx_k$ converges strongly to $\bx_{s}$, there is some subsequence $\{(\bx_{k_l},\by_{k_l}^{*})\}_{k=1}^{+\infty}$ that converges strongly to a point $(\bx_{s},\by_{s}^{*}) \in \xcal \times \ycals$. By Fact~\ref{fact:weighted_averages}, the subsequence of averages $\{(\bX_{k_l},\bY_{k_l}^{*})\}_{k=1}^{+\infty}$ also strongly converges to $(\bx_{s},\by_{s}^{*})$. A similar argument as the one described two paragraphs before shows then that $(\bx_{s},\by_{s}^{*})$ is a saddle point of the Lagrangian~\eqref{eq:lagrangian}, and moreover from Fact~\ref{fact:saddle_prob} we deduce that $\by_{s}^{*}$ is a solution to the dual problem~\eqref{eq:dual_prob}. This concludes the proof.


%% file: app_proof_propacc2.tex
We divide the proof into four parts, first deriving an auxiliary result, and then proving in turn the descent rule~\eqref{eq:descent_rule_s_acc_II} (Proposition~\ref{prop:acc_sq_algIII}(a)), the estimate~\eqref{eq:acc_sq__rateIII} and global bound~\eqref{eq:bound_seq_acc_II} (Proposition~\ref{prop:acc_sq_algIII}(b)), and the convergence properties of the accelerated nonlinear PDHG method~\eqref{eq:accIII_pdhg_alg} (Proposition~\ref{prop:acc_sq_algIII}(c)).

\bigbreak
\noindent
\textbf{Part 1.} First, we show that for every $(\bx,\by^{*}) \in \dom g \times \dom h^*$ and $k \in \mathbb{N}$, the quantity $\Delta_{k}(\bx,\by^{*})$ satisfies the lower bound
\begin{equation}\label{eq:app_delta-acc-III}
    \Delta_{k}(\bx,\by^{*}) \geqslant \frac{1}{\sigma}D_{\phiys}(\by^{*},\by_{k}^{*}).
\end{equation}
To do so, use Fact~\ref{fact:cauchy_ineq_cont} with $\alpha = 1/(\tau\theta \normop{\matr{A}})$ and use assumption (A5) to find
\[
\left | \left\langle\by_{k}^{*}-\by_{k-1}^{*},\matr{A}(\bx-\bx_{k}) \right\rangle \right |  \leqslant  \frac{1}{\tau\theta}D_{\phix}(\bx,\bx_{k}) + \tau\theta\normop{\matr{A}}^2 D_{\phiys}(\by_{k}^{*},\by_{k-1}^{*}).
\]
From the choice of parameters in~\eqref{eq:accIII_params}, we have the identity
\begin{equation}\label{eq:app_id-acc-III}
    \tau\sigma\theta\normop{\matr{A}}^2 = 1.
\end{equation}
Use this identity in the previous inequality to find
\begin{equation}\label{eq:app_part1_id_acc_III}
    \left | \left\langle\by_{k}^{*}-\by_{k-1}^{*},\matr{A}(\bx-\bx_{k}) \right\rangle \right |  \leqslant  \frac{1}{\tau\theta}D_{\phix}(\bx,\bx_{k}) + \frac{1}{\sigma} D_{\phiys}(\by_{k}^{*},\by_{k-1}^{*}).
\end{equation}
Substitute in $\Delta_{k}(\bx,\by^{*})$ to get
\begin{equation*}
\begin{alignedat}{1}
\Delta_{k}(\bx,\by^{*}) &\geqslant \frac{1}{\tau}D_{\phix}(\bx,\bx_{k}) + \frac{1}{\sigma}D_{\phiys}(\by^{*},\by_{k}^{*}) + \frac{\theta}{\sigma}D_{\phiys}(\by_{k}^{*},\by_{k-1}^{*}) \\
&\qquad  - \theta\left( \frac{1}{\tau\theta}D_{\phix}(\bx,\bx_{k}) + \frac{1}{\sigma}D_{\phiys}(\by_{k}^{*},\by_{k-1}^{*})\right)\\
&=  \frac{1}{\sigma}D_{\phiys}(\by^{*},\by_{k}^{*}).
\end{alignedat}
\end{equation*}
This proves the auxiliary result~\eqref{eq:app_delta-acc-III}.

\bigbreak
\noindent
\textbf{Part 2.} Let $(\bx,\by^{*}) \in \dom g \times \dom h^*$. By assumptions (A1)-(A7), Lemma~\ref{lem:descent_rule} holds, and we can apply the improved descent rule~\eqref{eq:descent_rule_sc} to the $(k+1)^{\mathrm{th}}$ iterate given by the accelerated nonlinear PDHG method~\eqref{eq:accIII_pdhg_alg} with the initial points $(\bxb, \bysb) = (\bx_{k},\by_{k}^{*})$, intermediate points $(\bxt,\byst) = (\bx_{k+1}, \by_{k}^{*} + \theta(\by_{k}^{*} - \by_{k-1}^{*}))$, output points $(\bxh,\bysh) = (\bx_{k+1},\by_{k+1}^{*})$, the strong convexity constants $\gamma_g > 0$, $\gamma_{h^*} > 0$, and the parameters $\tau$, $\sigma$, and $\theta$ defined in~\eqref{eq:accIII_params}:
\begin{equation}\label{eq:prop4_descent_rule_sc}
    \begin{alignedat}{1}
    \mathcal{L}(\bx_{k+1},\by^{*}) - \mathcal{L}(\bx,\by_{k+1}^{*}) &\leqslant \frac{1}{\tau}\left(D_{\phix}(\bx,\bx_{k}) - D_{\phix}(\bx_{k+1},\bx_{k}) - (1 + \gamma_g\tau)D_{\phix}(\bx,\bx_{k+1})\right) \\
    &\quad + \frac{1}{\sigma}\left(D_{\phiys}(\by^{*},\by_{k}^{*}) - D_{\phiys}(\by_{k+1}^{*},\by_{k}^{*}) - (1 + \gamma_{h^*}\sigma)D_{\phiys}(\by^{*},\by_{k+1}^{*})\right) \\
    &\quad + \left\langle \by_{k}^{*} + \theta(\by_{k}^{*} - \by_{k-1}^{*}) - \by_{k+1}^{*}, \matr{A}(\bx  - \bx_{k+1})\right\rangle.
    \end{alignedat}
\end{equation}

We wish to bound the last line on the right hand side of~\eqref{eq:prop4_descent_rule_sc} to eliminate the Bregman divergence term $D_{\phix}(\bx_{k+1},\bx_{k})$. To do so, first distribute the last line on the right hand side of~\eqref{eq:prop4_descent_rule_sc} as
\begin{equation}\label{eq:app_e_term1}
\theta\left\langle \by_{k}^{*} - \by_{k-1}^{*}, \matr{A}(\bx - \bx_{k+1})\right\rangle - \left\langle \by_{k+1}^{*} - \by_{k}^{*}, \matr{A}(\bx - \bx_{k+1}) \right\rangle.
\end{equation}
Write $\bx-\bx_{k+1} = (\bx-\bx_{k}) + (\bx_{k} - \bx_{k+1})$ and substitute in~\eqref{eq:app_e_term1} to get
\begin{equation}\label{eq:app_e_term2}
\begin{alignedat}{1}
\theta\left\langle \by_{k}^{*} - \by_{k-1}^{*}, \matr{A}(\bx - \bx_{k})\right\rangle &- \theta\left\langle \by_{k}^{*} - \by_{k-1}^{*}, \matr{A}(\bx_{k+1} - \bx_{k})\right\rangle \\
&\quad - \left\langle \by_{k+1}^{*} - \by_{k}^{*}, \matr{A}(\bx - \bx_{k+1}) \right\rangle. \\
\end{alignedat}
\end{equation}
Next, use inequality~\eqref{eq:app_part1_id_acc_III} with $\bx = \bx_{k+1}$ derived in Part 1 to bound the second bilinear form in~\eqref{eq:app_e_term2} as follows:
\begin{equation} \label{eq:app_e_term4}
\left|\left\langle \by_{k}^{*} - \by_{k-1}^{*}, \matr{A}(\bx_{k+1} - \bx_{k})\right\rangle \right| \leqslant \frac{1}{\tau\theta}D_{\phix}(\bx_{k+1},\bx_{k}) + \frac{1}{\sigma} D_{\phiys}(\by_{k}^{*},\by_{k-1}^{*})
\end{equation}
Finally, use~\eqref{eq:app_e_term1},~\eqref{eq:app_e_term2}, and~\eqref{eq:app_e_term4} to eliminate the Bregman divergence term $D_{\phix}(\bx_{k+1},\bx_{k})$ on the right hand side of the descent rule~\eqref{eq:prop4_descent_rule_sc}:
\begin{equation}\label{eq:app_e_dr1}
\begin{alignedat}{1}
    &\mathcal{L}(\bx_{k+1},\by^{*}) - \mathcal{L}(\bx,\by_{k+1}^{*}) + \left(\frac{1+\gamma_g\tau}{\tau}\right)D_{\phix}(\bx,\bx_{k+1}) + \left(\frac{1+\gamma_{h^*}\sigma}{\sigma}\right)D_{\phiys}(\by^{*},\by_{k+1}^{*}) \\
    &\qquad + \frac{1}{\sigma}D_{\phiys}(\by_{k+1}^{*},\by_{k}^{*}) + \left\langle\by_{k+1}^{*}-\by_{k}^{*},\matr{A}(\bx-\bx_{k+1}) \right\rangle \\
    &\leqslant \frac{1}{\tau}D_{\phix}(\bx,\bx_{k}) + \frac{1}{\sigma}D_{\phiys}(\by^{*},\by_{k}^{*}) \\
    &\qquad + \frac{\theta}{\sigma}D_{\phiys}(\by_{k}^{*},\by_{k-1}^{*}) + \theta\left\langle\by_{k}^{*}-\by_{k-1}^{*},\matr{A}(\bx-\bx_{k}) \right\rangle
\end{alignedat}
\end{equation}

We now want to express both sides of inequality~\eqref{eq:app_e_dr1} in terms of $\Delta_{k}(\bx,\by^{*})$ and $\Delta_{k+1}(\bx,\by^{*})$, starting from the left hand side. Note that the choice of parameters in~\eqref{eq:accIII_params} implies
\[
1 + \gamma_g\tau = 1/\theta \quad \mathrm{and} \quad 1+\gamma_{h^{*}}\sigma = 1/\theta.
\]
As such, the left hand side of~\eqref{eq:app_e_dr1} is equal to
\[
\mathcal{L}(\bx_{k+1},\by^{*}) - \mathcal{L}(\bx,\by_{k+1}^{*}) + \Delta_{k+1}(\bx,\by^{*})/\theta,
\]
and the right hand side of~\eqref{eq:app_e_dr1} is equal to $\Delta_{k}(\bx,\by^{*})$. Put together, we find
\begin{equation}\label{eq:app_e_dr2}
\mathcal{L}(\bx_{k+1},\by^{*}) - \mathcal{L}(\bx,\by_{k+1}^{*}) \leqslant \Delta_{k}(\bx,\by^{*}) -  \Delta_{k+1}(\bx,\by^{*})/\theta.
\end{equation}
This proves the descent rule~\eqref{eq:descent_rule_s_acc_II}.

\bigbreak
\noindent
\textbf{Part 3.} Use~\eqref{eq:app_e_dr2} and the averages $T_{K}$, $\bX_{K}$ and $\bY_{K}^{*}$ to compute the sum
\begin{equation}\label{eq:app_e_dr3}
\begin{alignedat}{1}
T_{K}\mathcal{L}(\bX_{K},\by^{*}) - \mathcal{L}(\bx,\bY_{K}^{*}) &= \sum_{k=1}^{K}\frac{1}{\theta^{k-1}}\left(\mathcal{L}(\bx_{k},\by^{*}) - \mathcal{L}(\bx,\by_{k}^{*})\right)\\
&\leqslant \sum_{k=1}^{K}\frac{1}{\theta^{k-1}}\left(\Delta_{k-1}(\bx,\by^{*}) -  \Delta_{k}(\bx,\by^{*})/\theta\right) \\
&= \sum_{k=1}^{K}\left(\Delta_{k-1}(\bx,\by^{*})/\theta^{k-1} - \Delta_{k}(\bx,\by^{*})/\theta^{k}\right) \\
&= \Delta_{0}(\bx,\by^{*}) - \Delta_{K}(\bx,\by^{*})/\theta^{K}. \\
\end{alignedat}
\end{equation}
This proves the estimate~\eqref{eq:acc_sq__rateIII}. Finally, substitute the saddle point $(\bx_{s},\by_{s}^{*})$ for $(\bx,\by^{*})$ in inequality~\eqref{eq:app_e_dr3} and use the saddle-point property $\mathcal{L}(\bx_{k},\by_{s}^{*}) - \mathcal{L}(\bx_{s},\by_{k}^{*}) \geqslant 0$ to get
\[
\Delta_{K}(\bx_{s},\by^{*}_{s}) \leqslant \theta^{K}\Delta_{0}(\bx_{s},\by^{*}_{s}).
\]
The global bound~\eqref{eq:bound_seq_acc_II} follows from this upper bound and the lower bound~\eqref{eq:app_delta-acc-III} derived in Part 1.

\bigbreak
\noindent
\textbf{Part 4.} The global bound~\eqref{eq:bound_seq_acc_II}, assumption (A5), and Fact~\ref{fact:breg_div_props}(viii) immediately imply that the sequence of iterates $\{\by_{k}^{*}\}_{k=1}^{+\infty}$ converges strongly to $\by_{s}$. It follows from this and Fact~\ref{fact:weighted_averages} that the sequence of averages $\{\bY_{K}^{*}\}_{K=1}^{+\infty}$ also converges strongly to $\by_{s}^{*}$.

Now, consider inequality~\eqref{eq:app_e_dr3} with $(\bx,\by^{*}) = (\bx_{s},\by_{s}^{*})$ written in full:
\begin{equation}\label{eq:app_bound_xk_acc_III}
\begin{alignedat}{1}
\frac{1}{\tau}D_{\phix}(\bx_{s},\bx_{k}) + \frac{1}{\sigma}D_{\phiys}(\by_{s}^{*},\by_{k}^{*}) &+ \frac{\theta}{\sigma}D_{\phiys}(\by_{k}^{*},\by_{k-1}^{*}) + \theta\left\langle\by_{k}^{*}-\by_{k-1}^{*},\matr{A}(\bx_{s}-\bx_{k}) \right\rangle \\
&\qquad \leqslant \theta^{K}\left(\frac{1}{\tau}D_{\phix}(\bx_{s},\bx_{0}) + \frac{1}{\sigma}D_{\phiys}(\by_{s}^{*},\by_{0}^{*})\right).
\end{alignedat}
\end{equation}
We wish to bound the bilinear form on the left hand side to obtain a bound on $\bx_{k}$. To do so, use Fact~\ref{fact:cauchy_ineq_cont} with $\alpha = 1/(2\tau\theta\normop{\matr{A}})$ and identity~\eqref{eq:app_id-acc-III} to obtain the bound
\[
\left | \left\langle\by_{k}^{*}-\by_{k-1}^{*},\matr{A}(\bx_{s}-\bx_{k}) \right\rangle \right | \leqslant \frac{1}{4\tau\theta}\normx{\bx_{s}-\bx_{k}}^2 + \frac{1}{\sigma}\normys{\by_{k}^{*}-\by_{k-1}^{*}}^2.
\]
Substitute in~\eqref{eq:app_bound_xk_acc_III} to get
\[
\begin{alignedat}{1}
\frac{1}{\tau}D_{\phix}(\bx_{s},\bx_{k}) + \frac{1}{\sigma}D_{\phiys}(\by_{s}^{*},\by_{k}^{*}) &+
\frac{\theta}{\sigma}D_{\phiys}(\by_{k}^{*},\by_{k-1}^{*}) - \frac{\theta}{\sigma}\normys{\by_{k}^{*}-\by_{k-1}^{*}}^2 - \frac{1}{4\tau}\normx{\bx_{s}-\bx_{k}}^2 \\
&\qquad \leqslant \theta^{K}\left( \frac{1}{\tau}D_{\phix}(\bx_{s},\bx_{0}) + \frac{1}{\sigma}D_{\phiys}(\by_{s}^{*},\by_{0}^{*})\right).
\end{alignedat}
\]
Finally, use the inequalities $D_{\phiys}(\by_{s}^{*},\by_{k}^{*}) \geqslant 0$, $D_{\phiys}(\by_{k}^{*},\by_{k-1}^{*}) \geqslant 0$, and $D_{\phix}(\bx_{s},\bx_{k}) \geqslant \frac{1}{2}\normx{\bx_{s}-\bx_{k}}^2$ (thanks to assumption (A5) and Fact~\ref{fact:breg_div_props}(viii) with $m = 1$) to obtain
\[
\frac{1}{4\tau}\normx{\bx_{s}-\bx_{k}}^2 \leqslant \frac{\theta}{\sigma}\normys{\by_{k}^{*}-\by_{k-1}^{*}}^2 + \theta^{K}\left( \frac{1}{\tau}D_{\phix}(\bx_{s},\bx_{0}) + \frac{1}{\sigma}D_{\phiys}(\by_{s}^{*},\by_{0}^{*})\right).
\]
Taking the limit $k \to +\infty$ yields $\lim_{k \to +\infty} \bx_{k} = \bx_{s}$. It follows from this and Fact~\ref{fact:weighted_averages} that the sequence of averages $\{\bX_{K}\}_{K=1}^{+\infty}$ also converges strongly to $\bx_{s}$. This concludes the proof.